\documentclass[a4paper]{article}
\usepackage[authoryear]{natbib}
\usepackage{a4wide}

\usepackage[T1]{fontenc}
\usepackage{graphicx}
\usepackage{mathtools}
\usepackage{amssymb}
\usepackage{amsthm}
\usepackage{amsmath}
\usepackage{thmtools}
\usepackage{xcolor}
\usepackage{nameref}
\usepackage{babel}
\usepackage{hyperref}

\usepackage{booktabs}
\usepackage{multirow}
\usepackage{algorithm, algorithmicx}
\usepackage{algpseudocode}

\algnewcommand\algorithmicinput{\textbf{Input:}}
\algnewcommand\Input{\item[\algorithmicinput]}
\algnewcommand\algorithmicoutput{\textbf{Output:}}
\algnewcommand\Output{\item[\algorithmicoutput]}

\theoremstyle{plain}

\newtheorem{theorem}{Theorem}[section]
\newtheorem{cor}[theorem]{Corollary}
\newtheorem{lemma}[theorem]{Lemma}
\newtheorem{prop}[theorem]{Proposition}
\theoremstyle{definition}

\newtheorem{remark}[theorem]{Remark}

\DeclareMathOperator{\var}{Var}
\DeclareMathOperator{\cov}{Cov}
\DeclareMathOperator{\E}{\mathbb{E}}

\newcommand{\setbinom}{\genfrac{\{}{\}}{0pt}{}}
\renewcommand{\binom}{\genfrac{(}{)}{0pt}{}}

\begin{document}

\title{Asymptotic confidence bands for centered purely random forests}

\author{Natalie Neumeyer$^1$, Jan Rabe$^1$ and Mathias Trabs$^2$\\
   \emph{Universität Hamburg$^1$ and Karlsruhe Institute of Technology$^2$}}

\maketitle

\begin{abstract}
In a multivariate nonparametric regression setting we construct explicit asymptotic uniform confidence bands for centered purely random forests. Since the most popular example in this class of random forests, namely the uniformly centered purely random forests, is well known to suffer from suboptimal rates, we propose a new type of purely random forests, called the Ehrenfest centered purely random forests, which achieve minimax optimal rates. Our main confidence band theorem applies to both random forests. The proof is based on an interpretation of random forests as generalized U-Statistics together with a Gaussian approximation of the supremum of empirical processes. Our theoretical findings are illustrated in simulation examples.
\end{abstract}



\section{Introduction}

Random forests are widely used in machine learning, yet many fundamental questions about their statistical properties remain open. We present asymptotic confidence bands for regression functions based on centered purely random forests. As by-products we obtain a rate of convergence of the mean squared error, a uniform rate of convergence, and pointwise asymptotic normality. 

Since the introduction of random forests by \citet{Breiman2001}, their statistical analysis is an active research area. A large part of the literature investigates consistency of various types of random forests. Examples are the results
by \citet{Biau2008} or \citet{Scornet2015}, who prove consistency
of Breiman's forest in an additive regression model. \citet{Chi2022}
show consistency in a high-dimensional regression model. \citet{Scornet2016}
analyzes the relationship between random forests and their infinite
version, i.e., a random forest with infinitely many trees, and links
the consistency of the finite forest to the infinite. Recently, \citet{Zhan-Liu-Xia} showed consistency of random forests based on oblique
decision trees. 

A substantial amount of literature on random forests deals with purely
random forests, in which the partitions of the trees do not depend
on the training sample. A popular version are centered purely random
forests that are characterized by always splitting each hyperrectangle
in its center orthogonal to a random coordinate.  They have been introduced by \citet{Breiman2004}, who analyzed their mean squared
error. The considered regression model 
has strong and weak features with higher split probabilities for strong features. \citet{Biau2012} showed consistency in this model and proved bounds on the bias and variance
terms. Both of these articles consider the infinite random forest
estimator.
\citet{Klusowski2021} improved the mean squared prediction error
over the bound by \citet{Biau2012}. Under the assumption that the
probabilities of selecting the feature to be cut are constant for
all steps in the tree construction, he further proved that centered
random forests cannot achieve the minimax optimal rate of $n^{-2/(p+2)}$,
see e.g. \citet{Tsybakov2009}, for the mean squared error in a regression
model with Lipschitz continuous regression function $m\colon\mathbb R^p\to\mathbb R$ and $n\in\mathbb N$ i.i.d. observations. Moreover, \citet{Klusowski2021}
showed that a uniform distribution over the set of relevant features
achieves the best possible rate for these random forests.
As an alternative to centered purely random forests, rate optimal purely random forests have been constructed based on random tessellations, in particular the Mondrian random forest introduced by \cite{Mourtada2020}.
The main reason to consider purely random forests is that the reduced data dependence of the algorithm reduces the technical difficulties when proving statistical guarantees.  Note that the aforementioned articles on purely random forests do not incorporate subsampling for the individual regression trees, but instead use the entire training sample for each tree, which is a considerable deviation from practical applications of random forests and a noteworthy distinction from our work.

In order to achieve minimax optimality with centered purely random forests, we propose a new centered purely random forest algorithm, called \emph{Ehrenfest centered purely random forest}. The method relies on a regularizing splitting distribution which is motivated by the Ehrenfest model on particle diffusion and which avoids too narrow cells with a too large diameter. As a result we obtain optimal pointwise and uniform convergence rates for Lipschitz regular regression functions.

The literature on statistical inference for random forests is even more limited and restricted to purely random forests or, slightly more general, honest random forests, where no observation is used for both, the construction of the partition and the estimation of the weights in each cell.
Pointwise asymptotic confidence intervals can be derived from asymptotic normality of the random forest. The latter was shown by  \citet{Mentch2016}, \citet{Wager2018} and \citet{Peng2022}. In those results the variance of the asymptotic
normal distribution is either unknown or described abstractly. Estimators for the variance of random forests are suggested by  \citet{Mentch2016}, \citet{Wager2018} and \citet{Xu2024}. A pointwise central limit theorem for Mondrian random forests has been proved by \cite{cattaneo2023inference}.

There are no results that allow for the construction of uniform asymptotic
confidence bands that we are aware of. Moreover, the lack of an explicit asymptotic variance is puzzling. We want to close these gaps and
prove both of these results for centered purely random forests. More precisely, our main result is an
asymptotic uniform confidence band $\mathcal{C}_{n}(x)$, $x\in\mathcal X$, centered arround a centered purely random forests satisfying
	\[
	\liminf_{n\to\infty}\mathbb{P}\big(m(x)\in\mathcal{C}_{n}(x),\,\forall x\in\mathcal{X}\big)\geq1-\beta \in (0,1),
	\]
    where $m$ denotes the regression function and $\mathcal X\subset\mathbb R^p$ is the covariate space.
    One of our aforementioned by-products is pointwise asymptotic normality with  a very explicit and simple representation of the asymptotic variance which is easy to estimate.

Despite the common tendency to categorize a random forest as a machine
learning algorithm, it can be considered as a nonparametric regression estimator. This
is especially true for purely random forests, that do not exploit
the training data in their partition construction. A substantial amount
of literature exists on the subject of confidence bands for different
nonparametric density and regression estimators. Most of the early results
in the literature were for the univariate case. One of the first results
is by \citet{Smirnov1950} and gives confidence bands for a probability
density based on a histogram estimator.
\citet{Bickel1973} study confidence bands for kernel density estimators.
\citet{Gine2010} investigated confidence bands for general
density estimators that adapt to the degree of smoothness of the density
function. 
\citet{Chernozhukov2014a} and \citet{Armstrong}  also derived adaptive confidence bands. 
\citet{Johnston1982} considered confidence bands for regression functions based on the Nadaraya-Watson
 estimator, and \citet{Haerdle1989} based on M-smoothers.
Typically, the confidence bands rely on an undersmoothing of the estimator
to reduce the bias relative to the stochastic error. An alternative
is a direct bias correction, which is used, for instance, by \citet{Eubank1993},
who considered a deterministic, uniform design for local constant
regression estimation, and in the article by \citet{Xia1998}, who
considered a random design under dependence and used local linear
estimation. Bootstrap confidence bands for nonparametric
regression were considered by \citet{Hall1993}, \citet{Neumann1998},
\citet{Claeskens2003} and \citet{Hall2013}. \citet{Sabbah2014} showed
confidence bands for quantile regression functions, and  \citet{Birke2010} and \citet{Proksch-Bissantz-Dette} proved confidence bands
in  inverse regression models.

The aforementioned articles on confidence bands for regression functions only considered univariate covariates.
In the multivariate case \citet{Konakov1984} analyzed the asymptotic
distribution of the maximal deviation for the Nadaraya-Watson estimate
in a random design setting. For multivariate regression models \citet{Proksch2016} proved confidence
bands in the case of fixed design, in particular based on local polynomial estimators, and \citet{Chao2017}  in the case of random design in different models with the special case of a mean regression function and Nadaraya-Watson estimator. Confidence bands for a broad class of partitioning estimators have been proposed by \cite{CattaneoEtAl2020}. The special case of histogram estimators has been revisited by \citet{Neumeyer-Rabe-Trabs}. 

Our method of proof relies on the interpretation
of a random forest as a generalized U-statistic, so that the Hoeffding decomposition
and the H\'ajek projection can be utilized to analyze the asymptotic
behavior, similarly as by 
\citet{Mentch2016} and \citet{Peng2022}. The dominating term in the asymptotic expansion of the random forest then has empirical process structure. Instead of approximating the empirical process uniformly by a Gaussian process as in classical results for confidence bands for densities and regression functions, it has advantages to  approximate the supremum of the empirical process by the supremum of a Gaussian process. To this end we apply the approximation result by  \citet{Chernozhukov2014}, similarly as  \citet{Chernozhukov2014a} and \citet{Patschkowski2019} for density confidence bands.
\medskip 

We introduce the nonparametric regression model and define centered purely random forests in Section \ref{sec:Random-forests}. In Section~\ref{sec:rates} we present the results for the mean squared error, the uniform convergence rate and the pointwise asymptotic normality, and introduce the Ehrenfest centered purely random forest. In Section \ref{sec:Confidence-Bands} the uniform confidence band result is presented, assumptions are discussed and the proof strategy is explained. Simulations are shown in Section \ref{chap:Simulation-study}. The proofs can be found in the supplementary material. In particular, all proofs for Section~\ref{sec:rates} are given in Section~\ref{sec:proofRates} while the proofs for Section~\ref{sec:Confidence-Bands} are detailed out in Sections~\ref{sec:ProofSketchMain} and~\ref{sec:Proofs chap CBs}. 
\medskip

Let us introduce some \emph{notation}: We write $[k]=\{1,\dots,k\}$ for any $k\in\mathbb N$. The Euclidean norm is denoted by $\Vert\cdot\Vert$. We use the  notation $\|\cdot\|_\infty$ for the supremum norm on $[0,1]^p$. The indicator function on a set $A$ is written as $\mathbb{I}\{A\}$. Number of elements of a set $A$ is denoted by $\#A$. The volume or Lebesgue measure of a set $A\subset\mathbb R^p$ is denoted by $\mathrm{vol}(A)$.  The set of all subsets $I\subseteq[n]$ of size $r$ is denoted by $\setbinom n r$. We abbreviate $A_n=\mathcal{O}(B_n)$ by $A_n\lesssim B_n$ and use the notation $A_n\sim B_n$ if $\lim_{n\to\infty} A_n/B_n\in (0,\infty)$.

\section{Regression model and centered purely random forests\label{sec:Random-forests}}

Throughout we will consider the following multivariate regression problem
\begin{equation*}
  Y=m(X)+\varepsilon,
\end{equation*}
where the covariates $X\in [0,1]^p$ admit a density $f_X$ such that 
\begin{equation}\label{eq:density assum}
0< c_{X}\leq f_{X}\leq C_{X},
\end{equation}
the regression function $m:[0,1]^p \to \mathbb{R}$ is $\alpha$-Hölder continuous for $\alpha \in (0,1]$,
the observation error $\varepsilon$ is independent of $X$, centered and satisfies $\mathbb E[|\varepsilon|^q]<\infty$ for some sufficiently large $q\ge2$. We set $\sigma^2 := \var (\varepsilon)$. The training sample $\mathcal{D}_n:=((X_1,Y_1),\ldots ,(X_n,Y_n))$ is given by $n\in\mathbb N$ independent copies of $(X,Y)$.

In a regression context, random forests are constructed as an aggregation of \emph{randomized regression trees} of the form
\begin{align*}
  \tilde{m}_{n}( x_0;\Theta)&:=\sum_{i\in  I } Y_{i} \frac{\mathbb{I}\{X_{i}\in A_{n}( x_0,\omega,\mathcal{D}_{n})\}}{\#_I A_n( x_0,\omega,\mathcal{D}_n)}\qquad \text{with}\\
  \#_I A_n( x_0,\omega,\mathcal{D}_n)&:=\sum_{j\in I}  \mathbb{I}\{X_{j}\in A_{n}( x_0,\omega,\mathcal{D}_{n})\}.
\end{align*}
This is a piecewise constant approximation of the regression function based on a randomized, possibly $\mathcal D_n$ dependent, partition of the future space $[0,1]^p$. We denote the cell of the partition which contains a point $x_0\in[0,1]^p$ by $A_{n}( x_0,\omega,\mathcal{D}_{n})$. In this notation the random variable $\omega\in\Omega$ encodes the random mechanism to construct the partition. Note that the above tree is constructed only based on a subsample $I\subset\{1,\dots,n\}$ which allows for an increased randomization between different regression trees. Thus, $\#_I A_n( x_0,\omega,\mathcal{D}_n)$ is the number of observations in the subsample $(X_i)_{i\in I}$ that fall into $A_n( x_0,\omega,\mathcal{D}_n)$. As originally proposed by \cite{Breiman2004}, we consider partitions constructed iteratively. Starting with $A_0=[0,1]^p$, the cells at level $l$ are generated by splitting the cells at a level $l-1$ orthogonal to one of the axes.

With $N\in\mathbb N$ trees, each relying on different subsets $I_j$ and different random partitions determined by $\omega_j$, $j=1,\dots,N$, we obtain the \emph{random forest}
\begin{equation*}
	\hat{m}_{N,n}( x_0):=\frac{1}{N}\sum_{j=1}^{N}\tilde{m}_{n}( x_0;\Theta_{j})
	= \frac{1}{N}\sum_{j=1}^{N} \sum_{i\in  I_j } Y_{i} \frac{\mathbb{I}\{X_{i}\in A_{n}( x_0,\omega_j,\mathcal{D}_{n})\}}{\#_I A_n( x_0,\omega_j,\mathcal{D}_n)}.
\end{equation*}
We throughout consider \emph{centered purely random forests} (CPRF).  \emph{Purely random} means that the partition does not depend on the training data. \emph{Centered} refers to the fact that the cells are always split in the middle. For the tree construction we perform $k\in\mathbb{N}$ iterations of splitting, and in each iteration every existing cell is split. Therefore, each final cell is the result of $k$ splits and has the volume $2^{-k}$.
Due to this construction we subsequently use the notation $A_k(x_0,\omega)$ instead of $A_{n}( x_0,\omega,\mathcal{D}_{n})$. Both properties simplify the analysis of CPRF compared to the CART-based random forests by \cite{Breiman1984} and makes them a good starting point to study statistical inference.  The statistical analysis of CPRFs dates back to \citet{Biau2012}, although statistical inference for CPRFs has only more recently attracted attention.

In order to construct confidence bands, we will interpret random forests as U-statistics. Recall that for i.i.d.\ $Z_{1},\ldots,Z_{n}\in[0,1]^p\times\mathbb R$ and a symmetric kernel $h\colon([0,1]^p\times\mathbb R)^{r_n}\times\Omega\to\mathbb R$ of order $r_n \overset{n \to \infty }{\longrightarrow}  \infty $ a generalized (incomplete) U-statistic is given by
\[
	U_{n,r_n} :=\frac{1}{\hat N}\sum_{I\in \setbinom{n}{r_n}}\rho_Ih((Z_i)_{i \in I};\omega_I),\qquad \rho_{I} \overset{\mathrm{i.i.d.}}{\sim}  \mathrm{Bin}\Big(1,\frac{N}{\binom{n}{r_{n}}}\Big),\quad \hat{N}:=\sum_{I\in \setbinom{n}{r_n}} \rho_{I}.
\]
As pointed out by \cite{Peng2022} with only minor modifications we can rewrite random forests based on subsamples of size $r_n$ in this form: Let $(\rho_I,\omega_I)\in\{0,1\}\times\Omega$ be i.i.d.\ for $I\in\setbinom{n}{r_n}$ and define
\begin{align}
	U_{n,r_{n},N,\omega}^{(\mathrm{RF})}({ x_0})
	:=&\frac{1}{\hat{N}}\sum_{I\in \setbinom{n}{r_n}} \rho_{I} \underbrace{\sum_{i\in I}Y_{i} \frac{ \mathbb{I}\{X_{i}\in A_{k}(x_0,\omega_I)\}}{\#_I A_{k}(x_0,\omega_I)}}_{h_{{ x_0}}((Z_i)_{i \in I};\omega_I)}\label{eq:PRF U-stat}\\
	\approx&\frac{1}{\binom{n}{r_n}}\sum_{I\in \setbinom{n}{r_n}} \sum_{i\in I}Y_{i} \frac{ \mathbb{I}\{X_{i}\in A_{k}(x_0,\omega_I)\}}{\#_I A_{k}(x_0,\omega_I)}=:U_{n,r_{n},\omega}^{(\mathrm{RF})}({ x_0}),\label{U_n-rn-omega}
\end{align}
where the approximation seems appropriate if $N$ is sufficiently large. We will see in the proofs, that $N\gg 2^kn$ is sufficient for asymptotic equivalence. For $N=\binom{n}{r_{n}}=\hat N$ the approximation is exact, such that $U_{n,r_{n},\omega}^{(\mathrm{RF})}$ is the complete U-statistic version of the random forest.

What remains to specify is the distribution of the split direction. For
$l\in[p]$, $t\in[k]$ let $S_{l,t}(x_{0},\omega)$ denote the number
of splits orthogonal to the $l$-th coordinate that were used in the
first $t$ steps in the construction of $A_{k}(x_{0},\omega)$. In particular, $\sum_{l=1}^{p}S_{l,t}(x_{0},\omega)=t$. For $t=k$ we abbreviate
$S_{l,k}(x_{0},\omega)=S_{l}(x_{0},\omega)$ and assume homogeneity in the sense of $S_{l}(x_{1},\omega)\overset{d}{=}S_{l}(x_{2},\omega)$ for all $l\in[p]$ and $x_{1},x_{2}\in[0,1]^{p}$.

The most prominent example is the \emph{uniformly centered purely random forest} where at each step the split direction is chosen uniformly among all $p$ directions. In this case $S_l(x_0,\omega)\sim\mathrm{Bin}(k,1/p)$ and the joint distribution of $S_1(x_0,\omega),\dots,S_p(x_0,\omega)$ is multinomial with $k$ trails and all probabilities equal to $1/p$. Uniform CPRFs have been studied by \cite{Biau2012} and \cite{Klusowski2021} among others, who have demonstrated that this construction leads to suboptimal rates of convergence. The underlying problem is, that the approximation quality of the trees depends on the diameter of the cells and uniformly distributed splits allow for long and narrow cells with too high probability. To cure this problem, we will introduce a random splitting algorithm which can be understood as a regularized version of the uniform CPRF.

\section{A rate optimal centered purely random forest}\label{sec:rates}
The rate of convergence of a CPRF is determined by the following three characteristics:
\begin{enumerate}
 \item The \emph{diameter of a cell} in the partition is important for the ability
of the estimator to approximate the regression function. For any $A\subset\mathbb{R}^{p}$
let us denote its diameter by
\begin{equation*}
\mathfrak{d}(A):=\sup_{x_{1},x_{2}\in A}\Vert x_{1}-x_{2}\Vert.
\end{equation*}
  \item The expected volume of \emph{cell intersections} $A_{k}(x_{0},\omega_{1})\cap A_{k}(x_{0},\omega_{2})$ for i.i.d.\ copies
$\omega_{1}$ and $\omega_{2}$ of $\omega$ is a key quantity for the stochastic error of a CPRF. We denote the expected volume of the intersection by
\begin{equation*}
\mathcal{V}_{\cap,k}:=\mathbb{E}\left[\mathrm{vol}(A_{k}(x_{0},\omega_{1})\cap A_{k}(x_{0},\omega_{2}))\right].
\end{equation*}
  \item By construction any CPRF has a \emph{finest partition of the feature space}. We call a set $A$ undividable if it is a subset of one of the cells in any realization of the random partition. We call a set $A$ maximum undividable if there is no other undividable set $\tilde{A}\neq A$ with $A\subset\tilde{A}$. The set of all maximum undividable sets, denoted by $\mathcal{S}_{k}$, is a partition itself. If $x_{1}$ and $x_{2}$ are in the same maximum undividable set, we have $A_{k}(x_{1},\omega)=A_{k}(x_{2},\omega)$ almost surely. We note that this also implies that every realization of $A_{k}(x,\omega)$ can be represented as a disjoint union of maximum undividable sets. We write
  $$\mathcal N_{k}:=\#\mathcal{S}_{k}$$ and let $\mathcal{X}_{k}\subset[0,1]^{p}$
  with $\#\mathcal{X}_{k}=\mathcal N_{k}$ be a set containing exactly one element in each of the maximum undividable sets. While $\mathcal N_{k}$ and $\mathcal{X}_{k}$ might depend on the distribution of $\omega$, we omit this dependence in their notation for convenience.
\end{enumerate}
We can relate both, the diameter and the cell intersections, to the splits per direction. Since the length of the cell $A_k(x_0,\omega)$ in direction $l\in[p]$ is given by $2^{-S_l(x_0,\omega)}$ we have
\begin{align}
  \mathfrak{d}(A_{k}(x_{0},\omega))&=\Big(\sum_{l=1}^{p}2^{-2S_{l}(x_{0},\omega)}\Big)^{1/2}\qquad\text{and}\nonumber\\
  \mathrm{vol}(A_{k}(x_{0},\omega_{1})\cap A_{k}(x_{0},\omega_{2})) &=\prod_{l=1}^p2^{-\max\{S_{l}(x_{0},\omega_{1}),S_{l}(x_{0},\omega_{2})\}}.\label{eq:vol intersec darst}
\end{align}
Together with $\sum_{l=1}^{p}S_{l}(x_{0},\omega)=k$ for all $\omega$ the volume of cell intersections thus always satisfies
\begin{equation}
2^{-2k}\leq\mathcal{V}_{\cap,k}\leq2^{-k}.\label{eq:V cap simple bounds}
\end{equation}

The  uniform split distribution leads to the following properties:
\begin{lemma}[Characteristics of uniform CPRF]\label{lem:CharacteristicsUCPRF}
  The uniform CPRF satisfies for $q\in(0,2]$
  \begin{gather}
     \mathbb{E}\left[\mathfrak{d}(A_{k}(x_{0},\omega))^{q}\right] \leq p\left(\frac{p-1+2^{-q}}{p}\right)^{k},\nonumber\\
     2^{-k}k^{-(p-1)}\lesssim\mathcal{V}_{\cap,k}\lesssim2^{-k}k^{-(p-1)/2},\label{eq:uni V cap rate}\\
\mathcal{N}_{k}=2^{kp}.\nonumber
  \end{gather}
\end{lemma}
The proof can be found in the Section \ref{sec:Proof-strategy chap CBs} of the supplement. 

In particular, we note that the expected diameter ($q=1$) for large dimensions $p$ is of the order $p(1-\frac{1}{2p})^{k}\approx pe^{-k/(2p)}\approx p1.65^{-k/p}$ which is considerably larger than the minimal possible diameter $\sqrt p2^{-k/p}$ of a deterministic partition into $2^k$ cubes along an equidistant grid. This is the reason for the suboptimal rates of convergence of uniform CPRF, cf. \cite{Biau2012}. On the other hand, the specific split distribution allows for a much tighter bound on the cell intersections compared to \eqref{eq:V cap simple bounds}.

To circumvent cells with too large diameter, we now introduce \emph{Ehrenfest centered purely random forests} whose construction is motivated by the Ehrenfest model of particle diffusion. The model
has $p$ numbered containers, each corresponding to one possible split direction. Each contain $B\in\mathbb N$ identical particles in the initial state. In each step, a uniformly chosen particle
from the total of $pB$ particles is moved from its current container to another. The current state of the model after $t$ steps is described by $(B_{l,t})_{l\in[p]}$, denoting the number of particles in the containers $l\in[p]$. The connection from this particle model to the CPRF is given by the following mechanism: If at time $t\in[k]$ a particle in container $l\in[p]$ is selected, the cell $A_t(x_0,\omega)$ is split orthogonal to direction $l$, that means
\[
\mathbb{P}\big(S_{l,t+1}(x_0,\omega)-S_{l,t}(x_0,\omega)=j\mid B_{l,t}\big)=\begin{cases}
\frac{B_{l,t}}{pB}, & j=1,\\
1-\frac{B_{l,t}}{pB}, & j=0.
\end{cases}
\]
We prohibit the addition of particles to a container if the corresponding
feature has already been split too many times in relation to $t/p$.
To this end, the target probability of a particle leaving container $i$, for entering
container $j\neq i$ is chosen as
\begin{equation}
\varrho((S_{l,t}(x_0,\omega))_{l=1}^{p},j,i)=\frac{\mathbb{I}\{pS_{j,t}(x_0,\omega)<t+p\Delta\}}{\sum_{l\in\{1,\ldots,p\}\setminus\{i\}}\mathbb{I}\{pS_{l,t}(x_0,\omega)<t+p\Delta\}},\label{eq:destination prob Ehr}
\end{equation}
for some $\Delta>B/p$. While for small $\Delta$ we obtain a histogram like behavior, for $B,\Delta\to\infty$ we would recover a uniform CPRF.

\begin{prop}[Characteristics of the Ehrenfest CPRF]\label{prop:CharacteristicsEhrenfest}
  The Ehrenfest CPRF with parameters $B$ and $\Delta$ satisfies for any $x_0\in[0,1]^p$ and any realization $\omega$ and any $x_0$
  \begin{align*}
\frac{k}{p}-C_{\Delta,B}^{(2)}\leq S_{l,k}(x_0,\omega) \leq\frac{k}{p}+C_{\Delta,B}^{(1)},\qquad l\in[p],
\end{align*}
with $C_{\Delta,B}^{(1)}:=\Delta+pB$ and $C_{\Delta,B}^{(2)}:=(p-1)(\Delta+pB)$. This implies
  \begin{gather*}
     \mathfrak{d}(A_{k}(x_{0},\omega))\leq\sqrt{p}2^{C_{\Delta,B}^{(1)}}2^{-k/p}\lesssim2^{-k/p},\\
	2^{-k}2^{-p(C_{\Delta,B}^{(1)}+C_{\Delta,B}^{(2)})/2}\leq\mathcal{V}_{\cap,k}\le2^{-k},\\
	\mathcal{N}_{k}\leq2^{k+pC_{\Delta,B}^{(1)}}\lesssim 2^k.
  \end{gather*}
\end{prop}
The proof can be found in the Section \ref{sec:Proof-strategy chap CBs} of the supplement. 

\begin{remark}
Proposition~\ref{prop:CharacteristicsEhrenfest} remains true for any other choice of the probabilities \linebreak $\varrho((S_{l,t}(x_{0},\omega))_{l=1}^{p},j,i)$ as long as they are zero, if $pS_{j,t}(x_{0},\omega)\geq t+p\Delta$. 
\end{remark}

Next, we investigate asymptotic properties of a general CPRF $U_{n,r_n,\omega}^{(\mathrm{RF})}$ from (\ref{U_n-rn-omega}) and the Ehrenfest CPRF in particular. Let $\mathcal{H}^\alpha(C_{H})$ denote the set of all bounded $\alpha$-H\"older functions with H\"older constant and supremum norm less or equal to $C_{H}$.

\begin{theorem}[Pointwise convergence of CPRFs]
\label{thm:MSE better} If $\mathbb{E}[\varepsilon^{2}]<\infty$, the subsample size satisfies
$2^{k}\leq r_{n}\le n$ and the regression function $m\in\mathcal{H}^\alpha(C_{H})$ is Hölder regular of order $\alpha\in(0,1]$ with $C_H>0$, then
\begin{align}
\mathbb{E}\left[\big(U_{n,r_{n},\omega}^{(\mathrm{RF})}(x_{0})-m(x_{0})\big)^{2}\right] 
& \lesssim\mathbb{E}\left[\mathfrak{d}(A_{k}(x_{0},\omega))^{\alpha}\right]^{2}+\frac{2^{k}}{n}\Big(2^{k}\mathcal{V}_{\cap,k}+\frac{2^{k}}{r_{n}}\Big)\nonumber \\
 & \qquad+e^{-c_Xr_n/2^k}+\frac{2^{k}}{r_{n}}\left(\frac{r_{n}}{n}\right)^{r_{n}}.\label{eq:MSE new}
\end{align}
If $r_n/n\le c\in(0,1)$ or $r_n=n$, the Ehrenfest CPRF satisfies
\[
\mathbb{E}\big[(U_{n,r_{n},\omega}^{(\mathrm{RF})}(x_{0})-m(x_{0}))^{2}\big]\lesssim2^{-2\alpha k/p}+\frac{2^{k}}{n}+e^{-c_Xr_n/2^k}.
\]
If $k=\lfloor\log_{2}(c'n^{\frac{1}{(1+2\alpha/p)}})\rfloor$ for
a constant $c'$ and $r_{n}$ satisfies $n^{\frac{1}{(1+2\alpha/p)}}\log n=o(r_{n})$, we have
\[
\mathbb{E}\big[(U_{n,r_{n},\omega}^{(\mathrm{RF})}(x_{0})-m(x_{0}))^{2}\big]\lesssim n^{-\frac{2\alpha}{2\alpha+p}}.
\]
\end{theorem}
The proof is given in Section \ref{sec:Proof-strategy chap CBs} of the supplement. 

We observe that the Ehrenfest CPRF achieves the minimax optimal rate which is a consequence of the control of the cell diameter. The first terms on the right hand side of (\ref{eq:MSE new}) correspond to the approximation error. The second term is an upper bound of the stochastic error, while the last two terms are usually negligible.

In addition to the rate of convergence, we obtain a central limit theorem in an undersmoothing regime, which not only allows for the construction of pointwise confidence intervals but also gives a first impression on how the uniform confidence band might look like.

\begin{prop}[Pointwise central limit theorem for CPRFs]\label{prop:CLT}
 Let $m\in\mathcal{H}^\alpha(C_{H})$ for $\alpha\in(0,1],C_H>0$. For any fixed $x_{0}\in[0,1]^p$ assume that $r_{n}/n\leq c$ for some $c<1$, $\mathbb{E}[\varepsilon^{2}]<\infty$, and
 \begin{gather*}
   \frac{1}{\mathcal{V}_{\cap,k}r_{n}}	\to0,\qquad
\frac{2^{k}}{r_{n}}\log\big(n2^{-2k}\mathcal{V}_{\cap,k}^{-1}\big)	\to0\qquad\text{and}\\
\mathbb{E}\left[\mathfrak{d}(A_{k}(x_{0},\omega))^{\alpha}\right]^{2}\frac{n}{2^{2k}\mathcal{V}_{\cap,k}}	\to0.
   \end{gather*}
Then
\begin{equation*}
 \sqrt{\frac{n}{\sigma^{2}\Psi_k(x_{0})}}\left(U^{(\mathrm{RF})}_{n,r_{n},\omega}(x_{0})-m(x_{0})\right)\stackrel{d}{\longrightarrow}\mathcal{N}(0,1),
\end{equation*}
where $\Psi_k(x_{0})$ will be defined in \eqref{eq:def psi}.
\end{prop}
The proof is given in Section \ref{sec:Proof-strategy chap CBs} of the supplement. 
The conditions in Proposition~\ref{prop:CLT} are weaker than for the central limit theorem by \cite{Peng2022} due to a different proof strategy. More precisely, the above result is a corollary of the proof of our main confidence band result. In contrast to \cite{Peng2022}, we allow for subsamples sizes in the order of $n$ and we only need second moments of the observation errors. In view of Lemma~\ref{lem:CharacteristicsUCPRF} and Proposition~\ref{prop:CharacteristicsEhrenfest} the first condition is $r_{n}2^{-k}	\to\infty$ up to a polynomial term in $k$ for the uniform CPRF and the Ehrenfest CPRF. With different logarithmic terms it thus coincides with the second condition and ensures consistent estimation in all cells. The third condition implies that the bias is negligible compared to the stochastic error and thus allows for a central limit theorem centered around the true regression function $m(x_0)$. The variance term $\sigma^2\Psi(x_0)/n$ will be discussed in more detail after the confidence band construction.

Similar to the pointwise rates of congergence, we obtain uniform convergence of CPRF which seems to be new for the uniform CPRF, too. 
\begin{theorem}[Uniform convergence of a CPRF]
\label{thm:uni convergence} Let $\alpha\in(0,1],C_H>0$. Consider a centered purely random forest
with at most $\mathcal N_{k}$ undividable cells. Assume $r_{n}/n\leq c<1$
and $r_{n}/k\to\infty$. Let $\nu>4$ with $\mathbb{E}[\vert\varepsilon\vert^{\nu}]<\infty$ and suppose $\frac{2^k}{n}\le n^{-2/\nu}(\log n)^{-3}$.
Then we have uniformly for all $m\in\mathcal{H}^\alpha(C_{H})$ that
\begin{align*}
\Vert & U_{n,r_{n},\omega}^{(\mathrm{RF})}-m\Vert_{\infty}
=\mathcal{O}_{\mathbb{P}}\left(\mathbb{E}\big[\mathfrak{d}(A_{k}(x,\omega))^{\alpha}\big]+\sqrt{\frac{2^k}{n}}\Big(1+\mathcal N_k^{1/\nu}\sqrt{\frac{2^k}{r_n}}+\sqrt{k2^k\mathcal V_{\cap,k}}\Big)+e^{-\tau r_n/2^{k}}\right)
\end{align*}
for some constant $\tau>0$.
\end{theorem}
The proof is given in Section \ref{sec:Proof-strategy chap CBs} of the supplement. 
Compared to the pointwise result in Theorem~\ref{thm:MSE better}, we only need slightly stronger assumptions. First, we need a certain speed of the convergence $2^k/n\to0$ which has to be polynomially in $n$, but if we have sufficiently many finite moments then the polynomial decay can be very slow and does not contradict rate optimal choices of $k$. We also require a minimal speed of $r_n\to\infty$, which basically has to be at least logarithmic in $n$. The most interesting case $r_n=cn$ for some $c\in(0,1)$ is allowed (as in Proposition~\ref{prop:CLT}), while the boundary case $r_n=n$ is permitted in contrast to Theorem~\ref{thm:MSE better}. Unsurprisingly, the uniform bound on the stochastic error is larger than in the pointwise result. On the one hand, it depends moderately on the number of undividable cells $\mathcal N_k$ and, on the other hand, the leading stochastic error term $\sqrt{k2^{2k}\mathcal V_{\cap,k}/n}$ is larger than in the pointwise result by a factor $\sqrt k$. Therefore, in combination with the properties from the Ehrenfest CPRF from Proposition~\ref{prop:CharacteristicsEhrenfest} a rate optimal choice $2^k=(n/\log n)^{1/(1+2\alpha/p)}$ and $r_n=cn$ yields the minimax optimal rate
\[
\Vert  U_{n,r_{n},\omega}^{(\mathrm{RF})}-m\Vert_{\infty}=\mathcal O_{\mathbb P}\left(\Big(\frac{\log n}{n}\Big)^{\alpha/(2\alpha+p)}\right).
\]
Again, the larger approximation error of the uniform CPRF will lead to a deterioration of the convergence rate for the uniform error.


\section{Confidence bands}\label{sec:Confidence-Bands}

This section is dedicated to the construction of confidence bands for a CPRF.
While our focus is on the complete version of the U-statistic $U_{n,r_n,\omega}^{(\mathrm{RF})}$  defined in (\ref{U_n-rn-omega}), the extension of the results to the incomplete version $U_{n,r_{n},N,\omega}^{(\mathrm{RF})}$ from (\ref{eq:PRF U-stat}) is feasible, see Corollary \ref{cor:CB Incomplete U-stat}.



First we need several
definitions. Let us denote
\begin{align}
p_{x_{0}}(\omega) & :=\mathbb{P}(X_{1}\in A_{k}(x_{0},\omega)\mid\omega) =\int_{A_{k}(x_{0},\omega)}f_{X}(x)dx\quad\text{and}\label{eq:def p_x(omega)}
\\
p_{x_{0}} & :=\mathbb{E}[p_{x_{0}}(\omega)]=\mathbb{P}(X_{1}\in A_{k}(x_{0},\omega)).\label{eq:px}
\end{align}
By construction of the centered trees, we have $\int_{A_{k}(x_{0},\omega)}dx=2^{-k}$
and thus (\ref{eq:density assum}) implies
\begin{align}
c_{X}2^{-k} & \leq p_{x_{0}}(\omega)\leq C_{X}2^{-k},
\label{eq:bound px omega}\\
c_{X}2^{-k} & \leq p_{x_{0}}\leq C_{X}2^{-k}.\label{eq:bound px}
\end{align}
Let us define
\begin{equation}
K_{k}(x_{0},x):=\mathbb{E}\left[\mathbb{I}\{x\in A_{k}(x_{0},\omega)\}p_{x_{0}}(\omega)^{-1}\right],\label{eq:RF Hajek kernel}
\end{equation}
which can be understood as a kernel function. Its second moment is denoted by
\begin{equation}
\Psi_{k}(x_{0}):=\mathbb{E}\big[K_{k}(x_{0},X_{1})^{2}\big]\label{eq:def psi}
\end{equation}
and determines the asymptotic variance in the pointwise central limit theorem in Proposition~\ref{prop:CLT}.
Under the assumptions on the density the variance function satisfies
\begin{equation}
\frac{c_{X}}{C_{X}^{2}}2^{2k}\mathcal{V}_{\cap,k}\leq\Psi_{k}(x_{0})\leq\frac{C_{X}}{c_{X}^{2}}2^{2k}\mathcal{V}_{\cap,k}\label{eq:uni bounds for Psi}
\end{equation}
uniformly in $x_{0}\in[0,1]^{p}$.
Moreover, we will require the function class
\begin{align}
\mathcal{F}_{k}:=\big\{&f_{x_{0},k}(x,s)\,\big|\, x_{0}\in[0,1]^{p}\big\},\qquad\text{where}\label{eq:F_k}\\
&f_{x_{0},k}(x,s):=\sigma^{-1}\Psi_{k}^{-1/2}(x_{0})sK_{k}(x_{0},x).\label{eq:f_x;n}
\end{align}


\begin{theorem}[Asymptotic confidence band for CPRF]
\label{thm:CB main01} Consider any centered purely random forest
with at most $\mathcal N_{k}$ undividable sets.
Let $\alpha\in(0,1],C_{H}>0$ and for some $c<1$ assume that
\begin{gather}
r_{n}/n  \leq c\qquad\textit{and}\label{eq:assum remainder 2}\\
\mathbb{E}\left[\mathfrak{d}(A_{k}(x,\omega))^{\alpha}\right]^{2}\frac{n(\log n)^{2}}{2^{2k}\mathcal{V}_{\cap,k}}  \to0\label{eq:assum approx diam}
\end{gather}
for some $x\in[0,1]^p$.
Let $\nu\in\mathbb{N}$ with $\nu\ge2$ such that
\begin{align}
\mathcal N_{k}^{1/\nu}\frac{(\log n)^{2}}{r_{n}\mathcal{V}_{\cap,k}} & \to0,\label{eq:assum remainder 1}\\
n^{1/\nu}\frac{(\log n)^{5}}{n\mathcal{V}_{\cap,k}} & \to0,\label{eq:assum Gauss sup}\\
\mathbb{E}\left[\vert\varepsilon\vert^{2\nu}\right] & <\infty.\label{eq:assum moments}
\end{align}
Let $\mathbf{S}_{k}$ be a sequence of random variables with $\mathbf{S}_{k}\overset{d}{=}\sup_{x_{0}\in[0,1]^{p}}\vert B_{k}f_{x_{0},k}\vert$,
where $B_{k}$ is a sequence of centered Gaussian processes indexed
by $\mathcal{F}_{k}$ and with covariance function
\begin{equation}
\cov\big(B_{k}(f_{x_{1},k}),B_{k}(f_{x_{2},k})\big)=\Psi_{k}^{-1/2}(x_{1})\Psi_{k}^{-1/2}(x_{2})\mathbb{E}\left[K_{k}(x_{1},X_{1})K_{k}(x_{2},X_{1})\right].\label{eq:Gaussian process}
\end{equation}
Let $\hat{\sigma}$ be an estimator of $\sigma$ with
\[
\mathbb{P}\big(\vert\hat{\sigma}^{2}-\sigma^{2}\vert>(\log n)^{-2}\big)\to0.
\]
For $c_{k}(\beta)=F_{\mathbf{S}_{k}}^{-1}(1-\beta):=\inf\{\xi\in\mathbb{R}:\mathbb{P}(\mathbf{S}_{k}\leq\xi)\geq1-\beta\}$
define
\begin{equation*}
\mathcal{C}_{n}(x)=\left[U_{n,r_{n},\omega}^{(\mathrm{RF})}(x)-\hat{\sigma}c_{k}(\beta)\sqrt{\frac{\Psi_{k}(x)}{n}},U_{n,r_{n},\omega}^{(\mathrm{RF})}(x)+\hat{\sigma}c_{k}(\beta)\sqrt{\frac{\Psi_{k}(x)}{n}}\right].
\end{equation*}
Then we have
\[
\liminf_{n\to\infty}\inf_{m\in\mathcal{H}^\alpha(C_{H})}\mathbb{P}(m(x)\in\mathcal{C}_{n}(x),\,\forall x\in[0,1]^{p})\geq1-\beta.
\]
\end{theorem}
\begin{proof}[Sketch of the proof]
  We first decompose into an approximation and a stochastic error:
  \begin{align*}
   	U_{n,r_{n},\omega}^{(\mathrm{RF})}(x_0)
	=&\frac{1}{\binom{n}{r_n}}\sum_{I\in \setbinom{n}{r_n}} \sum_{i\in I}Y_{i} \frac{ \mathbb{I}\{X_{i}\in A_{k}(x_0,\omega_I)\}}{\# A_{k}(x_0,\omega_I)}\\
	=&\underbrace{\frac{1}{\binom{n}{r_n}}\sum_{I\in \setbinom{n}{r_n}} \sum_{i\in I}m(X_{i}) \frac{ \mathbb{I}\{X_{i}\in A_{k}(x_0,\omega_I)\}}{\# A_{k}(x_0,\omega_I)}}_{\displaystyle =m(x_0)+\mathcal O_{\mathbb P}(\E[\mathfrak{d}(A_{k}(x,\omega))^{\alpha}])}
	+\underbrace{\frac{1}{\binom{n}{r_n}}\sum_{I\in \setbinom{n}{r_n}} \sum_{i\in I}\varepsilon_{i} \frac{ \mathbb{I}\{X_{i}\in A_{k}(x_0,\omega_I)\}}{\# A_{k}(x_0,\omega_I)}}_{\displaystyle=:U_{n,r_{n},\omega}^{(\varepsilon)}(x_0)}.
  \end{align*}
  To deal with the stochastic error term $U_{n,r_n,\omega}^{(\varepsilon)}(x_0)$, we apply a  H\'ajek projection to approximate
  \begin{align}
   	U_{n,r_{n},\omega}^{(\varepsilon)}(x_0)
	=&\frac{r_n}{n}\sum_{i=1}^n \varepsilon_i\frac{1}{\binom{n-1}{r_n-1}}\sum_{i\in I\in \setbinom{n}{r_n}}\frac{ \mathbb{I}\{X_{i}\in A_{k}(x_0,\omega_I )\}}{\# A_{k}(x_0,\omega_I)}\label{eq:Ueps}\\
	\approx&\frac{r_n}{n}\sum_{i=1}^n \varepsilon_i\frac{1}{\binom{n-1}{r_n-1}}\sum_{i\in I\in \setbinom{n}{r_n}}\frac{ \mathbb{I}\{X_{i}\in A_{k}(x_0,\omega_I)\}}{r_n p_{x_0}(\omega_I)}
	\approx\frac{1}{n}\sum_{i=1}^n \varepsilon_i K_k(x_0,X_i)\nonumber
  \end{align}
  with $p_{x_0}(\omega)\sim 2^{-k}$ from (\ref{eq:def p_x(omega)})  and $K_k(x_0,x)$ from (\ref{eq:RF Hajek kernel}).
  With $\Psi_k(x_0)$ from (\ref{eq:def psi}), we can thus approximate the stochastic error by the empirical process
  \begin{align*}
	U_{n,r_{n},\omega}^{(\varepsilon)}(x_0)\approx&\frac{1}{n}\sum_{i=1}^n \varepsilon_i K(x_0,X_i)
	=\sqrt{\frac{\sigma^2\Psi_k(x_0)}{n}}{\mathbb G_n} f_{x_0,k}
  \end{align*}
  with $f_{x_0,k}$ defined in (\ref{eq:f_x;n}). 
  Now applying the result from \cite{Chernozhukov2014} yields for the function class $\mathcal F_k$ from (\ref{eq:F_k}) and a sequence of Gaussian processes $B_k$
  \begin{align*}
    \sup_{f\in\mathcal F_k}|{\mathbb G_n}f|&\overset{n\to\infty}{\approx} \mathbf{S}_k\overset d=\sup_{f\in\mathcal F_k}|B_kf|\quad\text{with}\\
    \cov(B_kf_{x_1,k},B_kf_{x_2,k})&=\E[f_{x_1,k}(X,\varepsilon)f_{x_2,k}(X,\varepsilon)]\\
    &=(\Psi_k(x_1)\Psi_k(x_2))^{-1/2}\E[K_k(x_1,X_1)K_k(x_2,X_1)].
  \end{align*}
  It remains to quantify errors due to the above simplifications, which is detailed out in our proofs in the supplement material.
\end{proof}

The detailed proof is given in Section \ref{sec:Proof-strategy chap CBs} of the supplement. 
The proof idea is similar to deriving confidence bands based on histogram estimators considered by \citet{Neumeyer-Rabe-Trabs}, but considerably more challenging for random forests. For example the approximation of the stochastic error by the empirical process has to be shown in a different way. 

The confidence band from Theorem \ref{thm:CB main01} depends on the function $\Psi_{k}$ from \eqref{eq:def psi}. Since the distribution of $\omega$ is implicitly known from the partitioning algorithm, the function $K_k$ from (\ref{eq:RF Hajek kernel}) can be approximated by a Monte Carlo simulation. This yields an approximation of $\Psi_{k}$ if the distribution of $X$ is known. Otherwise one can estimate the expectations in the definitions of $p_{x_0}(\omega)$ and $\Psi_k$ by empirical means over the observed covariates.  The fact that $\Psi_{k}$ is in general not constant implies that the confidence band does not have a constant diameter. We obtain a smaller than average radius of the confidence band in regions of the feature space where $f_{X}$ is relatively large. 
Note that the result holds uniformly for $m$ in the class $\mathcal{H}^\alpha(C_{H})$ such that we obtain asymptotically honest confidence regions in the sense of \citet{Li1989}.

Our assumption on the variance estimator $\hat\sigma^2$ is very moderate.
In the fixed design case based on a kernel estimator \cite{Hall-Marron} considered a residual-based variance  estimator $\hat\sigma^2=n^{-1}\sum_{i=1}^n\hat\varepsilon_i^2$ which can also be applied with $\hat\varepsilon_i=Y_i-U_{n,r_n,\omega}^{(\mathrm{RF})}(X_i)$. Other possibilities are U-statistic based variance estimators considered by \cite{Mueller2003} and \cite{Shen2020}. 

Thanks to the approximation result by \cite{Chernozhukov2014}, a noteworthy aspect of Theorem \ref{thm:CB main01} is that it does not utilize quantiles
of a limit distribution. Instead, the quantiles $c_{k}(\beta)$ are
those of $\mathbf{S}_{k}$, the supremum of the Gaussian process $B_{k}$,
and thus depend on $k$. It is common in the literature for confidence
bands to be based on a limit distribution, usually an extreme value
distribution, see e.g., \citet{Bickel1973}. The structure of our
result has two advantages. First, we do not need to know the limit
distribution of $\mathbf{S}_{k}$. In particular, it is not even necessary
that a limit distribution exists. Second, there is no additional error term due to the convergence to the limit distribution.

Before discussing the assumptions in the theorem,
we state two corollaries, for the two types of CPRFs we introduced in Sections~\ref{sec:Random-forests} and~\ref{sec:rates}, respectively, and have a look on the resulting radii of the confidence bands.
\begin{cor}[Asymptotic confidence band for the Ehrenfest CPRF]
For the Ehrenfest CPRF let $\alpha\in(0,1],C_H>0$ and assume that $r_n/n\le c$ for some $c\in(0,1)$ and
\[
\frac{n(\log n)^{2}}{2^{k(1+2\alpha/p)}}\to0.
\]
Further let $\nu\in\mathbb{N},\nu\ge2$ such that
\begin{align*}
\frac{2^{k(1+1/\nu)}(\log n)^{2}}{r_{n}}  \to0,\qquad
\frac{2^{k}(\log n)^{5}}{n^{1-1/\nu}}  \to0\qquad \text{and}\qquad 
\mathbb{E}\left[\vert\varepsilon\vert^{2\nu}\right]  <\infty.
\end{align*}
Let $\hat{\sigma}$, $c_{k}(\beta)$ and $\mathcal{C}_{n}$ be as
is in Theorem \ref{thm:CB main01} for $\omega$ distributed according
to the Ehrenfest partition. Then,
\[
\liminf_{n\to\infty}\inf_{m\in\mathcal{H}^\alpha(C_{H})}\mathbb{P}(m(x)\in\mathcal{C}_{n}(x),\,\forall x\in[0,1]^{p})\geq1-\beta.
\]
\end{cor}
\begin{proof} The corollary for the Ehrenfest CPRF follows directly from Theorem
\ref{thm:CB main01} by noting that $\mathcal{V}_{\cap,k}\sim 2^{-k}$,
$\mathcal N_{k}\lesssim2^{k}$ and using $\mathfrak{d}(A_{k}(x,\omega))^{\alpha}\lesssim2^{-\alpha k/p}$
almost surely, see Proposition \ref{prop:CharacteristicsEhrenfest}. 
\end{proof}
If $\nu$ is sufficiently large, the
assumptions are satisfied, for instance, by $k=\lfloor\log_{2}((\log n)^{3}n^{\frac{p}{p+2\alpha}})\rfloor$
and $r_{n}=\lfloor cn\rfloor$ for some $c\in(0,1)$. In particular, $r_{n}\sim n$ is eligible here. This parameter choice implies
that the rate of the radius is
\begin{align*}
c_{k}(\beta)\sqrt{2^{k}/n} & 
\lesssim\sqrt{k2^k/n}\lesssim(\log n)^{2}n^{-\frac{\alpha}{p+2\alpha}},
\end{align*}
because $c_{k}(\beta)\lesssim\sqrt{k}$, 
which follows from the proof of Theorem~\ref{thm:CB main01}. 

The corollary for the uniform CPRF below requires slightly stronger
assumptions.
\begin{cor}[Asymptotic confidence band for the uniform CPRF]
\label{cor:CB uni}For the uniform CPRF let $\alpha\in(0,1],C_H>0$ and
assume that $r_n/n\le c$ for some $c\in(0,1)$ and
\[
\left(\frac{p-1+2^{-\alpha}}{p}\right)^{2k}\frac{n(\log n)^{2}k^{(p-1)}}{2^{k}}\to0.
\]
Further let $\nu\in\mathbb N$, $\nu\ge2$ such that
\begin{align*}
\frac{2^{k(1+p/\nu)}k^{p-1}(\log n)^{2}}{r_{n}}  \to0,\quad
\frac{2^{k}k^{p-1}(\log n)^{5}}{n^{1-1/\nu}}  \to0\quad \text{and}\quad
\mathbb{E}\left[\vert\varepsilon\vert^{2\nu}\right]  <\infty.
\end{align*}
Let $\hat{\sigma}$, $c_{k}(\beta)$ and $\mathcal{C}_{n}$ be as
is in Theorem \ref{thm:CB main01} for $\omega$ distributed according
to the uniform partition. Then,
\[
\liminf_{n\to\infty}\inf_{m\in\mathcal{H}^\alpha(C_{H})}\mathbb{P}(m(x)\in\mathcal{C}_{n}(x),\,\forall x\in[0,1]^{p})\geq1-\beta.
\]
\end{cor}
\begin{proof}
Corollary \ref{cor:CB uni} follows directly from Theorem \ref{thm:CB main01} and Lemma \ref{lem:CharacteristicsUCPRF}.
\end{proof}
We note that $r_{n}\sim n$ is eligible again. For $c\in(0,1)$ we thus choose $r_{n}=\lfloor cn\rfloor$ and $k=\lfloor\log_{2}((n(\log_{2}n)^{p+2})^{1-\mu})\rfloor$
for
\[
\mu=\frac{2\log_{2}b(p,\alpha)}{2\log_{2}b(p,\alpha)-1}\qquad\text{for}\qquad b(p,\alpha)=(p-1+2^{-\alpha})/p 
\]
such that $0<\mu\leq\frac{2\alpha}{2\alpha+p}$.
Together with $c_{k}(\beta)=\mathcal{O}(\sqrt{k})$ and (\ref{eq:uni V cap rate}) our choice of $r_n$ and $k$ imply that the rate of the radius satisfies
\[
c_{k}(\beta)\sqrt{\frac{\Psi_{k}(x)}{n}}\lesssim\sqrt{\frac{k2^{2k}\mathcal{V}_{\cap,k}}{n}}\lesssim\sqrt{k^{1-(p-1)/2}2^k/n} \lesssim\sqrt{(\log_{2}n)^{-\mu(p+2)+(p+7)/2}n^{-\mu}}.
\]
On the one hand, for a fixed $k$ the diameter of the confidence bands for the uniform CPRF is smaller than for the Ehrenfest CPRF by a polynomial factor in $k$. On the other hand, an asymptotically optimized choice of $k$ results in a worse rate of convergence in comparison to the Ehrenfest CPRF which is in line with the analysis of the mean squared error.
\medskip

Let us now discuss the assumptions in Theorem \ref{thm:CB main01}.
The two corollaries above allow us to evaluate the assumptions for the two specific CPRF types.

Assumption (\ref{eq:assum remainder 2}) is not a major restriction
at all, and it is important to note that it allows for $r_{n}=cn$, which is not possible with most asymptotic results from the literature if the subsampling is taken into account at all. Often $r_n=o(n)$ is imposed to prove asymptotic normality, see for example \citet[Theorem 1]{Mentch2016}, where relatively small $r_n$ helps to ensure that the first order terms, i.e.\ the H\'ajek projection, dominate in the Hoeffding decomposition of the U-statistic.

Assumption (\ref{eq:assum approx diam}) is an undersmoothing assumption that is common for confidence band constructions if no bias correction is applied. Note that $\mathbb{E}[\mathfrak{d}(A_{k}(x,\omega))^{\alpha}]$ is independent of $x$ due to the construction of the CPRF. 
If the diameter has the optimal rate $2^{-k/p}$ and $\mathcal{V}_{\cap,k}\sim 2^{-k}$, which is the case for the Ehrenfest forest, the assumption simplifies to
\[
\mathbb{E}[\mathfrak{d}(A_{k}(x,\omega))^{\alpha}]^{2}\frac{n(\log n)^{2}}{2^{2k}\mathcal{V}_{\cap,k}}\lesssim n(\log n)^{2}2^{-k(1+2\alpha/p)}\to0.
\]
For the uniform CPRF the exponent of $2^{-k}$ is smaller, namely $1+\log_{2}((p-1+2^{-\alpha})/p)$ instead of $1+2\alpha/p$, and additionally the logarithmic terms are larger. This is the only assumption directly depending on the dimension $p$ which affects the diameter and thereby the bias. For a fixed sample size $n$ a larger $p$ thus requires to choose a larger $k$ to meet the assumptions.

Assumption (\ref{eq:assum remainder 1}) is used
to bound the uniform difference of the stochastic error
to its H\'ajek projection.
Without the logarithmic terms, assumption (\ref{eq:assum remainder 1})
means
\begin{equation}
\frac{\mathcal N_{k}^{1/\nu}}{r_{n}\mathcal{V}_{\cap,k}}\to0.\label{eq:assum rem without log}
\end{equation}
Note that the assumption implicitly depends on $p$ via $\mathcal N_{k}$. Since $\mathcal N_{k}$ increases with $p$, we need higher moments of the errors for larger $p$. If we can choose $\nu$ large enough, (\ref{eq:assum rem without log})
roughly corresponds to $r_{n}\mathcal{V}_{\cap,k}\to\infty$,
which implies
\begin{equation}
2^{k}/r_{n}\leq\mathcal{V}_{\cap,k}^{-1}r_{n}^{-1}\to0.\label{eq:implication assum 2^k/r_n}
\end{equation}
We note that $2^{k}=\mathcal{O}(r_{n})$ is necessary anyway, because otherwise the number of empty cells in a tree partition increases
in $n$. Therefore, the only relevant case that contradicts (\ref{eq:implication assum 2^k/r_n})
is $2^{k}\sim r_{n}$. Recalling the assumption $r_{n}\leq cn$, this allows for the most reasonable choices for $r_{n}$.
For the considered CPRFs we cannot control the number of observations in the cells directly, but only through the tuning parameters $k$ and $r_{n}$. Therefore, an assumption like (\ref{eq:assum remainder 1})
is necessary to ensure non-empty cells.

Omitting again the logarithms, assumption (\ref{eq:assum Gauss sup}) simplifies to
\[
\mathcal{V}_{\cap,k}n^{1-1/\nu}\to\infty.
\]
For both CPRFs, $\mathcal{V}_{\cap,k}$ is of order $2^{-k}$, at least up to terms polynomial in $k$. Note that $\mathcal{V}_{\cap,k}\leq2^{-k}$
requires that $n$ grows at a rate exceeding $2^{k}$. We recall that
$2^{-k/p}$ is comparable to a bandwidth for kernel estimators, which
illustrates that (\ref{eq:assum Gauss sup}) is similar to standard
assumptions. The exact rate by which $n$ must exceed $2^{k}$ depends
mainly on the existing moments captured by $\nu$ and further on
the exact rate of $\mathcal{V}_{\cap,k}$ and the logarithmic terms
from the assumption.

The moment assumption (\ref{eq:assum moments}) on the noise is fulfilled for all sub-Gaussian distributions for an arbitrarily large $\nu$. The interplay of the existing finite moments with the assumptions~\eqref{eq:assum Gauss sup} and \eqref{eq:assum approx diam} controls which regimes of tuning parameters can be used, or vice versa, which moments have to exist in these regimes.

\medskip 

So far we have considered the complete generalized U-statistic. The
following corollary applies to the incomplete generalized U-statistic $U_{n,r_{n},N,\omega}^{(\mathrm{RF})}(x)$  defined in (\ref{eq:PRF U-stat}), where $N$ is the average number of trees in the random forests.
\begin{cor}
\label{cor:CB Incomplete U-stat}In addition to the assumptions of
Theorem \ref{thm:CB main01} assume that
\begin{equation*}
\frac{n(\log n)^{2}\mathcal N_{k}^{2}}{N2^{2k}\mathcal{V}_{\cap,k}}\to0.
\end{equation*}
Consider $U_{n,r_{n},N,\omega}^{(\mathrm{RF})}(x)$ from equation (\ref{eq:PRF U-stat})
on the event $\{\hat{N}>0\}$. For
\[
\mathcal{C}_{n,N}(x)=\left[U_{n,r_{n},N,\omega}^{(\mathrm{RF})}(x)-\hat{\sigma}c_{k}(\beta)\sqrt{\frac{\Psi_{k}(x)}{n}},U_{n,r_{n},N,\omega}^{(\mathrm{RF})}(x)+\hat{\sigma}c_{k}(\beta)\sqrt{\frac{\Psi_{k}(x)}{n}}\right]
\]
we have
\[
\liminf_{n\to\infty}\inf_{m\in\mathcal{H}^\alpha(C_{H})}\mathbb{P}\left(m(x)\in\mathcal{C}_{n,N}(x),\,\forall x\in[0,1]^{p}\right)\geq1-\beta.
\]
\end{cor}

The proof can be found in Section \ref{sec:Proof-strategy chap CBs} of the supplement. \textbf{}
It suffices to consider the estimator only on the event $\{\hat{N}>0\}$. Otherwise we would use an empty random forest. The corollary shows
that the incomplete version of the U-statistic can still be used to construct confidence bands provided that $N$ is large enough. For
the Ehrenfest forest with independent covariates, the condition simplifies
to
\[
\frac{n(\log n)^{2}\mathcal N_{k}^{2}}{2^{2k}\mathcal{V}_{\cap,k}}\frac{1}{N}\lesssim\frac{2^{k}n(\log n)^{2}}{N}\to0.
\]
In practice, it may not be necessary to select $N$ as large as this sufficient
condition demands. The simulations in Section \ref{chap:Simulation-study}
will suggest that a considerably smaller $N$ is sufficient in practice.

\section{Simulation study\label{chap:Simulation-study}}

In this section we present a simulation study to illustrate our theoretical results in practice. The complete source code related to this section
is available\footnote{GitHub repository available at https://github.com/Jan-Rabe/Asymptotic-CBs-for-CPRF}. 

Throughout we use uniformly
distributed $X\sim\mathcal{U}([0,1]^p)$. In this case the confidence band has a constant diameter, the asymptotic variance function is given by 
$\Psi_k(x_0)=2^{2k}\mathcal V_{\cap,k}$
and the covariance from (\ref{eq:Gaussian process}) simplifies to 
\begin{equation}
\cov(B_{k}(f_{x_{1},k}),B_{k}(f_{x_{2},k}))=\mathcal{V}_{\cap,k}^{-1}\mathbb{E}\left[\mathrm{vol}\left(A_{k}(x_{1},\omega_{1})\cap A_{k}(x_{2},\omega_{2})\right)\right].\label{eq:cov for sim}
\end{equation}
To construct the confidence band from Theorem
\ref{thm:CB main01}, we thus need $\mathcal{V}_{\cap,k}$
as well as the quantiles of $\mathbf{S}_{k}$.
To this end, the covariance matrix can be calculated on some fine grid in $[0,1]^{p}$, where it is not always computationally feasible to use representatives of all cells of the finest partition of the feature space. Simulating the Gaussian process $B_{k}$ on this grid, we compute its supremum via a Monte Carlo approximation. A detailed description of the numerical calculation of the covariance \eqref{eq:cov for sim} of the Gaussian process together with the quantiles of its supremum can be found in Section~\ref{sec:Dist of S_k} of the supplementary material.


For all simulations, we use $50\,000$ simulated pairs of splits to estimate the covariance matrix entries \eqref{eq:cov for sim} and $\mathcal{V}_{\cap,k}$. We simulate $100\,000$ suprema of the Gaussian process $B_k$ to approximate the quantiles of the supremum. For estimating the variance $\sigma^2$ the U-statistic estimator by \citet{Shen2020} with  bandwidth $h=n^{-1/2}$ and a Gaussian kernel is applied. 
 With the empirical quantiles, the estimation of $\mathcal{V}_{\cap,k}$ and the estimated standard deviation of the errors we are able to compute the confidence
band diameter. 

A comparison of the diameter to the maximal estimation error $\vert U_{n,r_{n},\omega}^{(\mathrm{RF})}(x_{0})-m(x_{0})\vert$ reveals the coverage properties of the confidence band. We evaluate the estimation error on the grid
\begin{equation}
x_0\in G_{\tilde{k}}^{(sup)}:=\left(\left\{a2^{-\tilde{k}}-\epsilon,a2^{-\tilde{k}}+\epsilon\;\Big|\; a\in\{1,\ldots,2^{\tilde{k}}-1\}\right\}\cup\{\epsilon,1-\epsilon\}\right)^{p}\label{eq:sup test grid}
\end{equation}
for a small $\epsilon>0$. Since the distance between a smooth function and a piecewise constant function is the largest at the jump points of the piecewise function, if the smooth function is monotone on the
corresponding intervals, the maximum over the above grid should almost capture the full uniform error. 

We start with simulations in the two dimensional case $p=2$, where we use the regression function
\begin{equation}
m(x^{(1)},x^{(2)})=\frac{1}{10}\big(\sin(2\pi x^{(1)})+x^{(2)}\big).\label{eq:m sim p2}
\end{equation}
Table \ref{tab:sim asymp 1} summarizes the simulation results for depth $k=5$, sample size $n=2000$, subsample size $r_{n}=\frac{3}{4}n=1500$, noise levels $\sigma\in\{0.75,1,1.25\}$ and 4 different error distributions, namely a normal, a
uniform and two $t$-distributions. The Ehrenfest forest uses the tuning parameters $B=12$ and $\Delta=7$. The table shows the empirical coverage and the average radius of the confidence
bands. For a fixed error distribution and $\sigma$, we use the same random seed for the training sample generation, to achieve a better comparison. In particular, the uniform CPRF and the Ehrenfest CPRFs build on the same data for both values of $N$. Since we apply an estimator $\hat\sigma$ which does not rely on a the random forest, the confidence band radius does not depend on $N$.

\begin{table}[t]
\caption[Empirical coverage and average confidence band radius of RF confidence
bands for different error distributions.]{\label{tab:sim asymp 1}Empirical coverage (top) and average confidence
band radius (Rd., bottom) of $1000$ RF confidence bands with $k=5$,
$n=2000$ and $r_{n}=1500$ for different error distributions.}

\centering{}%
\begin{tabular}{ccccccc}
\toprule
$\sigma$ & RF & $N$ & Normal & Uniform & $t$-dist. $4$ df & $t$-dist. $6$ df\tabularnewline
\midrule
 &  &  & \multicolumn{4}{c}{$1-\beta$}\tabularnewline
 &  &  & $.90$, $.95$, $.99$ & $.90$, $.95$, $.99$ & $.90$, $.95$, $.99$ & $.90$, $.95$, $.99$\tabularnewline
\midrule
\multirow{6}{*}{$0.75$} & \multirow{3}{*}{Uni.} & $50$ & $.656$, $.782$, $.941$ & $.679$, $.816$, $.954$ & $.617$, $.755$, $.911$ & $.606$, $.752$, $.924$\tabularnewline
 &  & $100$ & $.682$, $.812$, $.946$ & $.697$, $.828$, $.957$ & $.636$, $.772$, $.923$ & $.621$, $.760$, $.943$\tabularnewline
\cmidrule{3-7} \cmidrule{4-7} \cmidrule{5-7} \cmidrule{6-7} \cmidrule{7-7}
 &  & \multirow{1}{*}{Rd.} & $.228$, $.243$, $.272$ & $.228$, $.243$, $.273$ & $.228$, $.242$, $.272$ & $.228$, $.243$, $.272$\tabularnewline
\cmidrule{2-7} \cmidrule{3-7} \cmidrule{4-7} \cmidrule{5-7} \cmidrule{6-7} \cmidrule{7-7}
 & \multirow{3}{*}{Ehr.} & $50$ & $.648$, $.809$, $.945$ & $.681$, $.825$, $.958$ & $.646$, $.778$, $.918$ & $.620$, $.767$, $.933$\tabularnewline
 &  & $100$ & $.675$, $.810$, $.953$ & $.702$, $.826$, $.953$ & $.646$, $.786$, $.925$ & $.620$, $.767$, $.931$\tabularnewline
\cmidrule{3-7} \cmidrule{4-7} \cmidrule{5-7} \cmidrule{6-7} \cmidrule{7-7}
 &  & \multirow{1}{*}{Rd.} & $.236$, $.252$, $.283$ & $.236$, $.252$, $.283$ & $.235$, $.251$, $.282$ & $.236$, $.252$, $.282$\tabularnewline
\midrule
\multirow{6}{*}{$1$} & \multirow{3}{*}{Uni.} & $50$ & $.734$, $.839$, $.961$ & $.750$, $.884$, $.970$ & $.699$, $.821$, $.935$ & $.689$, $.815$, $.952$\tabularnewline
 &  & $100$ & $.748$, $.870$, $.965$ & $.773$, $.873$, $.975$ & $.722$, $.816$, $.941$ & $.707$, $.830$, $.953$\tabularnewline
\cmidrule{3-7} \cmidrule{4-7} \cmidrule{5-7} \cmidrule{6-7} \cmidrule{7-7}
 &  & \multirow{1}{*}{Rd.} & $.304$, $.324$, $.363$ & $.304$, $.324$, $.363$ & $.304$, $.323$, $.363$ & $.304$, $.323$, $.363$\tabularnewline
\cmidrule{2-7} \cmidrule{3-7} \cmidrule{4-7} \cmidrule{5-7} \cmidrule{6-7} \cmidrule{7-7}
 & \multirow{3}{*}{Ehr.} & $50$ & $.745$, $.866$, $.970$ & $.760$, $.876$, $.967$ & $.710$, $.825$, $.944$ & $.711$, $.822$, $.954$\tabularnewline
 &  & $100$ & $.754$, $.866$, $.966$ & $.770$, $.885$, $.971$ & $.724$, $.824$, $.942$ & $.701$, $.831$, $.952$\tabularnewline
\cmidrule{3-7} \cmidrule{4-7} \cmidrule{5-7} \cmidrule{6-7} \cmidrule{7-7}
 &  & \multirow{1}{*}{Rd.} & $.314$, $.336$, $.377$ & $.314$, $.336$, $.377$ & $.314$, $.335$, $.376$ & $.314$, $.335$, $.376$\tabularnewline
\midrule
\multirow{6}{*}{$1.25$} & \multirow{3}{*}{Uni.} & $50$ & $.777$, $.871$, $.969$ & $.789$, $.910$, $.975$ & $.735$, $.841$, $.945$ & $.738$, $.842$, $.956$\tabularnewline
 &  & $100$ & $.793$, $.897$, $.969$ & $.803$, $.908$, $.978$ & $.749$, $.838$, $.951$ & $.748$, $.859$, $.957$\tabularnewline
\cmidrule{3-7} \cmidrule{4-7} \cmidrule{5-7} \cmidrule{6-7} \cmidrule{7-7}
 &  & \multirow{1}{*}{Rd.} & $.380$, $.404$, $.454$ & $.380$, $.405$, $.454$ & $.379$, $.404$, $.454$ & $.380$, $.404$, $.454$\tabularnewline
\cmidrule{2-7} \cmidrule{3-7} \cmidrule{4-7} \cmidrule{5-7} \cmidrule{6-7} \cmidrule{7-7}
 & \multirow{3}{*}{Ehr.} & $50$ & $.790$, $.888$, $.973$ & $.804$, $.894$, $.975$ & $.752$, $.848$, $.950$ & $.740$, $.858$, $.961$\tabularnewline
 &  & $100$ & $.795$, $.888$, $.971$ & $.804$, $.908$, $.979$ & $.756$, $.854$, $.956$ & $.753$, $.865$, $.965$\tabularnewline
\cmidrule{3-7} \cmidrule{4-7} \cmidrule{5-7} \cmidrule{6-7} \cmidrule{7-7}
 &  & \multirow{1}{*}{Rd.} & $.393$, $.419$, $.471$ & $.393$, $.420$, $.471$ & $.392$, $.419$, $.470$ & $.392$, $.419$, $.470$\tabularnewline
\bottomrule
\end{tabular}
\end{table}

The coverage of the uniform RF and the Ehrenfest RF are similar, but both are smaller than the theoretical coverage probability $1-\beta$. There is no clear trend as to which CPRF has better coverage. However, the radius of the confidence bands is larger for the Ehrenfest forest which is in line with our theoretical discussion above. As the noise level $\sigma$ increases, the estimation error becomes dominated by the stochastic error resulting in a better coverage. 

On average, a larger number $N$ of trees in the forest increases the coverage. This effect is more pronounced for the uniform CPRF than for the Ehrenfest CPRF, which might indicate a computational advantage of the Ehrenfest CPRF. Indeed, the uniform CPRF has more diverse partitions, and thus one would expect it to require a larger $N$ to create this diversity in practice. For $N=100$ the coverage of the uniform CPRF is close to that of the Ehrenfest CPRF across the entire range of parameters. This is noteworthy because, at first glance, it seems to contradict the larger approximation error bound of the unifrom CPRF. Again, a possible reason could be the more diverse partitions leading to smaller undividable cells, even though the approximation error bound does not capture this effect.


The overall impression across the board is that the coverage for the uniform distribution is the best and  for the $t$-distribution with four degrees of freedom the worst. The coverage for the $t$-distribution with six degrees of freedom is worse than that for the normal distribution, but the difference is moderate. This suggests that a larger number of finite moments is helpful, which is again in line with
the theory.

\begin{figure}[!t]
\begin{centering}
\includegraphics[viewport=29bp 0bp 547bp 520bp,clip,width=0.7\textwidth]{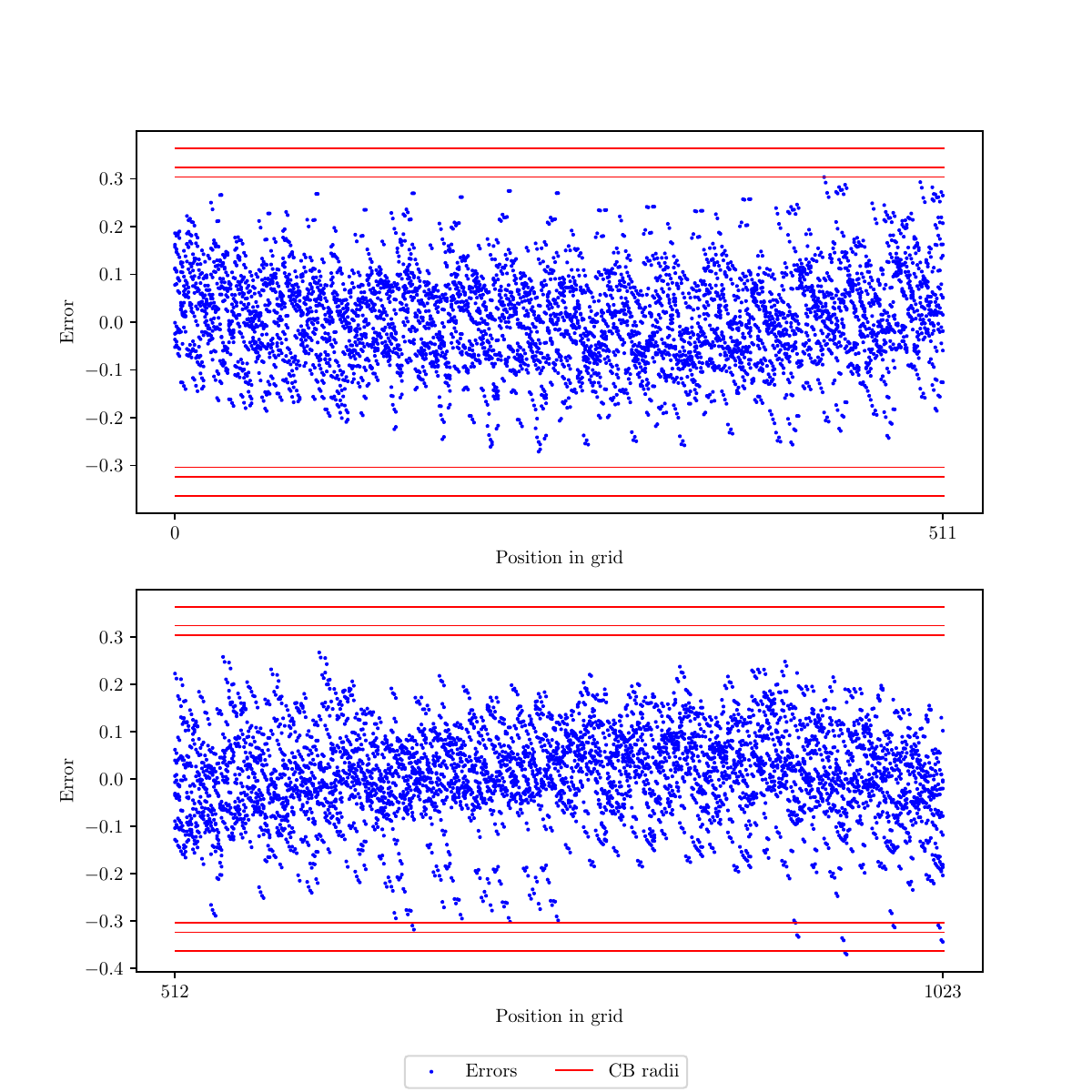}
\par\end{centering}
\caption{Scatter plot of estimation errors (blue dots) of ten uniform CPRF estimators with $n=2000$, $k=5$ and $\varepsilon\sim\mathcal{N}(0,1)$ on a non equidistant test grid and confidence band radii (red lines) corresponding to confidence levels $1-\beta\in\{0.9,0.95,0.99\}$. The horizontal axis corresponds to the test grid of the feature space from (\ref{eq:sup test grid}). Blocks of $32$ values on the horizontal axis correspond
to entries in the grid, that have the same first component. \label{fig:Errors scatter}}
\end{figure}

Figure~\ref{fig:Errors scatter} shows the estimation errors of ten uniform CPRF estimators in the case $n=2000$, $k=5$ and $\varepsilon\sim\mathcal{N}(0,1)$ together with the radii of the confidence bands for confidence levels $1-\beta\in\{0.9,0.95,0.99\}$. The $0.99$ and $0.95$
confidence bands cover nine of these ten estimators, while the $0.9$ confidence band covers eight. The plot provides insight into the dispersion of errors across the width of the confidence bands, demonstrating that errors outside the bands do not constitute extreme outliers relative to the remaining observations.

\begin{table}[t]
\begin{centering}
\caption{\label{tab:sim asymp 2}Empirical coverage and average confidence
band radius of $1000$ RF confidence bands with $\varepsilon\sim\mathcal{N}(0,1)$,
$N=100$ and $r_{n}=\frac{3}{4}n$.}
\par\end{centering}
\centering{}%
\begin{tabular}{cccccc}
\toprule
$n$ & RF & \multicolumn{2}{c}{$k=5$} & \multicolumn{2}{c}{$k=6$}\tabularnewline
\midrule
 &  & \multicolumn{4}{c}{$1-\beta$}\tabularnewline
 &  & $.90$, $.95$, $.99$ & $.90$, $.95$, $.99$ & $.90$, $.95$, $.99$ & $.90$, $.95$, $.99$\tabularnewline
 &  & Coverage & Radius & Coverage & Radius\tabularnewline
\midrule
\multirow{2}{*}{$250$} & \multirow{1}{*}{Uni.} & $.676$, $.778$, $.916$ & $.859$, $.915$, $1.027$ & $.599$, $.712$, $.877$ & $1.280$, $1.351$, $1.495$\tabularnewline
 & \multirow{1}{*}{Ehr.} & $.668$, $.781$, $.909$ & $.889$, $.949$, $1.065$ & $.627$, $.733$, $.897$ & $1.333$, $1.408$, $1.565$\tabularnewline
\midrule
\multirow{2}{*}{$500$} & \multirow{1}{*}{Uni.} & $.759$, $.852$, $.954$ & $.608$, $.648$, $.727$ & $.611$, $.756$, $.901$ & $.906$, $.956$, $1.058$\tabularnewline
 & \multirow{1}{*}{Ehr.} & $.770$, $.856$, $.946$ & $.629$, $.672$, $.754$ & $.629$, $.759$, $.903$ & $.944$, $.997$, $1.108$\tabularnewline
\midrule
\multirow{2}{*}{$1000$} & \multirow{1}{*}{Uni.} & $.802$, $.878$, $.963$ & $.430$, $.458$, $.514$ & $.778$, $.867$, $.962$ & $.641$, $.676$, $.748$\tabularnewline
 & \multirow{1}{*}{Ehr.} & $.805$, $.889$, $.964$ & $.445$, $.475$, $.533$ & $.760$, $.870$, $.955$ & $.667$, $.705$, $.783$\tabularnewline
\midrule
\multirow{2}{*}{$2000$} & \multirow{1}{*}{Uni.} & $.748$, $.870$, $.965$ & $.304$, $.324$, $.363$ & $.786$, $.881$, $.967$ & $.453$, $.478$, $.529$\tabularnewline
 & \multirow{1}{*}{Ehr.} & $.754$, $.866$, $.966$ & $.314$, $.336$, $.377$ & $.806$, $.887$, $.969$ & $.472$, $.498$, $.554$\tabularnewline
\midrule
\multirow{2}{*}{$4000$} & \multirow{1}{*}{Uni.} & $.658$, $.797$, $.947$ & $.215$, $.229$, $.257$ & $.814$, $.898$, $.967$ & $.320$, $.338$, $.374$\tabularnewline
 & \multirow{1}{*}{Ehr.} & $.674$, $.809$, $.946$ & $.222$, $.237$, $.266$ & $.813$, $.897$, $.968$ & $.333$, $.352$, $.391$\tabularnewline
\midrule
\multirow{2}{*}{$8000$} & \multirow{1}{*}{Uni.} & $.450$, $.640$, $.873$ & $.152$, $.162$, $.182$ & $.763$, $.856$, $.967$ & $.226$, $.239$, $.264$\tabularnewline
 & \multirow{1}{*}{Ehr.} & $.463$, $.639$, $.877$ & $.157$, $.168$, $.188$ & $.756$, $.870$, $.966$ & $.236$, $.249$, $.277$\tabularnewline
\bottomrule
\end{tabular}
\end{table}

In Table \ref{tab:sim asymp 2} we showcase the empirical coverage for different values of $n$ and $k$. For all results, the error distribution is $\varepsilon\sim\mathcal N(0,1)$, the tree number is $N=100$ and we use $r_{n}=0.75n$. Our theory suggests two important effects of $n$ and $k$. If $2^{k}$ is large in relation too $n$, the distribution of the stochastic error $U_{n,r_{n},\omega}^{(\varepsilon)}$, cf. \eqref{eq:Ueps}, is not well approximated by the asymptotic regime. If $2^{k}$ is small relative to $n$, the approximation error will be larger. Both effects may deteriorate the coverage of the confidence band. In the table we can see that the best coverage for $k=5$ is achieved at $n=1000$
(and $n=2000$ for the $99$\% bands), while in the case of $k=6$
the best coverage is achieved at $n=4000$ in line with the
theory. For the smaller values of $n$ we see a suboptimal coverage
for both values of $k$ which suggests that these sample sizes are too small for our asymptotics. For the largest values of $n$ we get the smallest radii of the confidence
bands, and we can see that this affects the coverage negatively because the approximation error  dominates the stochastic error for the given value of $k$ such that we are not anymore in an undersmoothing regime. This effect is more pronounced for the lower confidence levels.

Figure~\ref{fig:Comp CB cov and radii dep n} gives a graphical comparison of coverage and the radii of the uniform and Ehrenfest CPRF for $k=\{5,6\}$. As a reference estimator, we also added the corresponding numerical results for the classical histogram regression estimator relying on equidistant partition into squares of edge length $\delta\in\{1/5,1/7\}$, which is chosen such that the confidence bands have a similar radius as the CPRFs. The confidence band construction for the histogram estimator follows \cite{Neumeyer-Rabe-Trabs}. 
For all $\beta$ we observe that the radii of the two random forest bands for $k=5$ are slightly smaller than the radii of the histogram confidence bands for $\delta^{-1}=5.$
At the same time, the empirical coverage for these random forests is larger than that of the histogram for almost all $n$. The same observation can be made, when comparing the random forests for $k=6$ and the histogram for $\delta^{-1}=7$. Hence, the random forest bands work better throughout. As an illustration Figure~\ref{fig:heat map} shows a realization of the uniform CPRF estimator compared to the histogram estimator, $n=8000$, $k=5$, $\delta^{-1}=5$ for the same parameters as in Table \ref{tab:sim asymp 2}. 

\begin{figure}
\begin{centering}
\includegraphics[viewport=38bp 0bp 542bp 650bp,clip,width=0.9\textwidth]{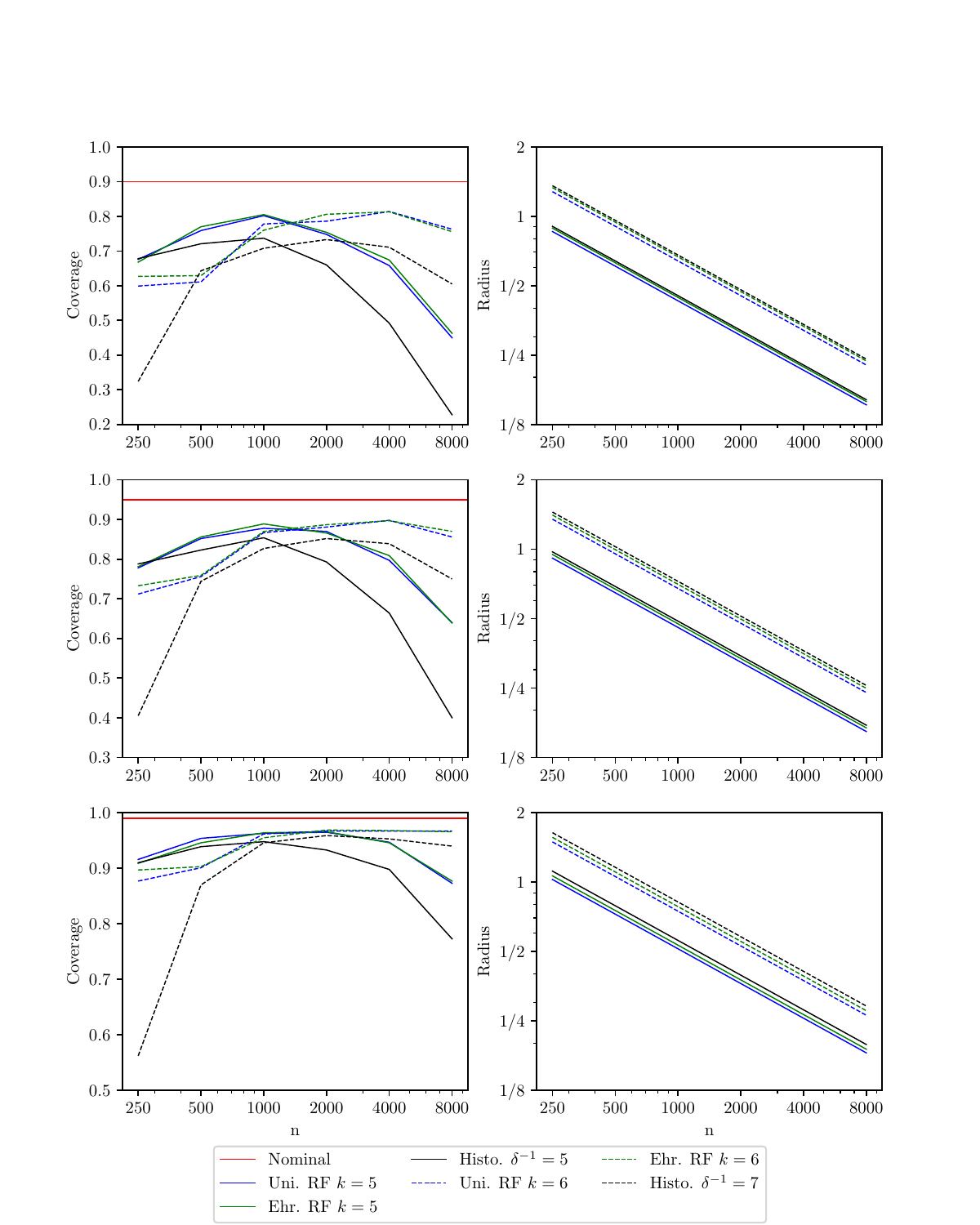}
\par\end{centering}
\caption{\label{fig:Comp CB cov and radii dep n}Comparison of confidence band coverage and radii for the uniform CPRF, the Ehrenfest CPRF and a histogram regression estimator in dependence on $n$. On the left we observe the empirical coverage together with the nominal coverage as horizonal line. On the right, the radii are plotted on a logarithmic scale. From top to bottom, the plots are for $1-\beta\in\{0.9,0.95,0.99\}$.}
\end{figure}

\begin{figure}
\begin{centering}
\includegraphics[viewport=30bp 14bp 431bp 330bp,clip,width=0.9\textwidth]{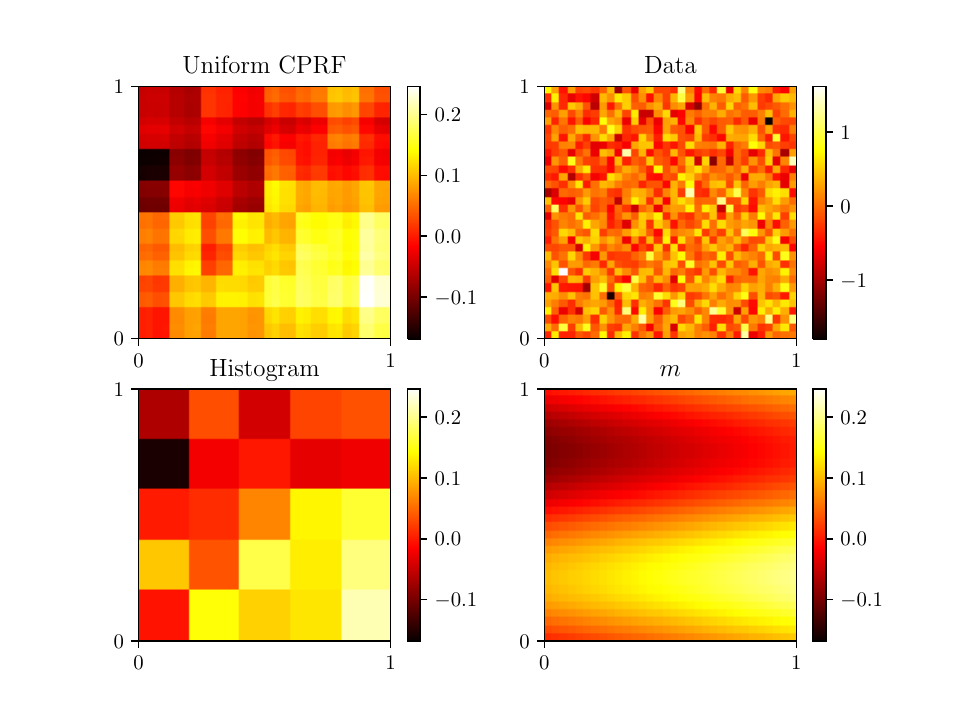}
\par\end{centering}
\caption{\label{fig:heat map}Realization of a uniform CPRF (top left) based on the observation shown in the heat map (top right), histogram estimation on the $2^{k}$-grid as a reference (bottom left) and true regression function $m$ from (\ref{eq:m sim p2}) (bottom right).}
\end{figure}

Finally, we compare the empirical coverage for two regression models with $p=2$ and $p=4$ for different values
of $\sigma$. To this end, we complement the  two-dimensional regression function from \eqref{eq:m sim p2} with
\[
m(x^{(1)},x^{(2)},x^{(3)},x^{(4)})=\frac{1}{10}\big(\sin(2\pi x^{(1)})+x^{(2)}+x^{(3)}x^{(4)}\big)
\]
for the case $p=4$.
Table~\ref{tab:sim asymp 3} provides the exact numbers, while a graphical representation of the results is given in Figure~\ref{fig:Comp CB cov rad dep sigma}.
The regression function in the first case is the same as before. In Theorem \ref{thm:CB main01} the dimension $p$ mainly effects the assumption on the diameter which controls the approximation error. Further, $p$ affects the quantities $\mathcal{V}_{\cap,k}$ and $\mathcal N_{k}$. The approximation error increases with $p$, which is standard in
multivariate regression problems due to the curse of dimensionality. The size of $\sigma$ directly affects the signal to noise ratio and thus the size of the approximation and the stochastic error. If the noise level is decreased similarly for $p=2$ and $p=4$, the empirical coverage decreases faster in the latter case. This is as expected
due to the larger approximation error. We point out that the confidence band radii are smaller for $p=4$ due to the effect on $\mathcal{V}_{\cap,k}$. When the empirical coverage is evaluated in a regression model with
$m=0$, i.e., one only considers $U_{n,r_{n},\omega}^{(\varepsilon)}$, the results are much better and comparable to the case $p=2$. This demonstrates that the effect of the higher dimension can be attributed almost exclusively to the approximation error.


\begin{table}
\caption{\label{tab:sim asymp 3}Empirical coverage (left) and average confidence
band radius (right) of $1000$ confidence bands with $\varepsilon\sim\mathcal{N}(0,\sigma^{2})$,
$n=2000$, $k=5$, $N=50$ and $r_{n}=\frac{3}{4}n$.}

\centering{}%
\begin{tabular}{cccccc}
\toprule
$\sigma$ & RF & \multicolumn{2}{c}{$p=2$} & \multicolumn{2}{c}{$p=4$}\tabularnewline
\midrule
 &  & \multicolumn{4}{c}{$1-\beta$}\tabularnewline
 &  & $.90$, $.95$, $.99$ & $.90$, $.95$, $.99$ & $.90$, $.95$, $.99$ & $.90$, $.95$, $.99$\tabularnewline
 &  & Coverage & Radius & Coverage & Radius\tabularnewline
\midrule
\multirow{2}{*}{$0.5$} & \multirow{1}{*}{Uni.} & $.416$, $.616$, $.849$ & $.152$, $.162$, $.182$ & $.002$, $.006$, $.041$ & $.124$, $.131$, $.146$\tabularnewline
 & \multirow{1}{*}{Ehr.} & $.416$, $.598$, $.868$ & $.157$, $.168$, $.188$ & $.001$, $.004$, $.046$ & $.126$, $.133$, $.149$\tabularnewline
\midrule
\multirow{2}{*}{$0.75$} & \multirow{1}{*}{Uni.} & $.656$, $.782$, $.941$ & $.228$, $.243$, $.272$ & $.065$, $.116$, $.359$ & $.186$, $.197$, $.219$\tabularnewline
 & \multirow{1}{*}{Ehr.} & $.648$, $.809$, $.945$ & $.236$, $.252$, $.283$ & $.056$, $.132$, $.382$ & $.189$, $.200$, $.223$\tabularnewline
\midrule
\multirow{2}{*}{$1$} & \multirow{1}{*}{Uni.} & $.734$, $.839$, $.961$ & $.304$, $.324$, $.363$ & $.188$, $.321$, $.607$ & $.248$, $.262$, $.291$\tabularnewline
 & \multirow{1}{*}{Ehr.} & $.745$, $.866$, $.970$ & $.314$, $.336$, $.377$ & $.202$, $.345$, $.635$ & $.252$, $.267$, $.297$\tabularnewline
\midrule
\multirow{2}{*}{$1.25$} & \multirow{1}{*}{Uni.} & $.777$, $.871$, $.969$ & $.380$, $.404$, $.454$ & $.320$, $.469$, $.747$ & $.310$, $.328$, $.364$\tabularnewline
 & \multirow{1}{*}{Ehr.} & $.790$, $.888$, $.973$ & $.393$, $.419$, $.471$ & $.338$, $.502$, $.742$ & $.315$, $.334$, $.371$\tabularnewline
\bottomrule
\end{tabular}
\end{table}

\begin{figure}
\begin{centering}
\includegraphics[viewport=38bp 0bp 538bp 650bp,clip,width=0.9\textwidth]{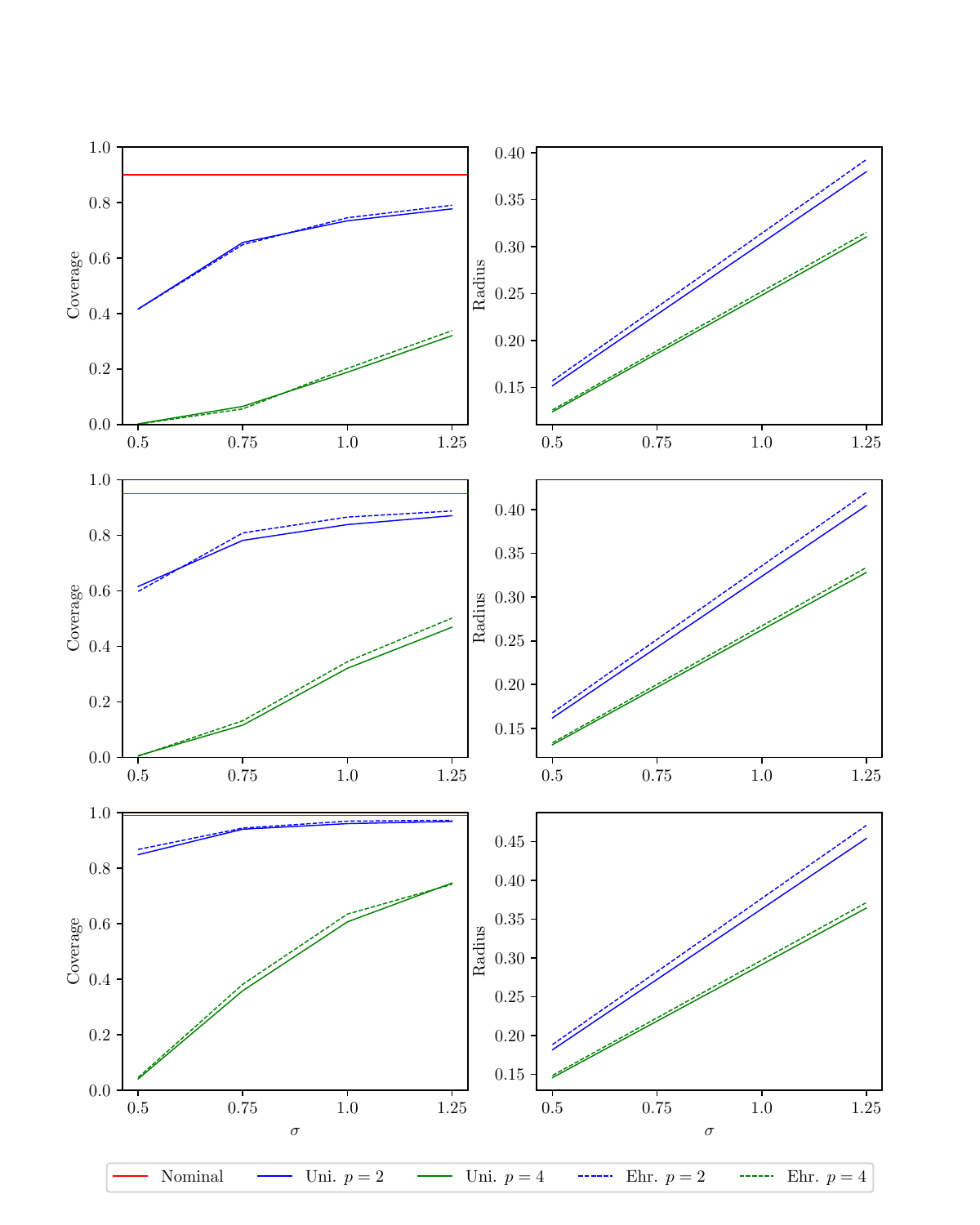}
\par\end{centering}
\caption{\label{fig:Comp CB cov rad dep sigma}Comparison of confidence band coverage and radii for the uniform CPRF and the Ehrenfest CPRF in dependence on $n$ and the dimension $p=2,4$. On the left we observe the empirical coverage together with the nominal coverage as horizonal line. On the right, the radii are plotted on a logarithmic scale. From top to bottom, the plots are for $1-\beta\in\{0.9,0.95,0.99\}$.}
\end{figure}

\appendix
\clearpage
\begin{center}
   {\Large\bf Supplementary material for asymptotic confidence bands for centered purely random forests}
\end{center}

\medskip

In Section \ref{sec:Dist of S_k}  the distribution of $\mathbf{S}_{k}$ is analyzed. In Section \ref{sec:Proof-strategy chap CBs}
all proofs of the  results in Section \ref{sec:Confidence-Bands} and \ref{sec:rates} are given. In the proofs moment inequalities are applied which are presented in Section \ref{sec:emp pro =000026 lit Results}. 

\section{Distribution of $\mathbf{S}_{k}$ }\label{sec:Dist of S_k}

\subsection{Theoretical properties}\label{secA1}

In this section we will analyze the distribution of $\mathbf{S}_{k}\overset{d}{=}\sup_{f\in\mathcal{F}_{k}}\vert B_{k}f\vert$ from Theorem \ref{thm:CB main01}.
The results will help to approximate the distribution of $\mathbf{S}_{k}$ by a Monte Carlo simulation. Recall the definition of the stochastic error term 
\begin{align*}
   	U_{n,r_{n},\omega}^{(\varepsilon)}(x_0)
	=&\frac{r_n}{n}\sum_{i=1}^n \varepsilon_i\frac{1}{\binom{n-1}{r_n-1}}\sum_{i\in I\in \setbinom{n}{r_n}}\frac{ \mathbb{I}\{X_{i}\in A_{n}(x_0,\omega_I )\}}{\# A_{n}(x_0,\omega_I)}
\end{align*}
from \eqref{eq:Ueps}. Figure \ref{fig:densities S_k} shows the estimated density of $\mathbf{S}_{k}$ for $k=5$ and the densities of
\[
\sqrt{\frac{n}{\sigma^{2}2^{2k}\mathcal{V}_{\cap,k}}}\Vert U_{n,r_{n},\omega}^{(\varepsilon)}\Vert_{\infty}
\]
for $k=5$ and different values of $n$. All densities are estimated
by a Gaussian kernel density estimator with the same bandwidth. The
density of $\Vert U_{n,r_{n},\omega}^{(\varepsilon)}\Vert_{\infty}$
is estimated based on $1\,000$ copies of the supremum for each $n$.
For $\mathbf{S}_{k}$, the more efficient simulation described in
Section~\ref{chap:Simulation-study} (and below) allows us to use $100\,000$ copies of $\mathbf{S}_{k}$ for the density estimation. 

\begin{figure}[h]
\centering{}\includegraphics[viewport=15bp 5bp 446bp 330bp,clip,width=0.7\textwidth]{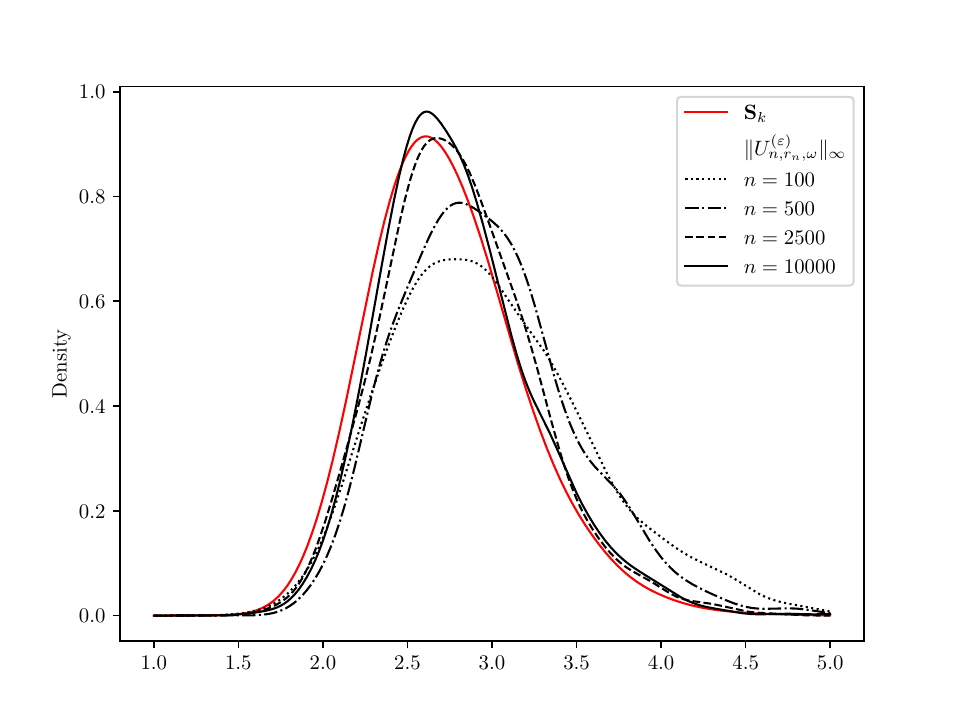}\caption[Estimated densities of $\mathbf{S}_{k}$ and standardized $\Vert U_{n,r_{n},\omega}^{(\varepsilon)}\Vert_{\infty}$
for different $n$.]{\label{fig:densities S_k}Estimated densities of standardized $\Vert U_{n,r_{n},\omega}^{(\varepsilon)}\Vert_{\infty}$
for different $n$ and $\mathbf{S}_{k}$, both for $k=5$.}
\end{figure}

We consider the covariance function of the processes $B_k$ given by
\begin{equation*}
\cov(B_{k}(f_{x_{1},k}),B_{k}(f_{x_{2},k}))=\Psi_{k}^{-1/2}(x_{1})\Psi_{k}^{-1/2}(x_{2})\mathbb{E}\left[K_{k}(x_{1},X_{1})K_{k}(x_{2},X_{1})\right],
\end{equation*}
see (\ref{eq:Gaussian process}).
The definition of $f_{x_{0},k}\in\mathcal{F}_{k}$ from (\ref{eq:f_x;n}) implies that the functions
only depend on $x_{0}$ via the cells $A_{k}(x_{0},\omega)$.
Since there are finitely many of these cells the function
class is finite with $\#\mathcal{F}_{k}=\mathcal N_{k}$ in the notations from Section~\ref{sec:rates}. Therefore, the covariance function can be characterized by a $\mathcal N_k\times\mathcal N_k$ matrix.

If $X$ is uniformly distributed, (\ref{eq:def p_x(omega)}) and (\ref{eq:px}) imply $p_{x_{0}}(\omega)=2^{-k}$
and  $p_{x_{0}}=2^{-k}$. For $K_{k}$ from (\ref{eq:RF Hajek kernel}) we obtain
$K_{k}(x_{0},x)=\mathbb{P}(x\in A_{k}(x_{0},\omega))2^{k}
$ and for $\Psi_{k}$ from (\ref{eq:def psi}) we have $\Psi_{k}(x_{0})=\mathcal{V}_{\cap,k}2^{2k}$. The covariance thus simplifies to
\[\cov(B_{k}(f_{x_{1},k}),B_{k}(f_{x_{2},k}))=\mathcal{V}_{\cap,k}^{-1}\mathbb{E}\left[\mathrm{vol}\left(A_{k}(x_{1},\omega_{1})\cap A_{k}(x_{2},\omega_{2})\right)\right]\]
from (\ref{eq:cov for sim}). To calculate the volume of the intersection of two cells recall that $S_{l}(x_{0},\omega)$ denotes the number of splits orthogonal to coordinate
$l$ used in the construction of the cell $A_{k}(x_{0},\omega)$.
\begin{lemma}
\label{lem:vol intersec}For $x_i=(x_i^{(1)},\dots,x_i^{(p)})\in[0,1]^p,i=1,2,$ we have
\begin{align*}
\mathrm{vol} & \left(A_{k}(x_{1},\omega_{1})\cap A_{k}(x_{2},\omega_{2})\right)\\
 & =2^{-\sum_{l=1}^{p}\max\{S_{l}(x_{1},\omega_{1}),S_{l}(x_{2},\omega_{2})\}}\\
 & \qquad\times\prod_{l=1}^{p}\mathbb{I}\left\{ \lfloor x_{1}^{(l)}2^{\min\{S_{l}(x_{1},\omega_{1}),S_{l}(x_{2},\omega_{2})\}}\rfloor=\lfloor x_{2}^{(l)}2^{\min\{S_{l}(x_{1},\omega_{1}),S_{l}(x_{2},\omega_{2})\}}\rfloor\right\} .
\end{align*}
\end{lemma}
The proof is given below in Subsection \ref{appendix A.3}. 
The lemma reveals that the size of the intersection is determined by the maximum number of cuts per direction if the intersection is not empty. Furthermore, it is not empty if and only if the $l$-th components
of $x_{1}$ and $x_{2}$ are in the same cell of the grid with cell
size $2^{-\min(S_{l}(x_{1},\omega_{1}),S_{l}(x_{2},\omega_{2})}$
for all $l\in[p]$. We can thus calculate the
size of the intersection based on the number of splits per direction in the partitions created by $\omega_{1}$ and $\omega_{2}$. For a fixed $x_{1},x_{2}$ we only need to simulate the number of cuts per direction.

Next we characterize the positions in the covariance matrix with equal entries, i.e., conditions on $x_1,\dots,x_4\in[0,1]^p$ such that
\[
\E\big[\mathrm{vol}\left(A_{k}(x_{1},\omega_{1})\cap A_{k}(x_{2},\omega_{2})\right)\big]=\E\big[\mathrm{vol}\left(A_{k}(x_{3},\omega_{1})\cap A_{k}(x_{4},\omega_{2})\right)\big].
\]

\begin{lemma}
\label{lem:equal cov mat entries}Let $\Pi(p)$ be the permutations of $\{1,\ldots,p\}$. We have
\begin{align*}
\mathbb{P} & \left(X_{1}\in A_{k}(x_{1},\omega_{1})\cap A_{k}(x_{2},\omega_{2})\right)=\mathbb{P}\left(X_{1}\in A_{k}(x_{3},\omega_{1})\cap A_{k}(x_{4},\omega_{2})\right)\\
 & \Leftrightarrow\exists\pi\in\Pi(p)\,\forall l\in[p]:\max\big\{t\in\{0,\ldots k\}:\lfloor x_{1}^{(l)}2^{t}\rfloor=\lfloor x_{2}^{(l)}2^{t}\rfloor\big\}\\
 & \qquad\qquad\qquad\qquad\qquad\qquad=\max\big\{t\in\{0,\ldots k\}:\lfloor x_{3}^{(\pi(l))}2^{t}\rfloor=\lfloor x_{4}^{(\pi(l))}2^{t}\rfloor\big\},
\end{align*}
where the probability is taken over the joint distribution of $X_1,\omega_1,\omega_2$.
\end{lemma}
The proof can be found in Subsection \ref{appendix A.3}. 
The lemma implies that the covariance is determined by the order statistic of
\begin{equation*}
\mathfrak{C}_{k}(x_{1},x_{2}):=\left(\max\{t\in\{0,\ldots k\}:\lfloor x_{1}^{(l)}2^{t}\rfloor=\lfloor x_{2}^{(l)}2^{t}\rfloor\}\right)_{l=1}^{p}\in\{0,\ldots,k\}^{p}.
\end{equation*}
Let $\mathfrak{C}_{k}^{(\uparrow)}(x_{1},x_{2})$ denote the order
statistic of the above vector, that is the increasingly ordered version
of $\mathfrak{C}_{k}(x_{1},x_{2})$. In the unordered vector above
each of the entries corresponds to the finest $2^{-t}$-grid on $[0,1]$
in which both $x_{1}^{(l)}$ and $x_{2}^{(l)}$ are in the same cell.
We can interpret the vector $\mathfrak{C}_{k}(x_{1},x_{2})$ as the componentwise closeness of $x_{1}$ and $x_{2}$ in this grid. Larger entries imply that they are closer, and if $x_{1}$ and $x_{2}$ are in the same undividable cell, we have $\mathfrak{C}_{k}(x_{0},x_{0})=(k,\ldots,k)$.
Denoting
\[
\Omega_{S}:=\left\{ (t_{1},\ldots,t_{p})\in\{0,\ldots k\}^p\mid t_{1}\leq t_{2}\leq\ldots\leq t_{p}\right\},
\]
we have $\mathfrak{C}_{k}^{(\uparrow)}(x_{1},x_{2})\in\Omega_{S}$
for all $x_{1},x_{2}\in[0,1]^{p}$. Hence, there are at most $\#\Omega_{S}=\binom{k+p}{p}$ distinct entries in the covariance matrix. 

Another implication of these two lemmas is that the volume of the intersection does not directly depend on $x_{1}$ and $x_{2}$ but
rather their relation to each other. We note that
\begin{align*}
\mathrm{vol} & \left(A_{k}(x_{1},\omega_{1})\cap A_{k}(x_{2},\omega_{2})\right)\\
 & =2^{-\sum_{l=1}^{p}\max\{S_{l}(x_{1},\omega_{1}),S_{l}(x_{2},\omega_{2})\}}\prod_{l=1}^{p}\mathbb{I}\left\{ \min\{S_{l}(x_{1},\omega_{1}),S_{l}(x_{2},\omega_{2})\}\leq\mathfrak{C}_{k}^{(l)}(x_{1},x_{2})\right\}.
\end{align*}
Together with the homogeneity $S_l(x_1,\omega)\overset{d}{=}S_l(x_2,\omega)$ for all $l\in[p]$ and $x_1,x_2\in[0,1]^p$ and independence of $\omega_1$ and $\omega_2$ we have
\begin{align*}
\mathrm{vol}  \left(A_{k}(x_{1},\omega_{1})\cap A_{k}(x_{2},\omega_{2})\right)&\stackrel{d}{=}\mathbb{V}_{\cap}\left(\mathbf{c},S(\omega_{1}),S(\omega_{2})\right)\notag\\
&:=2^{-\sum_{l=1}^{p}\max\{S_{l}(\omega_{1}),S_{l}(\omega_{2})\}}\prod_{l=1}^{p}\mathbb{I}\left\{ \min\{S_{l}(\omega_{1}),S_{l}(\omega_{2})\}\leq\mathbf{c}_{l}\right\}.
\end{align*}
for $\mathbf c=\mathfrak{C}_{k}(x_{1},x_{2})$ and  $S(\omega):=(S_{l}(\omega))_{l=1}^{p}\overset{d}{=}(S_{l}(x_{0},\omega))_{l=1}^{p}$ for a some arbitrary fixed $x_0\in[0,1]^p$.

We can use this to make another observation regarding the covariance.
\begin{lemma}\label{lem:max split per dir}
  If $\cov(B_{k}(f_{x_{1},k}),B_{k}(f_{x_{2},k}))=1$ for some $x_1,x_2\in[0,1]^p$, then $\mathbb{P}\left(S_{l}(\omega_{1})\leq \mathbf c_l\right)=1$ for all $l\in[p]$ with $\mathbf{c}=\mathfrak{C}_{k}(x_{1},x_{2})$.
\end{lemma}
The proof can be found in Subsection \ref{appendix A.3}. Since the axes are interchangeable in the feature space and in the partitioning algorithm, there is a $c^{*}\leq k$ such that 
\[
\mathbb{P}\left(S_{l}(\omega_{1})\leq c^*\right)=1\quad\text{for all } l\in\{1,\ldots,p\}.
\]
In particular, the CPRF estimator is constant on all cells in the $2^{-c^{*}}$grid because these cells
are never split by any partition.


\subsection{Simulation of $\mathbf{S}_{k}$\label{subsec:Estimation cov mat}}


To calculate the quantiles of the distribtuion of $\mathbf{S}_{k}\overset{d}{=}
\sup_{f\in\mathcal{F}_k}|B_kf|
$, we first have to sample trajectories of the centered Gaussian process $B_k$ with covariance structure from (\ref{eq:cov for sim}). Throughout we consider the uniformly distributed features $X\sim\mathcal U([0,1]^p)$.

Before we sample $B_k$, we have to approximate the covariance. Let us fix $x_1,x_2\in[0,1]^p$. We need to simulate a large number of independent vectors $S(x_{1},\omega_{1})\in[p]^k$ and $S(x_{2},\omega_{2})\in[p]^k$ containing the number of splits per direction for two cells $A_{k}(x_{1},\omega_{1})$ and $A_{k}(x_{2},\omega_{2})$, respectively. Lemma~\ref{lem:vol intersec} then allows us to compute the volume of their intersection. An empirical mean over these volumes yields an approximation of the expected volume. For $x_1=x_2$ we obtain $\mathcal{V}_{\cap,k}$ and for $x_1\neq x_2$ we obtain the covariance $\cov(B_{k}(f_{x_{1},k}),B_{k}(f_{x_{2},k}))$.

Lemma \ref{lem:equal cov mat entries} shows that there are only $\binom{k+p}{p}$ distinct covariance values for $x_{1},x_{2}\in[0,1]^{p}$ and instead of calculating the volume of the intersection for all pairs
of $x_{1}$ and $x_{2}$ it suffices to calculate $\mathbb{V}_{\cap}\left(\mathbf{c},S(\omega_{1}),S(\omega_{2})\right)$ for all closeness vectors $\mathbf{c}\in\Omega_{S}$.
Algorithm \ref{alg:Cov Estimation} describes the numerical approximation of the covariances
\[
\cov(\mathbf{c}):=\frac{1}{\mathcal{V}_{\cap,k}}\mathbb{E}\left[\mathbb{V}_{\cap}\left(\mathbf{c},S(\omega_{1}),S(\omega_{2})\right)\right],\qquad\mathbf{c}\in\Omega_{S}.
\]
To order the elements of $\Omega_{S}$, we use some fixed bijection
\[
\tau\colon\Omega_{S}\to\left\{ 1,\ldots,\binom{k+p}{p}\right\} .
\]
\begin{algorithm}[t]
\caption{\label{alg:Cov Estimation}Covariance approximation}

\begin{algorithmic}[1]
\Input Number of partitions $n_{P}$, number of splits $k$, split distribution of $\omega$.
\State $V_{cum}=\{0\}^{\binom{k+p}{p}}$
\For{$j=1,\ldots,n_{P}$}
\State Simulate $S(\omega_{1}) \in \{0,\ldots,k\}^p$ according to the cell split distribution.
\State Simulate $S(\omega_{2}) \in \{0,\ldots,k\}^p$ according to the cell split distribution.
\For{$i=1,\ldots,\binom{k+p}{p}$ }
\State $V_{cum}(i)=V_{cum}(i)+\mathbb{V}_{\cap}(\tau^{-1}(i),S(\omega_{1}),S(\omega_{2}))$
\EndFor
\EndFor
\State $\bar{V}=V_{cum}/n_{P}$
\State $\hat{\mathcal{V}}_{\cap,k}=\bar{V}(\tau(\{k\}^p))$
\State $\hat{C}=\bar{V}/\hat{\mathcal{V}}_{\cap,k}$
\State \Return $\hat{C}$ and $\hat{\mathcal{V}}_{\cap,k}$
\Output Estimators $\hat{C}(\tau(\mathbf{c}))$ of $\cov(\mathbf{c})$ for $\mathbf{c} \in \Omega_S$ and $\hat{\mathcal{V}}_{\cap,k}$ of $\mathcal{V}_{\cap,k}$.
\end{algorithmic}
\end{algorithm}
The algorithm generates approximations $\hat{C}(\tau(\mathbf{c}))$ for $\cov(\mathbf{c})$
and an approximation for $\mathcal{V}_{\cap,k}$ via
$\hat{\mathcal{V}}_{\cap,k}=\bar{V}\left(\tau(\{k\}^{p})\right)$,
due to $\mathfrak{C}_{k}(x_{0},x_{0})=\{k\}^{p}$.

To (approximately) generate the Gaussian process $B_k$ on a grid
\[
G_{\tilde{k}}:=\{a2^{-\tilde{k}}+2^{-\tilde{k}+1}\mid a\in\{0,\ldots,2^{\tilde{k}}-1\}\}^{p}
\]
for some $\tilde k\le k$, i.e. to sample a centered multivariate Gaussian vector with covariance matrix  
\[
\Sigma_{\tilde{k}}:=\frac{1}{\mathcal{V}_{\cap,k}}\left(\mathbb{E}\left[\mathrm{vol}\left(A_{k}(x_{1},\omega_{1})\cap A_{k}(x_{2},\omega_{2})\right)\right]\right)_{x_{1},x_{2}\in G_{\tilde{k}}},
\]
we approximate the latter with
\[
\hat{\Sigma}_{\tilde{k}}:=\left(\hat{C}\left(\tau\left(\mathfrak{C}_{k}^{(\uparrow)}(x_{1},x_{2})\right)\right)\right)_{x_{1},x_{2}\in G_{\tilde{k}}}
\]
and apply a Cholesky decomposition $\hat\Sigma_{\tilde k}=\hat L\hat L^{T}$ to obtain $\hat LZ\sim\mathcal N(0,\hat\Sigma_{\tilde k})$ for some standard $\tilde kp$-dimensional normal random vector $Z$.

It may be the case that the estimated covariance $\hat{\Sigma}_{\tilde{k}}$
is not positive semidefinite, which is necessary for the Cholesky decomposition. Therefore, we slightly adjust the estimated covariance matrix by adding a small constant times  the identity matrix to the approximated covariance matrix before computing the Cholesky decomposition. If this is not sufficient to get a positive semidefinite matrix, we use a reconstruction of the covariance matrix using only the eigenvalues that are larger than a small constant:
Let $Q_{\tilde{k}}$ denote a matrix with the eigenvectors of $\hat{\Sigma}_{\tilde{k}}$
as its columns and let $(\lambda_{i})_{i=1}^{2^{p\tilde{k}}}$ be
the corresponding eigenvalues. If $\Lambda_{\epsilon}$ denotes the diagonal
matrix with entries $(\max\{\lambda_{i},\epsilon\})_{i=1}^{2^{p\tilde{k}}}$, we may replace $\hat{\Sigma}_{\tilde{k}}$ by
\[
\hat{\Sigma}_{\tilde{k},2}:=Q_{\tilde{k}}\Lambda_{\epsilon}Q_{\tilde{k}}^{T}.
\]

It remains to discuss how we choose $\tilde{k}$ and thereby the grid for the simulation of the Gaussian process. A good choice of $\tilde k$ should reflect the trade-off between a sufficiently fine grid, to ensure a good approximation of the supremum 
$\sup_{f\in\mathcal{F}_k}|B_kf|$,
and numerical feasibility since the dimension of the covariance matrix grows like $\tilde kp$. 
A natural choice is the set $\mathcal{X}_{k}$ that contains one point in every undividable cell. We know that
a $2^{k}$-grid would always contain one point in every undividable cell, which is however computationally very expensive. As a way out, Lemma~\ref{lem:max split per dir}
can be used to determine which grid is fine enough. Since we estimate
the covariance for all $\mathbf{c}\in\Omega_{S}$ we can check which
estimated covariance entries are equal to one. If there is a $c^{*}\in\{0,\ldots,k\}$
with
\[
\mathbb{P}\left(S_{l}(\omega_{1})\leq c^{*}\right)=1\;\forall l\in\{1,\ldots,p\},
\]
the grid with $\tilde{k}=c^{*}$ is fine enough because its grid cells
are undividable.
This implies that all points in these cells correspond to the same
function in $\mathcal{F}_{k}$. Thus, it suffices to simulate the
Gaussian process on this grid to estimate the distribution of the
supremum. An equivalent perspective is, that the estimator and the
stochastic error are constant on these cells and thus it is sufficient
to check one point in each cell to calculate the supremum of the stochastic
error.

In practice, we can choose $\tilde{k}$ such that the covariance of
the Gaussian process at $x_{1}$ and $x_{2}$ with $\mathfrak{C}_{k}(x_{1},x_{2})=\{\tilde{k}\}^{p}$,
which is
\[
\frac{1}{\mathcal{V}_{\cap,k}}\mathbb{E}\left[\mathbb{V}_{\cap}\left(\{\tilde{k}\}^{p},S(\omega_{1}),S(\omega_{2})\right)\right],
\]
is reasonably close to one. In this case the probability $\mathbb{P}(S_{l}(\omega_{1})\leq\tilde{k})$
should also be close to one for every $l$. With this method we can
determine the finest grid spacing that is necessary. If this grid
is computationally feasible we choose $\tilde{k}$ accordingly. Otherwise,
we can use the method to get an idea how good the best computationally
feasible grid is.

\subsection{Proofs}\label{appendix A.3}

\begin{proof}[Proof of Lemma \ref{lem:vol intersec}]

We start with the case $A_{k}(x_{1},\omega_{1})\cap A_{k}(x_{2},\omega_{2})\neq\emptyset$.
Thus, there exists $x_{0}\in A_{k}(x_{1},\omega_{1})\cap A_{k}(x_{2},\omega_{2})$.
Equation (\ref{eq:vol intersec darst}) yields
\begin{align*}
\mathrm{vol}\left(A_{k}(x_{1},\omega_{1})\cap A_{k}(x_{2},\omega_{2})\right) & =\mathrm{vol}\left(A_{k}(x_{0},\omega_{1})\cap A_{k}(x_{0},\omega_{2})\right)\\
 & =2^{-\sum_{l=1}^{p}\max\{S_{l}(x_{0},\omega_{1}),S_{l}(x_{0},\omega_{2})\}}\\
 & =2^{-\sum_{l=1}^{p}\max\{S_{l}(x_{1},\omega_{1}),S_{l}(x_{2},\omega_{2})\}}.
\end{align*}
It remains to prove that
\begin{align}
\mathbb{I} & \left\{ A_{k}(x_{1},\omega_{1})\cap A_{k}(x_{2},\omega_{2})\neq\emptyset\right\} \nonumber \\
 & =\prod_{l=1}^{p}\mathbb{I}\left\{ \lfloor x_{1}^{(l)}2^{\min\{S_{l}(x_{1},\omega_{1}),S_{l}(x_{2},\omega_{2})\}}\rfloor=\lfloor x_{2}^{(l)}2^{\min\{S_{l}(x_{1},\omega_{1}),S_{l}(x_{2},\omega_{2})\}}\rfloor\right\} .\label{eq:lim obj lem 1 claim}
\end{align}
The projection of the cell $A_{k}(x_{0},\omega)$ on the $l$-th coordinate is given by
\begin{align}
A_{k}^{(l)}(x_{0},\omega):=\{\xi\in[0,1]:\exists x\in A_{k}(x_{0},\omega),x^{(l)}=\xi\}\notag\\
=2^{-S_{l}(x_{0},\omega)}\big(\lfloor x_{0}^{(l)}2^{S_{l}(x_{0},\omega)}\rfloor,\lfloor x_{0}^{(l)}2^{S_{l}(x_{0},\omega)}\rfloor+1\big].\label{eq:cell interval}
\end{align}
We have that
\begin{equation}
\mathbb{I}\left\{ A_{k}(x_{1},\omega_{1})\cap A_{k}(x_{2},\omega_{2})\neq\emptyset\right\} =\prod_{l=1}^{p}\mathbb{I}\left\{ A_{k}^{(l)}(x_{1},\omega_{1})\cap A_{k}^{(l)}(x_{2},\omega_{2})\neq\emptyset\right\} .\label{eq:lim obj lem 1 proof eq 1}
\end{equation}
In view of \eqref{eq:cell interval} the dyadic construction of the interval boundaries yields
\begin{align*}
A_{k}^{(l)} & (x_{1},\omega_{1})\cap A_{k}^{(l)}(x_{2},\omega_{2})\neq\emptyset\\
 & \Leftrightarrow\left(A_{k}^{(l)}(x_{1},\omega_{1})\subset A_{k}^{(l)}(x_{2},\omega_{2})\right)\text{ or }\left(A_{k}^{(l)}(x_{1},\omega_{1})\supset A_{k}^{(l)}(x_{2},\omega_{2})\right).
\end{align*}
Therefore, both $x_{1}^{(l)}$ and $x_{2}^{(l)}$ are contained in the larger of both cells. We conclude from \eqref{eq:cell interval}
\begin{align*}
A_{k}^{(l)} & (x_{1},\omega_{1})\cup A_{k}^{(l)}(x_{2},\omega_{2})\\
 & =2^{-\min\{S_{l}(x_{1},\omega_{1}),S_{l}(x_{2},\omega_{2})\}}\big[\lfloor x_{1}^{(l)}2^{\min\{S_{l}(x_{1},\omega_{1}),S_{l}(x_{2},\omega_{2})\}}\rfloor,\lfloor x_{1}^{(l)}2^{\min\{S_{l}(x_{1},\omega_{1}),S_{l}(x_{2},\omega_{2})\}}\rfloor+1\big).
\end{align*}
Together this implies
\begin{align*}
A_{k}^{(l)} & (x_{1},\omega_{1})\cap A_{k}^{(l)}(x_{2},\omega_{2})\neq\emptyset\\
 & \Leftrightarrow\left(A_{k}^{(l)}(x_{1},\omega_{1})\subset A_{k}^{(l)}(x_{2},\omega_{2})\right)\vee\left(A_{k}^{(l)}(x_{1},\omega_{1})\supset A_{k}^{(l)}(x_{2},\omega_{2})\right)\\
 & \Leftrightarrow x_{1}^{(l)},x_{2}^{(l)}\in2^{-\min\{S_{l}(x_{1},\omega_{1}),S_{l}(x_{2},\omega_{2})\}}\\
 & \qquad\times\big[\lfloor x_{1}^{(l)}2^{\min\{S_{l}(x_{1},\omega_{1}),S_{l}(x_{2},\omega_{2})\}}\rfloor,\lfloor x_{1}^{(l)}2^{\min\{S_{l}(x_{1},\omega_{1}),S_{l}(x_{2},\omega_{2})\}}\rfloor+1\big)\\
 & \Leftrightarrow\lfloor x_{1}^{(l)}2^{\min\{S_{l}(x_{1},\omega_{1}),S_{l}(x_{2},\omega_{2})\}}\rfloor=\lfloor x_{2}^{(l)}2^{\min\{S_{l}(x_{1},\omega_{1}),S_{l}(x_{2},\omega_{2})\}}\rfloor.
\end{align*}
Plugging this into (\ref{eq:lim obj lem 1 proof eq 1}) we obtain
(\ref{eq:lim obj lem 1 claim}).
\end{proof}

\begin{proof}[Proof of Lemma \ref{lem:equal cov mat entries}]

We can assume that
\begin{align*}
\big(\max & \{t\in\{0,\ldots k\}:\lfloor x_{1}^{(l)}2^{t}\rfloor=\lfloor x_{2}^{(l)}2^{t}\rfloor\}\big)_{l=1}^{p}\\
 & =\big(\max\{t\in\{0,\ldots k\}:\lfloor x_{3}^{(l)}2^{t}\rfloor=\lfloor x_{4}^{(l)}2^{t}\rfloor\}\big)_{l=1}^{p}
\end{align*}
without loss of generality because the CPRF is symmetric. For $s_{1},s_{2}\in\{0,\ldots,k\}^{p}$
we denote $s_{1}=(s_{1}(l))_{l=1}^{p}$ and $s_{2}$ analogously.
For $S(x,\omega)=(S_{l}(x,\omega))_{l=1}^{p}$ we obtain
\begin{align*}
\mathbb{P} & \left(X_{1}\in A_{k}(x_{1},\omega_{1})\cap A_{k}(x_{2},\omega_{2})\right)\\
 & =\mathbb{E}\left[\mathrm{vol}\left(A_{k}(x_{1},\omega_{1})\cap A_{k}(x_{2},\omega_{2})\right)\right]\\
 & =\mathbb{E}\Big[2^{-\sum_{l=1}^{p}\max\{S_{l}(x_{1},\omega_{1}),S_{l}(x_{2},\omega_{2})\}}\\
 & \qquad\times\prod_{l=1}^{p}\mathbb{I}\left\{ \lfloor x_{1}^{(l)}2^{\min\{S_{l}(x_{1},\omega_{1}),S_{l}(x_{2},\omega_{2})\}}\rfloor=\lfloor x_{2}^{(l)}2^{\min\{S_{l}(x_{1},\omega_{1}),S_{l}(x_{2},\omega_{2})\}}\rfloor\right\} \Big]\\
 & =\sum_{s_{1},s_{2}\in\{0,\ldots,k\}^{p}}\mathbb{P}\left(S(x_{1},\omega_{1})=s_{1},S(x_{2},\omega_{2})=s_{2}\right)\\
 & \qquad\times2^{-\sum_{l=1}^{p}\max\{s_{1}(l),s_{2}(l)\}}\prod_{l=1}^{p}\mathbb{I}\left\{ \lfloor x_{1}^{(l)}2^{\min\{s_{1}(l),s_{2}(l)\}}\rfloor=\lfloor x_{2}^{(l)}2^{\min\{s_{1}(l),s_{2}(l)\}}\rfloor\right\} \\
 & =\sum_{s_{1},s_{2}\in\{0,\ldots,k\}^{p}}\mathbb{P}\left(S(x_{3},\omega_{1})=s_{1},S(x_{4},\omega_{2})=s_{2}\right)\\
 & \qquad\times2^{-\sum_{l=1}^{p}\max\{s_{1}(l),s_{2}(l)\}}\prod_{l=1}^{p}\mathbb{I}\left\{ \lfloor x_{3}^{(l)}2^{\min\{s_{1}(l),s_{2}(l)\}}\rfloor=\lfloor x_{4}^{(l)}2^{\min\{s_{1}(l),s_{2}(l)\}}\rfloor\right\} \\
 & =\mathbb{P}\left(X_{1}\in A_{k}(x_{3},\omega_{1})\cap A_{k}(x_{4},\omega_{2})\right)
\end{align*}
because for all $l\in[p]$ and $t\in\{0,\ldots k\}$ we know that
\[
\mathbb{I}\left\{ \lfloor x_{1}^{(l)}2^{t}\rfloor=\lfloor x_{2}^{(l)}2^{t}\rfloor\right\} =\mathbb{I}\left\{ \lfloor x_{1}^{(l)}2^{t}\rfloor=\lfloor x_{2}^{(l)}2^{t}\rfloor\right\}
\]
and in particular,
\begin{align*}
\mathbb{I} & \left\{ \lfloor x_{1}^{(l)}2^{\min\{s_{1}(l),s_{2}(l)\}}\rfloor=\lfloor x_{2}^{(l)}2^{\min\{s_{1}(l),s_{2}(l)\}}\rfloor\right\} \\
 & =\mathbb{I}\left\{ \lfloor x_{1}^{(l)}2^{\min\{s_{1}(l),s_{2}(l)\}}\rfloor=\lfloor x_{2}^{(l)}2^{\min\{s_{1}(l),s_{2}(l)\}}\rfloor\right\} .\qedhere
\end{align*}
\end{proof}

\begin{proof}[Proof of Lemma~\ref{lem:max split per dir}]
If the covariance from (\ref{eq:cov for sim}) is equal to one we have 
\[
\mathbb{E}\left[\mathrm{vol}\left(A_{k}(x_{1},\omega_{1})\cap A_{k}(x_{2},\omega_{2})\right)\right]=\mathcal{V}_{\cap,k}=\mathbb{E}\left[\mathrm{vol}\left(A_{k}(x_{0},\omega_{1})\cap A_{k}(x_{0},\omega_{2})\right)\right].
\]
Let the ``relation'' between $\mathfrak{C}_{k}(x_{1},x_{2})=\mathbf{c}$
be fixed. The above implies that
\begin{align*}
\mathbb{E} & \left[2^{-\sum_{l=1}^{p}\max\{S_{l}(\omega_{1}),S_{l}(\omega_{2})\}}\prod_{l=1}^{p}\mathbb{I}\left\{ \min\{S_{l}(\omega_{1}),S_{l}(\omega_{2})\}\leq\mathbf{c}_{l}\right\} \right]\\
 & =\mathbb{E}\left[\mathbb{V}_{\cap}\left(\mathbf{c},S(\omega_{1}),S(\omega_{2})\right)\right]
  =\mathbb{E}\left[\mathrm{vol}\left(A_{k}(x_{1},\omega_{1})\cap A_{k}(x_{2},\omega_{2})\right)\right]\\
 & =\mathbb{E}\left[\mathrm{vol}\left(A_{k}(x_{0},\omega_{1})\cap A_{k}(x_{0},\omega_{2})\right)\right]
 =\mathbb{E}\left[\mathbb{V}_{\cap}\left(\{k\}^{p},S(\omega_{1}),S(\omega_{2})\right)\right]\\
 & =\mathbb{E}\left[2^{-\sum_{l=1}^{p}\max\{S_{l}(\omega_{1}),S_{l}(\omega_{2})\}}\prod_{l=1}^{p}\mathbb{I}\left\{ \min\{S_{l}(\omega_{1}),S_{l}(\omega_{2})\}\leq k\right\} \right]\\
 & =\mathbb{E}\left[2^{-\sum_{l=1}^{p}\max\{S_{l}(\omega_{1}),S_{l}(\omega_{2})\}}\right].
\end{align*}
Therefore,
\[
\mathbb{P}\left(\bigcap_{l=1}^{p}\left\{ \min\{S_{l}(\omega_{1}),S_{l}(\omega_{2})\}\leq\mathbf{c}_{l}\right\} \right)=1
\]
and hence
\[
\mathbb{P}\left(\min\{S_{l}(\omega_{1}),S_{l}(\omega_{2})\}\leq\mathbf{c}_{l}\right)=1\;\forall l\in\{1,\ldots,p\}.
\]
Since $\omega_{1}$ and $\omega_{2}$ are independent and identically distributed we conclude
\[
\mathbb{P}\left(S_{l}(\omega_{1})\leq\mathbf{c}_{l}\right)=1\;\forall l\in\{1,\ldots,p\}.\qedhere
\]
\end{proof}

\section{Proofs}\label{sec:Proof-strategy chap CBs}

\subsection{Proof strategy for Theorem~\ref{thm:CB main01}}\label{sec:ProofSketchMain}

We use a decomposition of the estimator that also leads to a decomposition
of the error. We denote
\begin{align}
U_{n,r_{n},\omega}^{(m)}(x_{0}) & :=
\frac{1}{\binom{n}{r_n}}\sum_{I\in \setbinom{n}{r_n}} \sum_{i\in I}m(X_i) \frac{ \mathbb{I}\{X_{i}\in A_{n}(x_0,\omega_I)\}}{\# A_{n}(x_0,\omega_I)}\label{eq:U^(m) CPRF}\\
\text{and}\qquad U_{n,r_{n},\omega}^{(\varepsilon)}(x_{0}) & :=
\frac{1}{\binom{n}{r_n}}\sum_{I\in \setbinom{n}{r_n}} \sum_{i\in I}\varepsilon_i \frac{ \mathbb{I}\{X_{i}\in A_{n}(x_0,\omega_I)\}}{\# A_{n}(x_0,\omega_I)}
,\label{eq:U^(eps) CPRF}
\end{align}
such that for the random forest defined in (\ref{U_n-rn-omega}) we have the decomposition
\begin{equation*}
U_{n,r_{n},\omega}^{(\mathrm{RF})}(x_{0})-m(x_{0})=U_{n,r_{n},\omega}^{(\varepsilon)}(x_{0})+U_{n,r_{n},\omega}^{(m)}(x_{0})-m(x_{0}).
\end{equation*}
We note that $U_{n,r_{n},\omega}^{(m)}$ and $U_{n,r_{n},\omega}^{(\varepsilon)}$
are both generalized complete U-statistics. For any set $I\subset[n]$,
with $\vert I\vert=r_{n}$ and $j\in I$ let us further denote
\[
W_{j,k}(x_{0},\omega,I):=\frac{\mathbb{I}\{X_{j}\in A_{k}(x_{0},\omega)\}}{\sum_{i\in I}\mathbb{I}\{X_{i}\in A_{k}(x_{0},\omega)\}}.
\]
To simplify the notation we use
\begin{equation}
W_{j,k}(x_{0},I):=W_{j,k}(x_{0},\omega_{I},I)\qquad\text{and }\qquad W_{j,k}(x_{0},\omega):=W_{j,k}(x_{0},\omega,[r_{n}]).\label{eq:notation weights abbrev}
\end{equation}
The kernels of the two U-statistics in (\ref{eq:U^(m) CPRF}) and
(\ref{eq:U^(eps) CPRF}) are
\begin{align*}
h_{n}^{(m)}(x_{0},(X_{j})_{j=1}^{r_{n}},\omega) & =\sum_{j=1}^{r_{n}}m(X_{j})W_{j,k}(x_{0},\omega)\qquad\text{and }\nonumber \\
h_{n}^{(\varepsilon)}(x_{0},(X_{j},\varepsilon_{j})_{j=1}^{r_{n}},\omega) & =\sum_{j=1}^{r_{n}}\varepsilon_{j}W_{j,k}(x_{0},\omega).
\end{align*}
The kernel $h_{n}^{(\varepsilon)}$ and thus also $U_{n,r_{n},\omega}^{(\varepsilon)}$
have expectation zero due to the independence of the $\varepsilon_{j}$
from $X_{j}$ and $\omega$. We note that the $W_{j,k}(x_{0},\omega,I)$
are the weights of the observations in the kernel applied to the subsample
$I$. If there is at least one observation in $A_{k}(x_{0},\omega)$
the sum of the weights is equal to one. Otherwise ou definition 
yields that their sum is zero.

We call $U_{n,r_{n},\omega}^{(m)}-m$ the {\it approximation error} and
$U_{n,r_{n},N,\omega}^{(\varepsilon)}$ the {\it stochastic error}. We need
to handle both these terms uniformly in $x_{0}$. The error can be
bounded by
\begin{align*}
\sup_{x_{0}\in[0,1]^{p}}\vert U_{n,r_{n},\omega}^{(\mathrm{RF})}(x_{0})-m(x_{0})\vert & =\Vert U_{n,r_{n},\omega}^{(\mathrm{RF})}-m\Vert_{\infty} \leq\Vert U_{n,r_{n},\omega}^{(m)}-m\Vert_{\infty}+\Vert U_{n,r_{n},\omega}^{(\varepsilon)}\Vert_{\infty}.
\end{align*}
For $K_{k}(x_{0},x)$ from (\ref{eq:RF Hajek kernel}) let us denote
\begin{equation}
\hat{U}_{n,r_{n},\omega}^{(\varepsilon)}(x_{0})=\frac{1}{n}\sum_{j=1}^{n}\varepsilon_{j}K_{k}(x_{0},X_{j}).\label{eq:U hat eps}
\end{equation}
This definition is based on the H\'ajek projection of the U-statistic
$U_{n,r_{n},\omega}^{(\varepsilon)}(x_{0})$. The only difference
to (\ref{eq:U hat eps}) is a term that converges to one uniformly
under reasonable assumptions on $k$ and $r_{n}$. Under the assumptions
in Theorem \ref{thm:CB main01}, the term $\hat{U}_{n,r_{n},\omega}^{(\varepsilon)}$
is the asymptotically leading term of $U_{n,r_{n},\omega}^{(\varepsilon)}$.
We call $U_{n,r_{n},\omega}^{(\varepsilon)}-\hat{U}_{n,r_{n},\omega}^{(\varepsilon)}$
the {\it projection error} due to its connection to the H\'ajek projection.
Later we will decompose the projection error into two terms, which
we will call remainder terms.

The idea for the leading term is that $\hat{U}_{n,r_{n},\omega}^{(\varepsilon)}$
is close to a Gaussian process in $x_{0}$ if $n$ is large enough.
We can express $\hat{U}_{n,r_{n},\omega}^{(\varepsilon)}(x_{0})$
as an empirical process applied to the observations $(X_{j},\varepsilon_{j})_{j=1}^{n}$
and the function class $\mathcal{F}_{k}$ from (\ref{eq:F_k}):
\begin{align*}
\hat{U}_{n,r_{n},\omega}^{(\varepsilon)}(x_{0}) & =\frac{1}{n}\sum_{j=1}^{n}\varepsilon_{j}K_{k}(x_{0},X_{j})
=\sqrt{\frac{\sigma^{2}\Psi_{k}(x_{0})}{n}}\mathbb{G}_{n}f_{x_{0},k},
\end{align*}
where $\mathbb{G}_{n}=n^{1/2}(P_n-P)$ denotes the empirical process based on the sample $(X_i,Y_i)$, $i=1,\dots,n$.
We use Corollary 2.2 by \citet{Chernozhukov2014} to approximate
\[
\sup_{f\in\mathcal{F}_{k}}\vert\mathbb{G}_{n}f\vert=\sqrt{\frac{n}{\sigma^{2}}}\sup_{x_{0}\in[0,1]^{p}}\vert\Psi_{k}^{-1/2}(x_{0})\hat{U}_{n,r_{n},\omega}^{(\varepsilon)}(x_{0})\vert
\]
by a sequence of random variables $\mathbf{S}_{k}$ with $\mathbf{S}_{k}\overset{d}{=}\sup_{f\in\mathcal{F}_{k}}\vert B_{k}f\vert$.
Therein $B_{k}$ is a sequence of centered Gaussian processes indexed
by $\mathcal{F}_{k}$ that has the covariance function
\begin{align*}
\cov(B_{k}(f_{x_{1},k}),B_{k}(f_{x_{2},k})) & =\mathbb{E}[f_{x_{1},k}(X_{1},\varepsilon_{1})f_{x_{2},k}(X_{1},\varepsilon_{1})]\\
 & =\Psi_{k}^{-1/2}(x_{1})\Psi_{k}^{-1/2}(x_{2})\sigma^{-2}\mathbb{E}[\varepsilon_{1}K_{k}(x_{1},X_{1})\varepsilon_{1}K_{k}(x_{2},X_{1})]\\
 & =\Psi_{k}^{-1/2}(x_{1})\Psi_{k}^{-1/2}(x_{2})\mathbb{E}[K_{k}(x_{1},X_{1})K_{k}(x_{2},X_{1})].
\end{align*}
In particular, this implies $\var(B_{k}(f))=1$. The application of
Corollary 2.2 by \citet{Chernozhukov2014} leads to the following
theorem, whose proof can be found in Section \ref{subsec:Proof sup approx appli}.
\begin{theorem}
\label{thm:supUhat-approx} Let $B_{k}(f)$, $f\in\mathcal{F}_{k}$,
be a sequence of centered Gaussian processes with covariance
\[
\cov(B_{k}(f_{x_{1},k}),B_{k}(f_{x_{2},k}))=\Psi_{k}^{-1/2}(x_{1})\Psi_{k}^{-1/2}(x_{2})\mathbb{E}\left[K_{k}(x_{1},X_{1})K_{k}(x_{2},X_{1})\right].
\]
If $\nu\in[4,\infty)$ and $\mathbb{E}[\vert\varepsilon_{1}\vert^{\nu}]<\infty$
there exists a sequence of random variables $\mathbf{S}_{k}\overset{\mathcal{D}}{=}\sup_{x_{0}\in[0,1]^{p}}\vert B_{k}f_{x_{0},k}\vert$
such that
\[
\bigg\vert\sqrt{\frac{n}{\sigma^{2}}}\sup_{x_{0}\in[0,1]^{p}}\vert\Psi_{k}^{-1/2}(x_{0})\hat{U}_{n,r_{n},\omega}^{(\varepsilon)}(x_{0})\vert-\mathbf{S}_{k}\bigg\vert=\mathcal{O}_{\mathbb{P}}\left(\frac{(\log n)^{3/2}}{\mathcal{V}_{\cap,k}^{1/2}n^{1/2-1/\nu}}+\frac{(\log n)^{5/4}}{\mathcal{V}_{\cap,k}^{1/4}n^{1/4}}+\frac{\log n}{\mathcal{V}_{\cap,k}^{1/6}n^{1/6}}\right).
\]
\end{theorem}

Using this direct approximation of the supremum it will not be necessary
to approximate the whole empirical process uniformly. This is an important
difference to previous results in the literature like those by \citet{Johnston1982},
\citet{Claeskens2003} or \citet{Chao2017}. With the above theorem
for the leading term of the asymptotic distribution of
\[
\sqrt{\frac{n}{\sigma^{2}}}\Vert\Psi_{k}^{-1/2}U_{n,r_{n},\omega}^{(\varepsilon)}\Vert_{\infty},
\]
it remains to handle the projection error uniformly in $x_{0}$.The observations concerning
the finite size of $\mathcal{F}_{k}$ for analogous terms imply that
the supremum over $x_{0}\in[0,1]^{p}$ of the projection error is  a maximum over $x_{0}\in\mathcal{X}_{k}$. In particular,
the argument implies that for any function $g$ and any realizations
$(w_{j})_{j\in J}$ of random variables $(\omega_{j})_{j\in J}$ it
holds that
\[
\sup_{x_{0}\in[0,1]^{p}}g\big((A_{k}(x_{0},w_{j}))_{j\in J},\eta\big)=\max_{x_{0}\in\mathcal{X}_{k}}g\big((A_{k}(x_{0},w_{j}))_{j\in J},\eta\big),
\]
where $\eta$ can be any arbitrary argument that does not depend on
$x_{0}$. The remainder terms from the projection error are of the
above form. This will allow for the utilization of union bounds in conjunction
with bounds of finite moments of the remainder terms, thereby enabling
their uniform handling.

For the approximation error $U_{n,r_{n},\omega}^{(m)}-m$ we cannot
directly use this argument because $m$ is continuous in $x_{0}$
and thus the supremum is not reduced to a maximum. But after using
the H\"older continuity the bound
\[
\vert m(X_{1})-m(x_{0})\vert\mathbb{I}\{X_{1}\in A_{k}(x_{0},\omega)\}\leq C_{H}\mathfrak{d}(A_{k}(x_{0},\omega))^{\alpha}
\]
only depends on $x_{0}$ via the cell $A_{k}(x_{0},\omega)$ and allows
us to use the same argument. To bound the uniform approximation error
we will use a bound for the expected diameter of the cells. The averaging
over many trees in the random forests leads to
\[
\frac{1}{\binom{n}{r_{n}}}\sum_{I\in \setbinom{n}{r_n}}\mathfrak{d}(A_{k}(x_{0},\omega_{I}))^{\alpha}\approx\mathbb{E}[\mathfrak{d}(A_{k}(x_{0},\omega))^{\alpha}].
\]
Hence, the expectation of the supremum will be reduced to supremum
of the expectation. These expectations are all the same and we will
be able to use a non uniform bound of the diameter.

\subsection{Proofs for Section \ref{sec:Confidence-Bands}\label{sec:Proofs chap CBs}}

\subsubsection{Proof of Theorem \ref{thm:CB main01}\label{subsec:Proof Thm main}}

The following proposition is essential for the proof of Theorem \ref{thm:CB main01}.
Its proof is postponed to Section \ref{subsec:Proof-of-Proposition main CB}.
\begin{prop}
\label{prop:asymp dist} Under the assumptions of Theorem \ref{thm:CB main01}
let $B_{k}$ be the sequence of Gaussian processes defined in equation
(\ref{eq:Gaussian process}). For any $\nu\geq4$ with $\mathbb{E}[\vert\varepsilon_{1}\vert^{\nu}]<\infty$
there exists a sequence of random variables $\mathbf{S}_{k}\overset{\mathcal{D}}{=}\sup_{x_{0}\in[0,1]^{p}}\vert B_{k}f_{x_{0},k}\vert$
with

\begin{align*}
\vert\sqrt{n} & \Vert\Psi_{k}^{-1/2}(U_{n,r_{n},\omega}^{(\mathrm{RF})}-m)\Vert_{\infty}-\sigma\mathbf{S}_{k}\vert\\
 & =\mathcal{O}_{\mathbb{P}}\left(\frac{(\log n)^{3/2}}{\mathcal{V}_{\cap,k}^{1/2}n^{1/2-1/\nu}}+\frac{(\log n)^{5/4}}{\mathcal{V}_{\cap,k}^{1/4}n^{1/4}}+\frac{\log n}{\mathcal{V}_{\cap,k}^{1/6}n^{1/6}}\right)+o_{\mathbb{P}}\left((\log n)^{-1}\right).
\end{align*}
\end{prop}

\begin{proof}[Proof of Theorem \ref{thm:CB main01}]
 We lower bound the coverage probability by
\begin{align}
\mathbb{P} & \left(m(x)\in\mathcal{C}_{n}(x),\,\forall x\in[0,1]^{p}\right)\nonumber \\
 & =\mathbb{P}\left(\Vert\Psi_{k}^{-1/2}(U_{n,r_{n},\omega}^{(\mathrm{RF})}-m)\Vert_{\infty}\leq\frac{\hat{\sigma}c_{k}(\beta)}{\sqrt{n}}\right)\nonumber \\
 & =1-\mathbb{P}\left(\Vert\Psi_{k}^{-1/2}(U_{n,r_{n},\omega}^{(\mathrm{RF})}-m)\Vert_{\infty}>\frac{\hat{\sigma}c_{k}(\beta)}{\sqrt{n}}\right)\nonumber \\
 & =1-\mathbb{P}\left(\Vert\Psi_{k}^{-1/2}(U_{n,r_{n},\omega}^{(\mathrm{RF})}-m)\Vert_{\infty}>\sigma\frac{\hat{\sigma}}{\sigma}\frac{c_{k}(\beta)}{\sqrt{n}},\vert\hat{\sigma}/\sigma-1\vert\leq(\log n)^{-2}\right)\nonumber \\
 & \qquad\quad-\mathbb{P}\left(\Vert\Psi_{k}^{-1/2}(U_{n,r_{n},\omega}^{(\mathrm{RF})}-m)\Vert_{\infty}>\sigma\frac{\hat{\sigma}}{\sigma}\frac{c_{k}(\beta)}{\sqrt{n}},\vert\hat{\sigma}/\sigma-1\vert>(\log n)^{-2}\right)\nonumber \\
 & \geq1-\mathbb{P}\left(\Vert\Psi_{k}^{-1/2}(U_{n,r_{n},\omega}^{(\mathrm{RF})}-m)\Vert_{\infty}>\sigma(1-(\log n)^{-2})\frac{c_{k}(\beta)}{\sqrt{n}},\vert\hat{\sigma}/\sigma-1\vert\leq(\log n)^{-2}\right)\nonumber \\
 & \qquad\quad-\mathbb{P}\left(\Vert\Psi_{k}^{-1/2}(U_{n,r_{n},\omega}^{(\mathrm{RF})}-m)\Vert_{\infty}>\sigma\frac{\hat{\sigma}}{\sigma}\frac{c_{k}(\beta)}{\sqrt{n}},\vert\hat{\sigma}/\sigma-1\vert>(\log n)^{-2}\right)\nonumber \\
 & \geq1-\mathbb{P}\left(\Vert\Psi_{k}^{-1/2}(U_{n,r_{n},\omega}^{(\mathrm{RF})}-m)\Vert_{\infty}>\sigma(1-(\log n)^{-2})\frac{c_{k}(\beta)}{\sqrt{n}}\right)\nonumber \\
 & \qquad\quad-\mathbb{P}\left(\vert\hat{\sigma}/\sigma-1\vert>(\log n)^{-2}\right)\nonumber \\
 & \geq1-\mathbb{P}\left(\big\vert\sqrt{n}\Vert\Psi_{k}^{-1/2}(U_{n,r_{n},\omega}^{(\mathrm{RF})}-m)\Vert_{\infty}-\sigma\mathbf{S}_{k}\big\vert+\sigma\mathbf{S}_{k}>\sigma(1-(\log n)^{-2})c_{k}(\beta)\right)\nonumber \\
 & \qquad\quad-\mathbb{P}\left(\vert\hat{\sigma}-\sigma\vert>\sigma(\log n)^{-2}\right)\nonumber \\
 & \geq1-\mathbb{P}\left(\sigma\mathbf{S}_{k}>\sigma(1-(\log n)^{-2})c_{k}(\beta)-(\log n)^{-1}\right)-\mathbb{P}\left(\vert\hat{\sigma}-\sigma\vert>\sigma(\log n)^{-2}\right)\nonumber \\
 & \qquad\quad-\mathbb{P}\left(\big\vert\sqrt{n}\Vert\Psi_{k}^{-1/2}(U_{n,r_{n},\omega}^{(\mathrm{RF})}-m)\Vert_{\infty}-\sigma\mathbf{S}_{k}\big\vert>(\log n)^{-1}\right).\label{eq:main proof first bound}
\end{align}
Assumption (\ref{eq:assum moments}) implies $\mathbb{E}[\vert\varepsilon_{1}\vert^{2\nu}]<\infty$
for $\nu\geq2$, therefore Proposition \ref{prop:asymp dist} yields
\begin{align*}
\vert\sqrt{n} & \Vert\Psi_{k}^{-1/2}(U_{n,r_{n},\omega}^{(\mathrm{RF})}-m)\Vert_{\infty}-\sigma\mathbf{S}_{k}\vert\\
 & =\mathcal{O}_{\mathbb{P}}\left(\frac{(\log n)^{3/2}}{\mathcal{V}_{\cap,k}^{1/2}n^{1/2-1/(2\nu)}}+\frac{(\log n)^{5/4}}{\mathcal{V}_{\cap,k}^{1/4}n^{1/4}}+\frac{\log n}{\mathcal{V}_{\cap,k}^{1/6}n^{1/6}}\right)+o_{\mathbb{P}}\left((\log n)^{-1}\right).
\end{align*}
We argue why these terms are $o_{\mathbb{P}}((\log n)^{-1})$. With
(\ref{eq:assum Gauss sup}) we have
\[
\frac{(\log n)^{5/2}}{\mathcal{V}_{\cap,k}^{1/2}n^{1/2-1/(2\nu)}}=\Big(\frac{(\log n)^{5}}{\mathcal{V}_{\cap,k}n}n^{1/\nu}\Big)^{1/2}\to0.
\]
For the second 
and third term it works similarly. Hence
\begin{equation}
\mathbb{P}\big(\vert\sqrt{n}\Vert\Psi_{k}^{-1/2}(U_{n,r_{n},\omega}^{(\mathrm{RF})}-m)\Vert_{\infty}-\sigma\mathbf{S}_{k}\vert>(\log n)^{-1}\big)\to0.\label{eq:main proof gauss approx to zero}
\end{equation}
Further we have
\begin{align}
\mathbb{P}\big(\vert\hat{\sigma}-\sigma\vert>\sigma(\log n)^{-2}\big) & =\mathbb{P}\left(\frac{\vert\hat{\sigma}^{2}-\sigma^{2}\vert}{\hat{\sigma}+\sigma}>\sigma(\log n)^{-2}\right)\nonumber \\
 & \leq\mathbb{P}\big(\vert\hat{\sigma}^{2}-\sigma^{2}\vert>\sigma^{2}(\log n)^{-2}\big)\to0\label{eq:main proof sigma est to zero}
\end{align}
by assumption on $\hat{\sigma}$. It remains to handle
\begin{align}
1 & -\mathbb{P}\left(\sigma\mathbf{S}_{k}>\sigma(1-(\log n)^{-2})c_{k}(\beta)-(\log n)^{-1}\right)\nonumber \\
 & =\mathbb{P}\left(\sigma\mathbf{S}_{k}\leq\sigma c_{k}(\beta)-\sigma(\log n)^{-2}c_{k}(\beta)-(\log n)^{-1}\right)\nonumber \\
 & =\mathbb{P}\left(\mathbf{S}_{k}\leq c_{k}(\beta)\right)-\mathbb{P}\left(\sigma c_{k}(\beta)-\sigma(\log n)^{-2}c_{k}(\beta)-(\log n)^{-1}<\sigma\mathbf{S}_{k}\leq\sigma c_{k}(\beta)\right)\nonumber \\
 & =1-\beta-\mathbb{P}\left(-(\log n)^{-2}c_{k}(\beta)-(\log n)^{-1}\sigma^{-1}<\mathbf{S}_{k}-c_{k}(\beta)\leq0\right)\nonumber \\
 & \geq1-\beta-\mathbb{P}\left(\vert\mathbf{S}_{k}-c_{k}(\beta)\vert\leq(\log n)^{-2}c_{k}(\beta)+(\log n)^{-1}\sigma^{-1}\right).\label{eq:main proof remaining term S_k}
\end{align}
Corollary 2.1.\ by \citet{Chernozhukov2014a} yields that
\begin{equation}\label{CCK2.1}
\sup_{\xi\in\mathbb{R}}\mathbb{P}\left(\vert\mathbf{S}_{k}-\xi\vert\leq\kappa\right)\leq4\kappa\left(\mathbb{E}\left[\sup_{f\in\mathcal{F}_{k}}\vert B_{k}f\vert\right]+1\right).
\end{equation}
We apply Corollary 2.2.8.\ by \citet{Vaart1996} to bound the above
expectation. Let $d$ be the standard deviation semimetric on $\mathcal{F}_{k}$
that is involved in this corollary. It holds that
\begin{align}
d(f_{x_{1},k},f_{x_{2},k}) & :=\mathbb{E}\left[\left(B_{k}f_{x_{1},k}-B_{k}f_{x_{2},k}\right)^{2}\right]^{1/2}\nonumber \\
 & \:=\mathbb{E}\left[(B_{k}f_{x_{1},k})^{2}-2B_{k}f_{x_{1},k}B_{k}f_{x_{2},k}+(B_{k}f_{x_{2},k})^{2}\right]^{1/2}\nonumber \\
 & \:=\left(2-2\mathbb{E}\left[B_{k}f_{x_{1},k}B_{k}f_{x_{2},k}\right]\right)^{1/2}\nonumber \\
 & \:=\sqrt{2}\left(1-\cov(B_{k}(f_{x_{1},k}),B_{k}(f_{x_{2},k}))\right)^{1/2}\nonumber \\
 & \:=\sqrt{2}\left(1-\Psi_{k}^{-1/2}(x_{1})\Psi_{k}^{-1/2}(x_{2})\mathbb{E}\left[K_{k}(x_{1},X_{1})K_{k}(x_{2},X_{1})\right]\right)^{1/2}
  \:\leq\sqrt{2}.\label{eq:bound std dev semimetric}
\end{align}
The finite size of $\mathcal{F}_{k}$, that is $\mathcal N_{k}\leq2^{kp}$,
implies that the packing number is bounded by $2^{kp}$ for any semimetric.
For the semimetric $d$ from above $D(\mathcal{F}_{k},d,\epsilon)\leq \mathcal N_{k}\leq2^{kp}$
is also implied by its equality to zero if $x_{1}$ and $x_{2}$ are
in one undividable cell. Note that (\ref{eq:bound std dev semimetric})
further implies that $D(\mathcal{F}_{k},d,\epsilon)=1$ if $\epsilon>\sqrt{2}$.
Using these observations Corollary 2.2.8.\ by \citet{Vaart1996}
yields for a universal constant $K$ that

\begin{align}
\mathbb{E}\left[\sup_{f\in\mathcal{F}_{k}}\vert B_{k}f\vert\right] & \leq\mathbb{E}\left[\vert B_{k}f_{x_{0},k}\vert\right]+K\int_{0}^{\infty}\sqrt{\log D(\mathcal{F}_{k},d,\epsilon)}d\epsilon\nonumber \\
 & \leq\sigma+K\sqrt{2}\sqrt{pk\log2}
 \lesssim\sqrt{k}.\label{eq:upper bound expect sup}
\end{align}
With (\ref{CCK2.1}) and (\ref{eq:upper bound expect sup}) we obtain
\begin{equation*}
\sup_{\xi\in\mathbb{R}}\mathbb{P}(\vert\mathbf{S}_{k}-\xi\vert\leq\kappa)\lesssim\kappa\sqrt{k}.
\end{equation*}
and thus
\begin{equation}
\mathbb{P}\left(\vert\mathbf{S}_{k}-c_{k}(\beta)\vert\leq(\log n)^{-2}c_{k}(\beta)+(\log n)^{-1}\sigma^{-1}\right)\lesssim\left((\log n)^{-2}c_{k}(\beta)+(\log n)^{-1}\sigma^{-1}\right)\sqrt{k}.\label{eq:o(1) mit c_k}
\end{equation}
Thus, we need an upper bound for $c_{k}(\beta)$. Again using (\ref{eq:upper bound expect sup}),
yields that
\[
\mathbb{P}(\mathbf{S}_{k}\leq\xi)=1-\mathbb{P}(\mathbf{S}_{k}>\xi)\geq1-\frac{1}{\xi}\mathbb{E}[\mathbf{S}_{k}]\geq1-\xi^{-1}C\sqrt{k}
\]
for some constant $C$. We recall that $c_{k}(\beta)=\inf\{\xi\in\mathbb{R}:\mathbb{P}(\mathbf{S}_{k}\leq\xi)\geq1-\beta\}$.
For $\xi=C\sqrt{k}\beta^{-1}$ we get
\begin{equation*}
c_{k}(\beta)\leq C\beta^{-1}\sqrt{k}\lesssim\sqrt{k}.
\end{equation*}
With (\ref{eq:assum remainder 2}) and (\ref{eq:assum remainder 1})
which implies $2^{k}=o(n)$ it holds that $k=\mathcal{O}(\log n)$
and therefore (\ref{eq:o(1) mit c_k}) is $o(1)$. Thus (\ref{eq:main proof remaining term S_k})
yields
\begin{align*}
\mathbb{P} & \left(\sigma\mathbf{S}_{k}\leq\sigma c_{k}(\beta)-\sigma(\log n)^{-2}c_{k}(\beta)-(\log n)^{-1}\right)\\
 & \geq1-\beta-\mathbb{P}\left(\vert\mathbf{S}_{k}-c_{k}(\beta)\vert\leq(\log n)^{-2}c_{k}(\beta)+(\log n)^{-1}\sigma^{-1}\right)\\
 & =1-\beta-o(1).
\end{align*}
Together with (\ref{eq:main proof first bound}), (\ref{eq:main proof gauss approx to zero})
and (\ref{eq:main proof sigma est to zero}) we obtain
\[
\liminf_{n\to\infty}\inf_{m\in\mathcal{H}^\alpha(C_{H})}\mathbb{P}\left(m(x)\in\hat{\mathcal{C}}_{n}(x),\,\forall x\in[0,1]^{p}\right)\geq1-\beta.\qedhere
\]
\end{proof}

\subsubsection{Proof of Proposition \ref{prop:asymp dist}\label{subsec:Proof-of-Proposition main CB}}

To prove the proposition we use Theorem \ref{thm:supUhat-approx}
and the auxiliary results below. Their proofs can be found in Section
\ref{subsec:Proof aux}. We first present the results that handle the approximation error
$U_{n,r_{n},\omega}^{(m)}(x_{0})-m(x_{0})$. For
\begin{equation}
U_{n,r_{n},\omega}^{(1)}(x_{0}):=\frac{1}{\binom{n}{r_{n}}}\sum_{I\in \setbinom{n}{r_n}}\mathbb{I}\{\exists j\in I:X_{j}\in A_{k}(x_{0},\omega_{I})\}\label{eq:U^(1)}
\end{equation}
we have
\begin{equation}
U_{n,r_{n},\omega}^{(m)}(x_{0})-m(x_{0})=U_{n,r_{n},\omega}^{(m)}(x_{0})-m(x_{0})U_{n,r_{n},\omega}^{(1)}(x_{0})+m(x_{0})(U_{n,r_{n},\omega}^{(1)}(x_{0})-1).\label{eq:approx error decomp}
\end{equation}
The two results below handle both terms from this decomposition. Their
proofs are postponed to Section \ref{subsec:Proof aux}.
\begin{lemma}
\label{lem: approx err diam}For an $\alpha$-H\"older continuous $m:[0,1]^{p}\to\mathbb{R}$
with H\"older constant $C_{H}$ it holds that
\[
\mathbb{E}\big[\Vert U_{n,r_{n},\omega}^{(m)}-mU_{n,r_{n},\omega}^{(1)}\Vert_{\infty}\big]\leq C_{H}\left(\mathbb{E}[\mathfrak{d}(A_{k}(x,\omega))^{\alpha}]+p^{\alpha/2}\mathcal N_{k}\binom{n}{r_{n}}^{-1/2}\right)
\]
for an arbitrary $x\in[0,1]^{p}$.
\end{lemma}

\begin{lemma}
\label{lem: approx err full cells}For a bounded $m:[0,1]^{p}\to\mathbb{R}$
it holds that
\[
\mathbb{E}\big[\Vert m(U_{n,r_{n},\omega}^{(1)}-1)\Vert_{\infty}\big]\leq\Vert m\Vert_{\infty}2^{k}(1-c_{X}2^{-k})^{r_{n}}.
\]
\end{lemma}

The next results will handle the projection error introduced in Section
\ref{sec:Proof-strategy chap CBs}. We have from (\ref{eq:U^(eps) CPRF})
\begin{align*}
U_{n,r_{n},\omega}^{(\varepsilon)}(x_{0}) 
 & =\frac{r_{n}}{n}\sum_{j=1}^{n}\varepsilon_{j}\frac{1}{\binom{n-1}{r_{n}-1}}\sum_{I\in \setbinom{n}{r_n}:j\in I}\frac{\mathbb{I}\{X_{j}\in A_{k}(x_{0},\omega_{I})\}}{\sum_{i\in I}\mathbb{I}\{X_{i}\in A_{k}(x_{0},\omega_{I})\}}.
\end{align*}
For $\hat{U}_{n,r_{n},\omega}^{(\varepsilon)}$ from (\ref{eq:U hat eps})
we obtain

\begin{align}
 & U_{n,r_{n},\omega}^{(\varepsilon)}(x_{0})-\hat{U}_{n,r_{n},\omega}^{(\varepsilon)}(x_{0})\nonumber \\
 & =\frac{1}{n}\sum_{j=1}^{n}\varepsilon_{j}\left(\frac{r_{n}}{\binom{n-1}{r_{n}-1}}\sum_{I\in \setbinom{n}{r_n}:j\in I}\frac{\mathbb{I}\{X_{j}\in A_{k}(x_{0},\omega_{I})\}}{\sum_{i\in I}\mathbb{I}\{X_{i}\in A_{k}(x_{0},\omega_{I})\}}-K_{k}(x_{0},X_{j})\right)\nonumber \\
 & =\frac{1}{n}\sum_{j=1}^{n}\varepsilon_{j}\frac{r_{n}}{\binom{n-1}{r_{n}-1}}\sum_{I\in \setbinom{n}{r_n}:j\in I}\left(\frac{\mathbb{I}\{X_{j}\in A_{k}(x_{0},\omega_{I})\}}{\sum_{i\in I}\mathbb{I}\{X_{i}\in A_{k}(x_{0},\omega_{I})\}}-\frac{1}{r_{n}p_{x_{0}}(\omega_{I})}\mathbb{I}\{X_{j}\in A_{k}(x_{0},\omega_{I})\}\right)\nonumber \\
 & \qquad+\frac{1}{n}\sum_{j=1}^{n}\varepsilon_{j}\left(\frac{r_{n}}{\binom{n-1}{r_{n}-1}}\sum_{I\in \setbinom{n}{r_n}:j\in I}\frac{1}{r_{n}p_{x_{0}}(\omega_{I})}\mathbb{I}\{X_{j}\in A_{k}(x_{0},\omega_{I})\}-K_{k}(x_{0},X_{j})\right)\nonumber \\
 & =\frac{1}{\binom{n}{r_{n}}}\sum_{I\in \setbinom{n}{r_n}}\sum_{j\in I}\varepsilon_{j}\mathbb{I}\{X_{j}\in A_{k}(x_{0},\omega_{I})\}\left(\frac{1}{\sum_{i\in I}\mathbb{I}\{X_{i}\in A_{k}(x_{0},\omega_{I})\}}-\frac{1}{r_{n}p_{x_{0}}(\omega_{I})}\right)\nonumber \\
 & \qquad+\frac{1}{n}\sum_{j=1}^{n}\varepsilon_{j}\frac{1}{\binom{n-1}{r_{n}-1}}\sum_{I\in \setbinom{n}{r_n}:j\in I}\left(p_{x_{0}}(\omega_{I})^{-1}\mathbb{I}\{X_{j}\in A_{k}(x_{0},\omega_{I})\}-K_{k}(x_{0},X_{j})\right)\nonumber \\
 & =:R_{n,r_{n},\omega}^{(1)}(x_{0})+R_{n,r_{n},\omega}^{(2)}(x_{0}).\label{eq:proj error decomp}
\end{align}
The lemmas below handle the two remainder terms $R_{n,r_{n},\omega}^{(1)}(x_{0})$
and $R_{n,r_{n},\omega}^{(2)}(x_{0})$ separately. Their proofs are
again postponed and can be found in Section \ref{subsec:Proof aux}.
\begin{lemma}
\label{lem:remainder 1} Let $q\in2\mathbb{N}$ be fixed with $q\leq r_{n}$
and $r_{n}\to\infty$. Assume that $\mathbb{E}\left[\vert\varepsilon_{1}\vert^{q}\right]<\infty$
and $2^{k}\leq r_{n}$, then there exists a constant $C$ that depends
on $q$ but not on $n$ and $k$ such that
\[
\mathbb{E}\left[R_{n,r_{n},\omega}^{(1)}(x_{0})^{q}\right]\leq C\left(\frac{2^{2k}}{r_{n}n}\right)^{q/2}.
\]
\end{lemma}

\begin{lemma}
\label{lem:remainder 2} For $R_{n,r_{n},\omega}^{(2)}(x_{0})$ it
holds that
\[
\mathbb{E}\left[R_{n,r_{n},\omega}^{(2)}(x_{0})^{2}\right]=\var\left(R_{n,r_{n},\omega}^{(2)}(x_{0})\right)\leq\frac{\sigma^{2}C_{X}}{c_{X}^{2}}\frac{2^{k}}{r_{n}}\left(\frac{r_{n}}{n}\right)^{r_{n}}.
\]
\end{lemma}

 Combining  these results, we are
able to prove the Proposition \ref{prop:asymp dist}.
\begin{proof}[Proof of Proposition \ref{prop:asymp dist}]
 Using the reverse triangle inequality we get
\begin{align*}
\big\vert\Vert\Psi_{k}^{-1/2} & \hat{U}_{n,r_{n},\omega}^{(\varepsilon)}\Vert_{\infty}-\Vert\Psi_{k}^{-1/2}(U_{n,r_{n},\omega}^{(\mathrm{RF})}-m)\Vert_{\infty}\big\vert\\
 & \leq\Vert\Psi_{k}^{-1/2}(\hat{U}_{n,r_{n},\omega}^{(\varepsilon)}-(U_{n,r_{n},\omega}^{(\mathrm{RF})}-m))\Vert_{\infty}\\
 & \leq\Vert\Psi_{k}^{-1/2}(\hat{U}_{n,r_{n},\omega}^{(\varepsilon)}-U_{n,r_{n},\omega}^{(\varepsilon)})\Vert_{\infty}+\Vert\Psi_{k}^{-1/2}(U_{n,r_{n},\omega}^{(m)}-m)\Vert_{\infty}\\
 & \leq\Vert\Psi_{k}^{-1/2}R_{n,r_{n},\omega}^{(1)}\Vert_{\infty}+\Vert\Psi_{k}^{-1/2}R_{n,r_{n},\omega}^{(2)}\Vert_{\infty}+\Vert\Psi_{k}^{-1/2}(U_{n,r_{n},\omega}^{(m)}-m)\Vert_{\infty}.
\end{align*}
Hence, for the sequence of random variables $\mathbf{S}_{k}$ as defined
in the theorem we get
\begin{align}
\vert\sqrt{n} & \Vert\Psi_{k}^{-1/2}(U_{n,r_{n},\omega}^{(\mathrm{RF})}-m)\Vert_{\infty}-\sigma\mathbf{S}_{k}\vert\nonumber \\
 & =\Big\vert\sqrt{n}\left(\Vert\Psi_{k}^{-1/2}(U_{n,r_{n},\omega}^{(\mathrm{RF})}-m)\Vert_{\infty}-\Vert\Psi_{k}^{-1/2}\hat{U}_{n,r_{n},\omega}^{(\varepsilon)}\Vert_{\infty}+\Vert\Psi_{k}^{-1/2}\hat{U}_{n,r_{n},\omega}^{(\varepsilon)}\Vert_{\infty}\right)-\sigma\mathbf{S}_{k}\Big\vert\nonumber \\
 & \leq\sqrt{n}\Vert\Psi_{k}^{-1/2}\Vert_{\infty}\left(\Vert U_{n,r_{n},\omega}^{(m)}-m\Vert_{\infty}+\Vert R_{n,r_{n},\omega}^{(1)}\Vert_{\infty}+\Vert R_{n,r_{n},\omega}^{(2)}\Vert_{\infty}\right)\nonumber \\
 & \qquad+\vert\sqrt{n}\Vert\Psi_{k}^{-1/2}\hat{U}_{n,r_{n},\omega}^{(\varepsilon)}\Vert_{\infty}-\sigma\mathbf{S}_{k}\vert\nonumber \\
 & \lesssim\Big(\frac{n}{2^{2k}\mathcal{V}_{\cap,k}}\Big)^{1/2}\big(\Vert U_{n,r_{n},\omega}^{(m)}-m\Vert_{\infty}+\Vert R_{n,r_{n},\omega}^{(1)}\Vert_{\infty}+\Vert R_{n,r_{n},\omega}^{(2)}\Vert_{\infty}\big)\nonumber \\
 & \qquad+\vert\sqrt{n}\Vert\Psi_{k}^{-1/2}\hat{U}_{n,r_{n},\omega}^{(\varepsilon)}\Vert_{\infty}-\sigma\mathbf{S}_{k}\vert\label{eq:error asymp dist decomp proof prop}
\end{align}
because
\[
\Vert\Psi_{k}^{-1/2}\Vert_{\infty}\leq\left(\frac{C_{X}^{2}}{c_{X}}2^{2k}\mathcal{V}_{\cap,k}\right)^{-1/2}
\]
owing to (\ref{eq:uni bounds for Psi}). We omit the constants for
a shorter notation. In the rest of the proof we will show that all
terms from (\ref{eq:error asymp dist decomp proof prop}) except the
last are $o_{\mathbb{P}}((\log n)^{-1})$. For the last term we apply
Theorem \ref{thm:supUhat-approx}.

For the approximation error we have with (\ref{eq:approx error decomp})
that
\[
\Vert U_{n,r_{n},\omega}^{(m)}-m\Vert_{\infty}\leq\Vert U_{n,r_{n},\omega}^{(m)}-mU_{n,r_{n},\omega}^{(1)}\Vert_{\infty}+\Vert m(U_{n,r_{n},\omega}^{(1)}-1)\Vert_{\infty}.
\]
Using Lemma \ref{lem: approx err diam} we get
\begin{align}
\mathbb{P} & \left(\Vert U_{n,r_{n},\omega}^{(m)}-mU_{n,r_{n},\omega}^{(1)}\Vert_{\infty}\geq\kappa(\log n)^{-1}\Big(\frac{n}{2^{2k}\mathcal{V}_{\cap,k}}\Big)^{-1/2}\right)\nonumber \\
 & \leq\left(\frac{n(\log n)^{2}}{\kappa^{2}2^{2k}\mathcal{V}_{\cap,k}}\right)^{1/2}\mathbb{E}\left[\Vert U_{n,r_{n},\omega}^{(m)}-mU_{n,r_{n},\omega}^{(1)}\Vert_{\infty}\right]\nonumber \\
 & \leq C_{H}\left(\frac{n(\log n)^{2}}{\kappa^{2}2^{2k}\mathcal{V}_{\cap,k}}\right)^{1/2}\left(\mathbb{E}[\mathfrak{d}(A_{k}(x,\omega))^{\alpha}]+p^{\alpha/2}\mathcal N_{k}\binom{n}{r_{n}}^{-1/2}\right)\label{eq:approx second term proof prop}
\end{align}
for all $\kappa>0$. Assumption (\ref{eq:assum approx diam}) yields
\[
C_{H}\big(\frac{n(\log n)^{2}}{\kappa^{2}2^{2k}\mathcal{V}_{\cap,k}}\big)^{1/2}\mathbb{E}[\mathfrak{d}(A_{k}(x,\omega))^{\alpha}]\to0
\]
for all $\kappa>0$. The fact that $\mathbb{E}[\mathfrak{d}(A_{k}(x,\omega))^{\alpha}]\geq p^{\alpha/2}2^{-\alpha k/p}$
and $\mathcal{V}_{\cap,k}\leq2^{-k}$ together with (\ref{eq:assum approx diam})
imply
\[
p^{\alpha}2^{-2\alpha k/p}\frac{n}{2^{k}}\leq\mathbb{E}[\mathfrak{d}(A_{k}(x,\omega))^{\alpha}]^{2}\frac{n(\log n)^{2}}{2^{2k}\mathcal{V}_{\cap,k}}.
\]
Hence $n=o(2^{k(1+2\alpha/p)})$. Again using $\mathcal{V}_{\cap,k}\leq2^{-k}$
assumption (\ref{eq:assum remainder 1}) implies that $2^{k}=o(r_{n})$.
For the squared second term from (\ref{eq:approx second term proof prop})
we obtain with (\ref{eq:assum remainder 2}) that
\begin{align}
C_{H}^{2}\frac{n(\log n)^{2}}{\kappa^{2}2^{2k}\mathcal{V}_{\cap,k}}p^{2\alpha}\mathcal N_{k}^{2}\binom{n}{r_{n}}^{-1} & \leq\frac{C_{H}^{2}p^{2\alpha}}{\kappa^{2}}n(\log n)^{2}2^{2kp}\big(\frac{r_{n}}{n}\big)^{r_{n}}\nonumber \\
 & =o\big(2^{k(2p+2+2\alpha/p)}c^{r_{n}}\big)\nonumber \\
 & =o\big(r_{n}^{(2p+2+2\alpha/p)}\exp(r_{n}\log(c))\big)\to0\label{eq:Nf(k) und bin coeff convergence}
\end{align}
because $(\log n)^{2}=o(2^{k})$ and $\log(c)<0$. Lemma \ref{lem: approx err full cells}
yields
\begin{align*}
\mathbb{P} & \left(\Vert m(U_{n,r_{n},\omega}^{(1)}-1)\Vert_{\infty}\geq\kappa(\log n)^{-1}\Big(\frac{n}{2^{2k}\mathcal{V}_{\cap,k}}\Big)^{-1/2}\right)\\
 & \leq\kappa^{-1}\left(\frac{n(\log n)^{2}}{2^{2k}\mathcal{V}_{\cap,k}}\right)^{1/2}\mathbb{E}\left[\Vert m(U_{n,r_{n},\omega}^{(1)}-1)\Vert_{\infty}\right]\\
 & =\kappa^{-1}\left(\frac{n(\log n)^{2}}{\mathcal{V}_{\cap,k}}\right)^{1/2}\Vert m\Vert_{\infty}(1-c_{X}2^{-k})^{r_{n}}\to0
\end{align*}
for all $\kappa>0$ with (\ref{eq:assum remainder 1}).

Assumption (\ref{eq:assum moments}) implies that $\mathbb{E}[\vert\varepsilon_{1}\vert^{2\nu}]<\infty$.
Hence for every $\kappa>0$ Lemma \ref{lem:remainder 1} yields
\begin{align*}
\mathbb{P}\left(\Vert R_{n,r_{n},\omega}^{(1)}\Vert_{\infty}\geq\kappa(\log n)^{-1}\Big(\frac{n}{2^{2k}\mathcal{V}_{\cap,k}}\Big)^{-1/2}\right) & \leq\mathbb{E}\left[\Vert R_{n,r_{n},\omega}^{(1)}\Vert_{\infty}^{2\nu}\right]\left(\frac{n(\log n)^{2}}{\kappa^{2}2^{2k}\mathcal{V}_{\cap,k}}\right)^{\nu}\\
 & \lesssim\sum_{x_{0}\in\mathcal{X}_{k}}\mathbb{E}\left[R_{n,r_{n},\omega}^{(1)}(x_{0})^{2\nu}\right]\left(\frac{n(\log n)^{2}}{2^{2k}\mathcal{V}_{\cap,k}}\right)^{\nu}\\
 & \lesssim \mathcal N_{k}\left(\frac{(\log n)^{2}}{r_{n}\mathcal{V}_{\cap,k}}\right)^{\nu}\to0
\end{align*}
due to Assumption (\ref{eq:assum remainder 1}). Lemma \ref{lem:remainder 2}
implies with a union bound and by using $\mathcal N_{k}\leq2^{kp}$ and
$\mathcal{V}_{\cap,k}\geq2^{-2k}$ that
\begin{align*}
\mathbb{P}\left(\Vert R_{n,r_{n},\omega}^{(2)}\Vert_{\infty}\geq\kappa(\log n)^{-1}\Big(\frac{n}{2^{2k}\mathcal{V}_{\cap,k}}\Big)^{-1/2}\right) & \leq\mathbb{E}\left[\Vert R_{n,r_{n},\omega}^{(2)}\Vert_{\infty}^{2}\right]\frac{n(\log n)^{2}}{\kappa^{2}2^{2k}\mathcal{V}_{\cap,k}}\\
 & \leq\frac{n(\log n)^{2}}{\kappa^{2}2^{2k}\mathcal{V}_{\cap,k}}\sum_{x_{0}\in\mathcal{X}_{k}}\mathbb{E}\left[R_{n,r_{n},\omega}^{(2)}(x_{0})^{2}\right]\\
 & \leq 
\mathcal N_{k}\frac{\sigma^{2}C_{X}}{c_{X}^{2}\kappa^{2}}\frac{(\log n)^{2}}{2^{k}\mathcal{V}_{\cap,k}}\left(\frac{r_{n}}{n}\right)^{r_{n}-1}\\
 & \leq\frac{\sigma^{2}C_{X}}{c_{X}^{2}\kappa^{2}}2^{k(p+1)}(\log n)^{2}c^{r_{n}-1}\to0
\end{align*}
with the same argument we already used for (\ref{eq:Nf(k) und bin coeff convergence}).
Applying (\ref{eq:error asymp dist decomp proof prop}) and Theorem
\ref{thm:supUhat-approx} for any $\nu\geq4$ with $\mathbb{E}[\vert\varepsilon_{1}\vert^{\nu}]<\infty$
we end up with
\begin{align*}
\big\vert\sqrt{n} & \Vert\Psi_{k}^{-1/2}(U_{n,r_{n},\omega}^{(\mathrm{RF})}-m)\Vert_{\infty}-\sigma\mathbf{S}_{k}\big\vert\\
 & \lesssim\Big(\frac{n}{2^{2k}\mathcal{V}_{\cap,k}}\Big)^{1/2}\big(\Vert U_{n,r_{n},\omega}^{(m)}-m\Vert_{\infty}+\Vert R_{n,r_{n},\omega}^{(1)}\Vert_{\infty}+\Vert R_{n,r_{n},\omega}^{(2)}\Vert_{\infty}\big)\\
 & \qquad+\vert\sqrt{n}\Vert\Psi_{k}^{-1/2}\hat{U}_{n,r_{n},\omega}^{(\varepsilon)}\Vert_{\infty}-\sigma\mathbf{S}_{k}\vert\\
 & =o_{\mathbb{P}}\left((\log n)^{-1}\right)+\mathcal{O}_{\mathbb{P}}\left(\frac{(\log n)^{3/2}}{\mathcal{V}_{\cap,k}^{1/2}n^{1/2-1/\nu}}+\frac{(\log n)^{5/4}}{\mathcal{V}_{\cap,k}^{1/4}n^{1/4}}+\frac{\log n}{\mathcal{V}_{\cap,k}^{1/6}n^{1/6}}\right).\qedhere
\end{align*}
\end{proof}

\subsubsection{Proof of Corollary \ref{cor:CB Incomplete U-stat}}

We prove the corollary on the event that $\{\hat{N}>0\}$. The case
$\hat{N}=0$ is not of interest as it corresponds to an empty random
forest. We omit the indicator by bounding it by one whenever possible.
The incomplete generalized U-statistic is
\[
U_{n,r_{n},N,\omega}^{(\mathrm{RF})}(x_{0})=\frac{1}{\hat{N}}\sum_{I\in \setbinom{n}{r_n}}\rho_{I}\sum_{j\in I}Y_{j}W_{j,k}(x_{0},I)
\]
with
$W_{j,k}$ 
from (\ref{eq:notation weights abbrev}). We decompose its difference
to the complete U-statistic
\begin{align}
U_{n,r_{n},N,\omega}^{(\mathrm{RF})} & (x_{0})-U_{n,r_{n},\omega}^{(\mathrm{RF})}(x_{0})\nonumber \\
 & =\frac{1}{\hat{N}}\sum_{I\in \setbinom{n}{r_n}}\rho_{I}\sum_{j\in I}Y_{j}W_{j,k}(x_{0},I)-\frac{1}{\binom{n}{r_{n}}}\sum_{I\in \setbinom{n}{r_n}}\sum_{j\in I}Y_{j}W_{j,k}(x_{0},I)\nonumber \\
 & =\Big(\frac{1}{\hat{N}}-\frac{1}{N}\Big)\sum_{I\in \setbinom{n}{r_n}}\rho_{I}\sum_{j\in I}Y_{j}W_{j,k}(x_{0},I)\label{eq:proof cor inc. decomp}\\
 & \qquad+\sum_{I\in \setbinom{n}{r_n}}\Big(\frac{\rho_{I}}{N}-\frac{1}{\binom{n}{r_{n}}}\Big)\sum_{j\in I}Y_{j}W_{j,k}(x_{0},I).\label{eq:proof cor inc. decomp-2}
\end{align}
We start with a bound that will be used for both terms. We exploit
that the sum of the $W_{j,k}$ is either zero or one. In the first
case the following bound holds trivially. We recall the abbreviated
notation $W_{j,k}(x_{0},\omega)$ for the set $[r_{n}]$ from (\ref{eq:notation weights abbrev}).
In the second case we use Jensen's inequality, Lemma \ref{lem:Gy=0000F6rfi 4.1}
and the bound for $p_{x_{0}}(\omega)$ from (\ref{eq:bound px omega})
to obtain
\begin{align}
\mathbb{E} & \left[\left(\sum_{j=1}^{r_{n}}Y_{j}W_{j,k}(x_{0},\omega)\right)^{2}\right]
  \leq\mathbb{E}\left[\sum_{j=1}^{r_{n}}Y_{j}^{2}W_{j,k}(x_{0},\omega)\right]\nonumber \\
 & =r_{n}\mathbb{E}\left[Y_{1}^{2}\frac{\mathbb{I}\{X_{1}\in A_{k}(x_{0},\omega)\}}{\sum_{i=1}^{r_{n}}\mathbb{I}\{X_{i}\in A_{k}(x_{0},\omega)\}}\right]\nonumber \\
 & =r_{n}\mathbb{E}\left[(m(X_{1})+\varepsilon_{1})^{2}\mathbb{I}\{X_{1}\in A_{k}(x_{0},\omega)\}\mathbb{E}\left[\frac{1}{1+\sum_{i=2}^{r_{n}}\mathbb{I}\{X_{i}\in A_{k}(x_{0},\omega)\}}\:\big\vert\:X_{1},\varepsilon_{1},\omega\right]\right]\nonumber \\
 & \leq r_{n}2\mathbb{E}\left[(m(X_{1})^{2}+\varepsilon_{1}^{2})\mathbb{I}\{X_{1}\in A_{k}(x_{0},\omega)\}\frac{1}{r_{n}p_{x_{0}}(\omega)}\right]\nonumber \\
 & \leq2c_{X}^{-1}2^{k}\mathbb{E}\left[(\Vert m\Vert_{\infty}^{2}+\varepsilon_{1}^{2})\mathbb{I}\{X_{1}\in A_{k}(x_{0},\omega)\}\right] =2c_{X}^{-1}(\Vert m\Vert_{\infty}^{2}+\sigma^{2}).\label{eq:bound of r weighted Y squared}
\end{align}
Using that the $\rho_{I}$ are i.i.d.\ and independent of all the
other random variables and (\ref{eq:bound of r weighted Y squared})
we get for the expectation of the term from (\ref{eq:proof cor inc. decomp-2})
\begin{align}
 & \sum_{I\in \setbinom{n}{r_n}}\mathbb{E}\left[\Big(\rho_{I}\frac{1}{N}-\frac{1}{\binom{n}{r_{n}}}\Big)^{2}\right]\mathbb{E}\left[\left(\sum_{j=1}^{r_{n}}Y_{j}W_{j,k}(x_{0},\omega)\right)^{2}\right]\nonumber \\
 & =\binom{n}{r_{n}}\frac{1}{N^{2}}\var(\rho_{I})2c_{X}^{-1}(\Vert m\Vert_{\infty}^{2}+\sigma^{2})\nonumber \\
 & =\binom{n}{r_{n}}\frac{1}{N^{2}}\frac{N}{\binom{n}{r_{n}}}\left(1-\frac{N}{\binom{n}{r_{n}}}\right)2c_{X}^{-1}(\Vert m\Vert_{\infty}^{2}+\sigma^{2})\nonumber \\
 & \leq\frac{1}{N}2c_{X}^{-1}(\Vert m\Vert_{\infty}^{2}+\sigma^{2})
  =\mathcal{O}(N^{-1}).\label{eq:proof cor inc. bound second part}
\end{align}
With (\ref{eq:proof cor inc. bound second part}) and (\ref{eq:bound of r weighted Y squared})
we further obtain
\begin{align}
\mathbb{E} & \left[\left(\sum_{I\in \setbinom{n}{r_n}}\rho_{I}\sum_{j\in I}Y_{j}W_{j,k}(x_{0},I)\right)^{2}\right]\nonumber \\
 & \leq2\mathbb{E}\left[\left(\sum_{I\in \setbinom{n}{r_n}}\left(\rho_{I}-\frac{N}{\binom{n}{r_{n}}}\right)\sum_{j\in I}Y_{j}W_{j,k}(x_{0},I)\right)^{2}\right]+2\mathbb{E}\left[\left(\sum_{I\in \setbinom{n}{r_n}}\frac{N}{\binom{n}{r_{n}}}\sum_{j\in I}Y_{j}W_{j,k}(x_{0},I)\right)^{2}\right]\nonumber \\
 & =2N^{2}\mathbb{E}\left[\left(\sum_{I\in \setbinom{n}{r_n}}\left(\frac{\rho_{I}}{N}-\frac{1}{\binom{n}{r_{n}}}\right)\sum_{j\in I}Y_{j}W_{j,k}(x_{0},I)\right)^{2}\right]+2N^{2}\mathbb{E}\left[\left(\frac{1}{\binom{n}{r_{n}}}\sum_{I\in \setbinom{n}{r_n}}\sum_{j\in I}Y_{j}W_{j,k}(x_{0},I)\right)^{2}\right]\nonumber \\
 & =\mathcal{O}(N)+\mathcal{O}(N^{2}).\label{eq:proof inc. cor bound 1}
\end{align}
We note that $\hat{N}\sim\text{Bin}(N,N/\binom{n}{r_{n}})$. Thus,
for the first term from (\ref{eq:proof cor inc. decomp}) we get with
the Cauchy-Schwarz inequality, Lemma \ref{lem:moments binomial fractures}
and (\ref{eq:proof inc. cor bound 1}) that
\begin{align*}
\mathbb{E} & \left[\mathbb{I}\{\hat{N}>0\}\bigg\vert\left(\frac{1}{\hat{N}}-\frac{1}{N}\right)\sum_{I\in \setbinom{n}{r_n}}\rho_{I}\sum_{j\in I}Y_{j}W_{j,k}(x_{0},I)\bigg\vert\right]^{2}\\
 & \leq\mathbb{E}\left[\mathbb{I}\{\hat{N}>0\}\left(\frac{1}{\hat{N}}-\frac{1}{N}\right)^{2}\right]\mathbb{E}\left[\left(\sum_{I\in \setbinom{n}{r_n}}\rho_{I}\sum_{j\in I}Y_{j}W_{j,k}(x_{0},I)\right)^{2}\right]\\
 & =\mathcal{O}(N^{-3})\mathcal{O}(N^{2}) =\mathcal{O}(N^{-1}).
\end{align*}
Together with (\ref{eq:proof cor inc. decomp}) and (\ref{eq:proof cor inc. bound second part})
this implies
\[
\mathbb{E}\left[\mathbb{I}\{\hat{N}>0\}\vert U_{n,r_{n},N,\omega}^{(\mathrm{RF})}(x_{0})-U_{n,r_{n},\omega}^{(\mathrm{RF})}(x_{0})\vert\right]=\mathcal{O}(N^{-1/2}).
\]
We obtain
\begin{align*}
\mathbb{P} & \bigg(\mathbb{I}\{\hat{N}>0\}\sqrt{n}\sup_{x_{0}\in[0,1]^{p}}\Psi_{k}^{-1/2}(x_{0})\vert U_{n,r_{n},N,\omega}^{(\mathrm{RF})}(x_{0})-U_{n,r_{n},\omega}^{(\mathrm{RF})}(x_{0})\vert>(\log n)^{-1}\bigg)\\
 & \leq\log n\mathbb{E}\bigg[\mathbb{I}\{\hat{N}>0\}\sqrt{n}\sup_{x_{0}\in[0,1]^{p}}\Psi_{k}^{-1/2}(x_{0})\vert U_{n,r_{n},N,\omega}^{(\mathrm{RF})}(x_{0})-U_{n,r_{n},\omega}^{(\mathrm{RF})}(x_{0})\vert\bigg]\\
 & \leq\log n\sqrt{\frac{n}{2^{2k}\mathcal{V}_{\cap,k}}}\sum_{x_{0}\in\mathcal{X}_{k}}\mathbb{E}\left[\mathbb{I}\{\hat{N}>0\}\vert U_{n,r_{n},N,\omega}^{(\mathrm{RF})}(x_{0})-U_{n,r_{n},\omega}^{(\mathrm{RF})}(x_{0})\vert\right]\\
 & =\mathcal{O}\left(\log n\sqrt{\frac{n}{2^{2k}\mathcal{V}_{\cap,k}}}\frac{\mathcal N_{k}}{N^{1/2}}\right)\to0.
\end{align*}
Including this in the decomposition in the proof of Proposition \ref{prop:asymp dist}
yields
\begin{align*}
\mathbb{I} & \{\hat{N}>0\}\vert\sqrt{n}\Vert\Psi_{k}^{-1/2}(U_{n,r_{n},N,\omega}^{(\mathrm{RF})}-m)\Vert_{\infty}-\sigma\mathbf{S}_{k}\vert\\
 & =\mathcal{O}_{\mathbb{P}}\left(\frac{(\log n)^{3/2}}{\mathcal{V}_{\cap,k}^{1/2}n^{1/2-1/\nu}}+\frac{(\log n)^{5/4}}{\mathcal{V}_{\cap,k}^{1/4}n^{1/4}}+\frac{\log n}{\mathcal{V}_{\cap,k}^{1/6}n^{1/6}}\right)+o_{\mathbb{P}}\left((\log n)^{-1}\right).
\end{align*}
Using this instead of Proposition \ref{prop:asymp dist} in the proof
of Theorem \ref{thm:CB main01} yields the result. $\qed$

\subsection{Proofs for Section~\ref{sec:rates}}\label{sec:proofRates}

\begin{proof}[Proof of Lemma~\ref{lem:CharacteristicsUCPRF}]
  Since $S_l(x_0,\omega)\sim\mathrm{Bin}(k,1/p)$, we have for $q\le2$
  \begin{align*}
\mathbb{E}\left[\mathfrak{d}(A_{k}(x_{0},\omega))^{q}\right] & \leq
p\mathbb{E}\left[\exp\big(\log(2^{-q})S_{1}(x_{0},\omega)\big)\right]=p\left(\frac{p-1+2^{-q}}{p}\right)^{k}.
\end{align*}
We proceed with $\mathcal{V}_{\cap,k}$. Note that $(S_{l}(x_{0},\omega_{1}))_{l=1}^{p}$ is multinomial distributed with $p$ trials and probabilities all equal to $1/p$.
The same holds for its independent copy $(S_{l}(x_{0},\omega_{2}))_{l=1}^{p}$.
Lemma S.1 by \citet{Klusowski2021} in the supplementary material
implies
\[
2^{-k}\frac{p^{p}}{(47)^{p}k^{p-1}}\leq\mathcal{V}_{\cap,k}\leq2^{-k}\frac{8^{p}p^{p/2}}{\sqrt{k^{p-1}}},
\]
and for the rate in $k$ we obtain
\begin{equation*}
2^{-k}k^{-(p-1)}\lesssim\mathcal{V}_{\cap,k}\lesssim2^{-k}k^{-(p-1)/2}.
\end{equation*}
Finally, the finest partition of the feature space is given via an equidistant grid with mesh length $2^{-k}$, such that $\mathcal{N}_{k}=2^{kp}$.
\end{proof}

\begin{proof}[Proof of Proposition~\ref{prop:CharacteristicsEhrenfest}]
We drop the dependence of $S_{l,t}$ on $x_{0}$ and $\omega$
in the notation for convenience. Without loss of generality we consider the splits in direction $1$. We first prove the upper bound. If $pS_{1,t}<t+p\Delta$ for all $t\in[k]$, we have nothing to prove. Otherwise, let $t$ be the last time such that $pS_{1,t}\geq t+p\Delta$ and $pS_{1,t-1}<t-1+p\Delta$. Hence, 
\[
pS_{1,t}=p(S_{1,t-1}+1)<t-1+p\Delta+p=t+p\Delta+(p-1).
\]
To exceed the threshold in time step $t$ one particle had to leave container $1$. Thus there can be at most $pB-1$ particles left in the container and due to (\ref{eq:destination prob Ehr}) no particle will enter the first container between times $t$ and $k$. We conclude for any $s=0,1,2,\dots$
\[
pS_{1,t+s}\le p(S_{1,t}+\min\{s,pB-1\})\le t+p\Delta+(p-1)+p(pB-1)\le t+s+p\Delta+p^2B-1.
\]
Therefore, we have $S_{1,k}\leq\frac{k}{p}+\Delta+pB=\frac{k}{p}+C_{\Delta,B}^{(1)}$ which proves the first claim. Using this we get the lower bound
\begin{align*}
S_{1,k}  =k-\sum_{l=2}^{p}S_{l,k}
 & \geq k-(p-1)\left(\frac{k}{p}+\Delta+pB\right)\\
 & =\frac{t}{p}-(p-1)(\Delta+pB)=\frac{t}{p}-C_{\Delta,B}^{(2)}.
\end{align*}

The upper bound on $S_{j,k}$ directly implies an almost sure bound for the diameter:
\[
\mathfrak{d}(A_{k}(x_{0},\omega))^{2}=\sum_{l=1}^{p}2^{-2S_{l}(x_{0},\omega)}=2^{-2k/p}\sum_{l=1}^{p}2^{2(k/p-S_{l}(x_{0},\omega))}\leq2^{-2k/p}p2^{2C_{\Delta,B}^{(1)}}
\]
Taking the square root this yields the claim. The resulting term
is $\mathcal{O}(2^{-k/p})$ because $C_{\Delta,B}^{(1)}$ does not
depend on $k$.

To analyse the cell intersections, we use equation \eqref{eq:vol intersec darst} to deduce
\[
\mathcal{V}_{\cap,k}=\mathbb{E}\left[2^{-\sum_{l=1}^{p}\max\{S_{l}(x_{0},\omega_{1}),S_{l}(x_{0},\omega_{2})\}}\right]=2^{-k}\mathbb{E}\left[2^{-\frac{1}{2}\sum_{l=1}^{p}\vert S_{l}(x_{0},\omega_{1})-S_{l}(x_{0},\omega_{2})\vert}\right].
\]
Together with $
\frac{k}{p}-C_{\Delta,B}^{(2)}\leq S_{l}(x_{0},\omega)\leq\frac{k}{p}+C_{\Delta,B}^{(1)}
$ from above, we obtain
\[
\vert S_{l}(x_{0},\omega_{1})-S_{l}(x_{0},\omega_{2})\vert\leq C_{\Delta,B}^{(1)}+C_{\Delta,B}^{(2)}
\]
which leads to
\[
\mathcal{V}_{\cap,k}=2^{-k}\mathbb{E}\left[2^{-\frac{1}{2}\sum_{l=1}^{p}\vert S_{l}(x_{0},\omega_{1})-S_{l}(x_{0},\omega_{2})\vert}\right]\geq2^{-k}2^{-p(C_{\Delta,B}^{(1)}+C_{\Delta,B}^{(2)})/2}.
\]
This implies $2^{-k}=\mathcal{O}(\mathcal{V}_{\cap,k})$ because the
constants do not depend on $k$ and with the bound $\mathcal{V}_{\cap,k}\leq2^{-k}$
from (\ref{eq:V cap simple bounds}) it holds that $\mathcal{V}_{\cap,k}\sim 2^{-k}$.

It remains to investigate the finest partition. The upper bound $
S_{j,k}\leq\frac{k}{p}+C_{\Delta,B}^{(1)}
$
for all $l$ and $k$ implies that the volume of the undividable
sets is larger than or equal to
\[
2^{-p(\frac{k}{p}+C_{\Delta,B}^{(1)})}=2^{-(k+pC_{\Delta,B}^{(1)})}.
\]
This implies for the number of the cells that
$
\mathcal N_{k}\leq2^{k+pC_{\Delta,B}^{(1)}}\lesssim2^{k}.
$
\end{proof}

\begin{proof}[Proof of Theorem~\ref{thm:MSE better}]

We only consider the stochastic error from
\[
\mathbb{E}\left[\big(U_{n,r_{n},\omega}^{(\mathrm{RF})}(x_{0})-m(x_{0})\big)^{2}\right]=\mathbb{E}\left[\big(U_{n,r_{n},\omega}^{(m)}(x_{0})-m(x_{0})\big)^{2}\right]+\mathbb{E}\left[U_{n,r_{n},\omega}^{(\varepsilon)}(x_{0})^{2}\right].
\]
For the approximation error we use similar arguments to the proof
of Lemma \ref{lem: approx err diam} and Lemma \ref{lem: approx err full cells}.
With (\ref{eq:U hat eps}) and the decomposition from (\ref{eq:proj error decomp})
we have
\[
U_{n,r_{n},\omega}^{(\varepsilon)}(x_{0})=\hat{U}_{n,r_{n},\omega}^{(\varepsilon)}(x_{0})+R_{n,r_{n},\omega}^{(1)}(x_{0})+R_{n,r_{n},\omega}^{(2)}(x_{0}).
\]
The definition of $K_{k}$ in (\ref{eq:RF Hajek kernel}) and (\ref{eq:uni bounds for Psi})
yield
\[
\mathbb{E}\left[\hat{U}_{n,r_{n},\omega}^{(\varepsilon)}(x_{0})^{2}\right]=\frac{1}{n}\mathbb{E}\left[\varepsilon_{1}^{2}K_{k}(x_{0},X_{1})^{2}\right]=\frac{\sigma^{2}}{n}\Psi_{k}(x_{0})\sim 2^{2k}\mathcal{V}_{\cap,k}/n.
\]
Lemma \ref{lem:remainder 1} implies that
\[
\mathbb{E}\left[R_{n,r_{n},\omega}^{(1)}(x_{0})^{2}\right]=\mathcal{O}\left(\frac{2^{2k}}{r_{n}n}\right),
\]
and Lemma \ref{lem:remainder 2} implies
\[
\mathbb{E}\left[R_{n,r_{n},\omega}^{(2)}(x_{0})^{2}\right]=\mathcal{O}\left(\frac{2^{k}}{r_{n}}\left(\frac{r_{n}}{n}\right)^{r_{n}}\right).
\]
Together this yields the claim of the corollary.
\end{proof}

\begin{proof}[Proof of Proposition~\ref{prop:CLT}]
From the proof of Theorem \ref{thm:CB main01} we know that
the remainder terms are negligible. We only need to consider
\begin{align*}
\hat{U}_{n,r_{n},\omega}^{(\varepsilon)}(x_{0})= & \frac{1}{n}\sum_{j=1}^{n}\varepsilon_{j}K_{k}(x_{0},X_{j}).
\end{align*}
We know that
\[
\var\left(\varepsilon_{j}K_{k}(x_{0},X_{j})\right)=\sigma^{2}\Psi(x_{0}).
\]
We apply the Lindeberg Feller central limit theorem to $\hat{U}_{n,r_{n},\omega}^{(\varepsilon)}(x_{0})$.
The Lindeberg condition is fulfilled because
\begin{align*}
 & \lim_{n\to\infty}\frac{1}{n\sigma^{2}\Psi(x_{0})}\sum_{j=1}^{n}\mathbb{E}\left[\varepsilon_{j}^{2}K_{k}^{2}(x_{0},X_{j})\mathbb{I}\left\{ \varepsilon_{j}^{2}K_{k}^{2}(x_{0},X_{j})>\kappa^{2}n\sigma^{2}\Psi(x_{0})\right\} \right]\\
= & \lim_{n\to\infty}\frac{1}{\sigma^{2}\Psi(x_{0})}\mathbb{E}\left[\varepsilon_{1}^{2}K_{k}^{2}(x_{0},X_{1})\mathbb{I}\left\{ \sigma^{-2}\varepsilon_{1}^{2}K_{k}^{2}(x_{0},X_{1})\Psi^{-1}(x_{0})>\kappa^{2}n\right\} \right]\\
\to & 0
\end{align*}
because
\[
\frac{1}{\sigma^{2}\Psi(x_{0})}\mathbb{E}\left[\varepsilon_{1}^{2}K_{k}^{2}(x_{0},X_{1})\right]=1
\]
and $n\to\infty$. Hence
\[
\sqrt{\frac{n}{\sigma^{2}\Psi(x_{0})}}\hat{U}_{n,r_{n},\omega}^{(\varepsilon)}(x_{0})\to\mathcal{N}(0,1).
\]
We have several rest terms. As previously we have
\begin{align*}
 & \mathbb{P}\left(\vert U_{n,r_{n},\omega}^{(m)}(x_{0})-m(x_{0})\vert>\sqrt{\frac{2^{2k}\mathcal{V}_{\cap,k}}{n}}\right)\\
\leq & \mathbb{E}\left[\vert U_{n,r_{n},\omega}^{(m)}(x_{0})-m(x_{0})\vert\right]\sqrt{\frac{n}{2^{2k}\mathcal{V}_{\cap,k}}}\\
\leq & \sqrt{\frac{n}{2^{2k}\mathcal{V}_{\cap,k}}}\mathbb{E}\left[\mathfrak{d}(A_{k}(x_{0},\omega))^{\alpha}\right]\\
 & \qquad+\sqrt{\frac{n}{2^{2k}\mathcal{V}_{\cap,k}}}m(x_{0})(1-c_{X}2^{-k})^{r_{n}}\to0.
\end{align*}
Using that $\mathbb{E}[\varepsilon_{1}^{2}]<\infty$ we have
\begin{align*}
  \mathbb{P}\left(\bigg\vert R_{n,r_{n},\omega}^{(1)}(x_{0})\bigg\vert>\sqrt{\frac{2^{2k}\mathcal{V}_{\cap,k}}{n}}\right)
\leq  \mathbb{E}\left[\left(R_{n,r_{n},\omega}^{(1)}(x_{0})\right)^{2}\right]\frac{n}{2^{2k}\mathcal{V}_{\cap,k}}
\leq  \frac{C}{r_{n}\mathcal{V}_{\cap,k}}\to0.
\end{align*}
Further
\[
\mathbb{P}\left(\bigg\vert R_{n,r_{n},\omega}^{(2)}(x_{0})\bigg\vert>\sqrt{\frac{2^{2k}\mathcal{V}_{\cap,k}}{n}}\right)\to0
\]
because $r_{n}/n\leq c<1$.
\end{proof}

\begin{proof}[Proof of Theorem~\ref{thm:uni convergence}]

Using the triangle inequality we have
\begin{align}
\Vert U_{n,r_{n},\omega}^{(\mathrm{RF})}-m\Vert_{\infty} & \leq\Vert\hat{U}_{n,r_{n},\omega}^{(\varepsilon)}\Vert_{\infty}+\Vert\hat{U}_{n,r_{n},\omega}^{(\varepsilon)}-U_{n,r_{n},\omega}^{(\varepsilon)}\Vert_{\infty}+\Vert U_{n,r_{n},\omega}^{(m)}-m\Vert_{\infty}\nonumber \\
 & \leq\Vert\Psi_{k}^{1/2}\Vert_{\infty}\Vert\Psi_{k}^{-1/2}\hat{U}_{n,r_{n},\omega}^{(\varepsilon)}\Vert_{\infty}+\Vert\hat{U}_{n,r_{n},\omega}^{(\varepsilon)}-U_{n,r_{n},\omega}^{(\varepsilon)}\Vert_{\infty}+\Vert U_{n,r_{n},\omega}^{(m)}-m\Vert_{\infty}\nonumber \\
 & \leq\Vert\Psi_{k}^{1/2}\Vert_{\infty}\sqrt{\frac{\sigma^{2}}{n}}\mathbf{S}_{k}+\Vert\Psi_{k}^{1/2}\Vert_{\infty}\bigg\vert\Vert\Psi_{k}^{-1/2}\hat{U}_{n,r_{n},\omega}^{(\varepsilon)}\Vert_{\infty}-\sqrt{\frac{\sigma^{2}}{n}}\mathbf{S}_{k}\bigg\vert\nonumber \\
 & \qquad+\Vert R_{n,r_{n},\omega}^{(1)}\Vert_{\infty}+\Vert R_{n,r_{n},\omega}^{(2)}\Vert_{\infty}+\Vert U_{n,r_{n},\omega}^{(m)}-m\Vert_{\infty}.\label{eq:prf uni conv triangle inequ}
\end{align}
We handle the terms one by one. We note that (\ref{eq:uni bounds for Psi})
implies
\begin{equation}
\Vert\Psi_{k}^{1/2}\Vert_{\infty}\sim 2^{k}\mathcal{V}_{\cap,k}^{1/2}.\label{eq:prf uni conv psi ^1/2}
\end{equation}
With (\ref{eq:upper bound expect sup}) we get
\[
\mathbb{P}\left(\mathbf{S}_{k}>\kappa\right)\leq\frac{1}{\kappa}\mathbb{E}\left[\mathbf{S}_{k}\right]\lesssim\frac{1}{\kappa}\sqrt{k},
\]
and hence $\mathbf{S}_{k}=\mathcal{O}_{\mathbb{P}}(\sqrt{k})$. Thus,
the first term satisfies
\begin{equation}
\Vert\Psi_{k}^{1/2}\Vert_{\infty}\sqrt{\frac{\sigma^{2}}{n}}\mathbf{S}_{k}=\mathcal{O}\left(\sqrt{2^{2k}\mathcal{V}_{\cap,k}k/n}\right).\label{eq:prf uni conv S_k term}
\end{equation}
Theorem \ref{thm:supUhat-approx} and (\ref{eq:prf uni conv psi ^1/2})
yield
\begin{align}
\Vert\Psi_{k}^{1/2}\Vert_{\infty} & \Big\vert\Vert\Psi_{k}^{-1/2}\hat{U}_{n,r_{n},\omega}^{(\varepsilon)}\Vert_{\infty}-\sqrt{\frac{\sigma^{2}}{n}}\mathbf{S}_{k}\Big\vert\nonumber \\
 & =\sqrt{\sigma^{2}/n}\Vert\Psi_{k}^{1/2}\Vert_{\infty}\Big\vert\sqrt{n/\sigma^{2}}\Vert\Psi_{k}^{-1/2}\hat{U}_{n,r_{n},\omega}^{(\varepsilon)}\Vert_{\infty}-\mathbf{S}_{k}\Big\vert\nonumber \\
 &= \mathcal{O}(n^{-1/2}2^{k}\mathcal{V}_{\cap,k}^{1/2})\mathcal{O}_{\mathbb{P}}\left(\frac{(\log n)^{3/2}}{\mathcal{V}_{\cap,k}^{1/2}n^{1/2-1/\nu}}+\frac{(\log n)^{5/4}}{\mathcal{V}_{\cap,k}^{1/4}n^{1/4}}+\frac{\log n}{\mathcal{V}_{\cap,k}^{1/6}n^{1/6}}\right)\nonumber \\
 & =\mathcal{O}_{\mathbb{P}}\left(2^{k}\left(\frac{(\log n)^{3/2}}{n^{1-1/\nu}}+\frac{\mathcal{V}_{\cap,k}^{1/4}(\log n)^{5/4}}{n^{3/4}}+\frac{\mathcal{V}_{\cap,k}^{1/3}\log n}{n^{2/3}}\right)\right)\notag\\
 &=\mathcal O_{\mathbb P}\Big(\sqrt{\frac{2^k}{n}}\Big),\label{eq:prf uni conv gauss approx}
\end{align}
where we used assumption $\frac{2^k}{n}\le n^{-2/\nu}(\log n)^{-3}\le1$ together with $\mathcal V_{\cap,k}\le 2^{-k}$ in the last line.
Lemma \ref{lem:remainder 1} yields for $\nu$ with $\mathbb{E}[\vert\varepsilon_{1}\vert^{\nu}]<\infty$
that
\[
\mathbb{P}\left(\Vert R_{n,r_{n},\omega}^{(1)}\Vert_{\infty}>\kappa\right)\leq\sum_{x_{0}\in\mathcal{X}_{k}}\frac{1}{\kappa^{\nu}}\mathbb{E}\left[\left(R_{n,r_{n},\omega}^{(1)}(x_{0})\right)^{\nu}\right]\leq \mathcal N_{k}\frac{1}{\kappa^{\nu}}C\left(\frac{2^{2k}}{r_{n}n}\right)^{\nu/2}.
\]
For $\epsilon>0$ let $\kappa>(\epsilon/C)^{-1/\nu}$, we obtain
\[
\mathbb{P}\left(\Vert R_{n,r_{n},\omega}^{(1)}\Vert_{\infty}>\kappa\sqrt{\frac{2^{2k}}{r_{n}n}}\mathcal N_{k}^{1/\nu}\right)\leq C\frac{1}{\kappa^{\nu}}<\epsilon,
\]
and subsequently we have
\begin{equation}
\Vert R_{n,r_{n},\omega}^{(1)}\Vert_{\infty}=\mathcal{O}_{\mathbb{P}}\left(\sqrt{\frac{2^{2k}}{r_{n}n}}\mathcal N_{k}^{1/\nu}\right).\label{eq:prf uni conv rem 1}
\end{equation}
With Lemma \ref{lem:remainder 2}, $\mathcal N_{k}\leq2^{kp}$, $\mathcal{V}_{\cap,k}\geq2^{-2k}$
and the assumption $r_{n}/n\leq c<1$ we get
\begin{align*}
\mathbb{P} & \left(\Vert R_{n,r_{n},\omega}^{(2)}\Vert_{\infty}>\sqrt{2^{2k}\mathcal{V}_{\cap,k}k/n}\right)
\leq\frac{n}{2^{2k}\mathcal{V}_{\cap,k}k}\mathbb{E}\left[\Vert R_{n,r_{n},\omega}^{(2)}\Vert_{\infty}^{2}\right]\\
 & \leq\frac{n}{2^{2k}\mathcal{V}_{\cap,k}k}\sum_{x_{0}\in\mathcal{X}_{k}}\mathbb{E}\left[R_{n,r_{n},\omega}^{(2)}(x_{0})^{2}\right]\\
 & \leq\frac{n}{2^{2k}\mathcal{V}_{\cap,k}k}\mathcal N_{k}\frac{\sigma^{2}C_{X}}{c_{X}^{2}}\frac{2^{k}}{r_{n}}\left(\frac{r_{n}}{n}\right)^{r_{n}}
 \leq\frac{\sigma^{2}C_{X}}{c_{X}^{2}}2^{k(p+1)}\frac{1}{k}c^{r_{n}-1}\\
 & =\frac{\sigma^{2}C_{X}}{c_{X}^{2}}\exp\left(k(p+1)\log2-\log k+(r_{n}-1)\log c\right)\to0
\end{align*}
because we assumed that $r_{n}/k\to\infty$. Hence
\begin{equation}
\Vert R_{n,r_{n},\omega}^{(2)}\Vert_{\infty}=o_{\mathbb{P}}\left(\sqrt{2^{2k}\mathcal{V}_{\cap,k}k/n}\right)\label{eq:prf uni conv rem 2}
\end{equation}
which is negligible compared to (\ref{eq:prf uni conv S_k term}).
We note that
\[
\Vert U_{n,r_{n},\omega}^{(m)}-m\Vert_{\infty}\leq\Vert U_{n,r_{n},\omega}^{(m)}-mU_{n,r_{n},\omega}^{(1)}\Vert_{\infty}+\Vert m(U_{n,r_{n},\omega}^{(1)}-1)\Vert_{\infty}.
\]
Let $\epsilon>0$ be arbitrary and choose $\kappa>C_{H}/\epsilon$.
Lemma \ref{lem: approx err diam} yields
\begin{align*}
\mathbb{P} & \left(\Vert U_{n,r_{n},\omega}^{(m)}-mU_{n,r_{n},\omega}^{(1)}\Vert_{\infty}>\kappa\big(\mathbb{E}[\mathfrak{d}(A_{k}(x,\omega))^{\alpha}]+p^{\alpha/2}\mathcal N_{k}\binom{n}{r_{n}}^{-1/2}\big)\right)\\
 & \leq\frac{\mathbb{E}\left[\Vert U_{n,r_{n},\omega}^{(m)}-mU_{n,r_{n},\omega}^{(1)}\Vert_{\infty}\right]}{\kappa\big(\mathbb{E}[\mathfrak{d}(A_{k}(x,\omega))^{\alpha}]+p^{\alpha/2}\mathcal N_{k}\binom{n}{r_{n}}^{-1/2}\big)}
 \leq\frac{C_{H}}{\kappa}<\epsilon
\end{align*}
and in combination with $\binom{n}{r_{n}}^{-1/2}\le(r_n/n)^{r_n/2}$ we obtain
\begin{equation}
\Vert U_{n,r_{n},\omega}^{(m)}-mU_{n,r_{n},\omega}^{(1)}\Vert_{\infty}=\mathcal{O}_{\mathbb{P}}\left(\mathbb{E}[\mathfrak{d}(A_{k}(x,\omega))^{\alpha}]+\mathcal N_{k}(r_n/n)^{r_n/2}\right).\label{eq:prf uni conv approx 1}
\end{equation}
Lemma \ref{lem: approx err full cells} directly implies
\begin{equation}
\Vert m(U_{n,r_{n},\omega}^{(1)}-1)\Vert_{\infty}=\mathcal{O}_{\mathbb{P}}(2^{k}(1-c_{X}2^{-k})^{r_{n}}).\label{eq:prf uni conv approx 2}
\end{equation}
In conjunction (\ref{eq:prf uni conv triangle inequ}), (\ref{eq:prf uni conv S_k term}),
(\ref{eq:prf uni conv gauss approx}), (\ref{eq:prf uni conv rem 1}),
(\ref{eq:prf uni conv rem 2}), (\ref{eq:prf uni conv approx 1})
and (\ref{eq:prf uni conv approx 2}) yield
\begin{align*}
\Vert U_{n,r_{n},\omega}^{(\mathrm{RF})}-m\Vert_{\infty} & \leq\Vert\Psi_{k}^{1/2}\Vert_{\infty}\Big\vert\Vert\Psi_{k}^{-1/2}\hat{U}_{n,r_{n},\omega}^{(\varepsilon)}\Vert_{\infty}-\sqrt{\frac{\sigma^{2}}{n}}\mathbf{S}_{k}\Big\vert\\
 & \qquad+\sqrt{\frac{\sigma^{2}}{n}}\mathbf{S}_{k}\Vert\Psi_{k}^{1/2}\Vert_{\infty}+\Vert R_{n,r_{n},\omega}^{(1)}\Vert_{\infty}+\Vert R_{n,r_{n},\omega}^{(2)}\Vert_{\infty}+\Vert U_{n,r_{n},\omega}^{(m)}-m\Vert_{\infty}\\
 & =\mathcal{O}_{\mathbb{P}}\left(\sqrt{\frac{2^k}{n}}\Big(1+\mathcal N_k^{1/\nu}\sqrt{\frac{2^k}{r_n}}+\sqrt{k2^k\mathcal V_{\cap,k}}\Big)\right)+\mathcal{O}_{\mathbb{P}}\left(\mathbb{E}[\mathfrak{d}(A_{k}(x,\omega))^{\alpha}]\right)\\
 & \qquad+\mathcal{O}_{\mathbb{P}}\left(\mathcal N_{k}(r_n/n)^{r_n/2}+2^{k}(1-c_{X}2^{-k})^{r_{n}}\right).
\end{align*}
Since  and $\mathcal N_k\le 2^{kp}$ we have $\mathcal N_k\binom{n}{r_{n}}^{-1/2}\le e^{-c' r_n}$ for some $c'>0$ if $r_n\le cn$ and $r_n/k\to\infty$. Similarly, $k2^k/r_n\to0$ yields the upper bound $e^{-c'' 2^{-k}r_n}$ for some $c''>0$ for the last term. 
\end{proof}

\subsection{Proofs of the auxiliary results\label{subsec:Proof aux}}

In this section we collect the proofs of the auxiliary results that
have been used so far. We start with the proof of Theorem \ref{thm:supUhat-approx},
continue with the proofs for the approximation error results, and
conclude the section with the proofs for the remainder terms that
arise from the projection error.

\begin{proof}[Proof of Theorem \ref{thm:supUhat-approx}]\label{subsec:Proof sup approx appli}

We want to apply Corollary 2.2 by \citet{Chernozhukov2014}. We note
that the function class in the theorem is not dependent on $n$. However,
we can apply the theorem to $\mathcal{F}_{k}$ for every $n$ or $k$,
respectively. Also the constant in the theorem does not depend on
$n$. This is why we get a sequence of random variables in the claim
of the theorem. \citet{Chernozhukov2014} also point this out in their
Remark 2.1. We consider the function class $\mathcal{F}_{k}$ from
(\ref{eq:F_k}). To get the claimed result we need to consider $\mathcal{F}_{k}\cup-\mathcal{F}_{k}$.
Due to \citet[Corollary 2.1]{Chernozhukov2014} we consider $\mathcal{F}_{k}$
without loss of generality.

The function class  $\mathcal{F}_{k}$
is finite, and the functions $f_{x_{0},k}$
are measurable. Now, we prove that $\mathcal{F}_{k}$ is a VC type
class. Using the definition of $K_{k}$ in (\ref{eq:RF Hajek kernel})
together with the bound for $p_{x_{0}}(\omega)$ in (\ref{eq:bound px omega})
and $\Psi_{k}(x_{0})\geq c_{X}C_{X}^{-2}2^{2k}\mathcal{V}_{\cap,k}$
from (\ref{eq:uni bounds for Psi}) we obtain
\begin{align}
\sup_{x_{0}\in[0,1]^{p}}\vert f_{x_{0},k}(x,s)\vert & =\sup_{x_{0}\in[0,1]^{p}}\sigma^{-1}\vert s\vert\Psi_{k}^{-1/2}(x_{0})K_{k}(x_{0},x)\nonumber \\
 & \leq\sigma^{-1}\vert s\vert\frac{C_{X}}{c_{X}^{1/2}}2^{-k}\mathcal{V}_{\cap,k}^{-1/2}\sup_{x_{0}\in[0,1]^{p}}\mathbb{E}\left[\mathbb{I}\{x\in A_{k}(x_{0},\omega)\}p_{x_{0}}(\omega)^{-1}\right]\nonumber \\
 & \leq\sigma^{-1}\vert s\vert\frac{C_{X}}{c_{X}^{1/2}}2^{-k}\mathcal{V}_{\cap,k}^{-1/2}c_{X}^{-1}2^{k}\nonumber \\
 & =\sigma^{-1}\vert s\vert\frac{C_{X}}{c_{X}^{3/2}}\mathcal{V}_{\cap,k}^{-1/2}
 =:F(x,s).\label{eq:envelope cherno}
\end{align}
Therefore, $\mathcal{F}_{k}$ is equipped with the measurable envelope
$F$ that does not depend on $x$. The finite size $\vert\mathcal{F}_{k}\vert=\mathcal N_{k}\leq2^{kp}$
is an upper bound for any covering number of $\mathcal{F}_{k}$ and
thus,
\[
\sup_{Q\in\mathcal{Q}}N(\mathcal{F}_{k},\Vert\cdot\Vert_{Q,2},\kappa\Vert F\Vert_{Q,2})\leq2^{kp}\leq2^{kp}/\kappa
\]
for all $\kappa\in(0,1]$, implies that $\mathcal{F}_{k}$ is a VC
type class with $A$ and $v$ from \citet[Corollary 2.2]{Chernozhukov2014}
satisfying $A=2^{kp}$ and $v=1$.

We proceed by verifying the conditions on the moments. With the definitions
of $K_{k}$ in (\ref{eq:RF Hajek kernel}) and $\Psi_{k}$ in (\ref{eq:def psi})
we obtain
\begin{align*}
\sup_{f\in\mathcal{F}_{k}}P\vert f\vert^{2} & =\sup_{x_{0}\in[0,1]^{p}}\sigma^{-2}\Psi_{k}^{-1}(x_{0})\mathbb{E}\left[\varepsilon_{1}^{2}K_{k}^{2}(x_{0},X_{1})\right]\\
 & =\sup_{x_{0}\in[0,1]^{p}}\Psi_{k}^{-1}(x_{0})\mathbb{E}\left[K_{k}^{2}(x_{0},X_{1})\right]=1.
\end{align*}
Thus, $\tilde{\sigma}$ from \citet[Corollary 2.2]{Chernozhukov2014}
is equal to $1$. For $q\leq\nu$ we denote $\tau_{q}:=\mathbb{E}[\vert\varepsilon_{1}\vert^{q}]^{1/q}$.
The definitions used above, the lower bound for $p_{x_{0}}$ in (\ref{eq:bound px})
and $\Psi_{k}(x_{0})\geq c_{X}C_{X}^{-2}2^{2k}\mathcal{V}_{\cap,k}$
from (\ref{eq:uni bounds for Psi}) further yield
\begin{align*}
\sup_{f\in\mathcal{F}_{k}}P\vert f\vert^{3} & =\sup_{x_{0}\in[0,1]^{p}}\sigma^{-3}\Psi_{k}^{-3/2}(x_{0})\mathbb{E}\left[\vert\varepsilon_{1}K_{k}(x_{0},X_{1})\vert^{3}\right]\\
 & =\frac{\mathbb{E}\left[\vert\varepsilon_{1}\vert^{3}\right]}{\sigma^{3}}\sup_{x_{0}\in[0,1]^{p}}\Psi_{k}^{-3/2}(x_{0})\mathbb{E}\left[K_{k}^{2}(x_{0},X_{1})\mathbb{E}\left[\mathbb{I}\{X_{1}\in A_{k}(x_{0},\omega)\}p_{x_{0}}(\omega)^{-1}\mid X_{1}\right]\right]\\
& \leq
\frac{\tau_{3}^{3}}{\sigma^{3}}c_{X}^{-1}2^{k}\sup_{x_{0}\in[0,1]^{p}}\Psi_{k}^{-1/2}(x_{0}) \leq
 \frac{\tau_{3}^{3}}{\sigma^{3}}\frac{C_{X}}{c_{X}^{3/2}}\mathcal{V}_{\cap,k}^{-1/2}.
\end{align*}
For the envelope $F$ from (\ref{eq:envelope cherno}) we obtain
\[
\Vert F\Vert_{P,\nu}=\mathbb{E}\left[\vert F(X_{1},\varepsilon_{1})\vert^{\nu}\right]^{1/\nu}=\sigma^{-1}\frac{C_{X}}{c_{X}^{3/2}}\mathcal{V}_{\cap,k}^{-1/2}\mathbb{E}\left[\vert\varepsilon_{1}\vert^{\nu}\right]^{1/\nu}=\frac{C_{X}}{c_{X}^{3/2}}\frac{\tau_{\nu}}{\sigma}\mathcal{V}_{\cap,k}^{-1/2}.
\]
We choose $b$ from \citet[Corollary 2.2]{Chernozhukov2014} as
\[
b:=\mathcal{V}_{\cap,k}^{-1/2}C_{b}:=\mathcal{V}_{\cap,k}^{-1/2}\frac{C_{X}}{c_{X}^{3/2}}\max\left\{ \frac{\tau_{\nu}}{\sigma},\frac{\tau_{3}^{3}}{\sigma^{3}}\right\} .
\]
It holds that $b\geq1$ because $c_{X}\leq1\leq C_{X}$, $\mathcal{V}_{\cap,k}^{-1/2}\geq2^{k/2}\geq1$
and $\tau_{q}\geq\sigma$ due to Jensen's inequality. Further we have
$b=\mathcal{O}(\mathcal{V}_{\cap,k}^{-1/2})$ because $C_{b}$ is
a constant that only depends on the distributions of $\varepsilon_{1}$
and $X_{1}$. We get
\begin{align*}
\sup_{f\in\mathcal{F}_{k}\cup-\mathcal{F}_{k}}\mathbb{G}_{n}f & =\sup_{f\in\mathcal{F}_{k}}\vert\mathbb{G}_{n}f\vert =\sup_{x_{0}\in[0,1]^{p}}\vert\mathbb{G}_{n}\sigma^{-1}\Psi_{k}^{-1/2}(x_{0})sK_{k}(x_{0},x)\vert\\
 & =\sup_{x_{0}\in[0,1]^{p}}\vert\frac{1}{\sqrt{\sigma^{2}n}}\Psi_{k}^{-1/2}(x_{0})\sum_{j=1}^{n}\varepsilon_{j}K_{k}(x_{0},X_{j})\vert\\
 & =\sqrt{\frac{n}{\sigma^{2}}}\sup_{x_{0}\in[0,1]^{p}}\vert\Psi_{k}^{-1/2}(x_{0})\hat{U}_{n,r_{n},\omega}^{(\varepsilon)}(x_{0})\vert.
\end{align*}
In summary, the parameters from \citet[Corollary 2.2]{Chernozhukov2014}
are $\tilde{\sigma}=1$, $b=\mathcal{V}_{\cap,k}^{-1/2}C_{b}$, $A=2^{kp}$
and $v=1$. We use $\mathcal{V}_{\cap,k}\geq2^{-2k}$ from (\ref{eq:V cap simple bounds})
and the fact that $2^{k}=o(n)$ (see (\ref{eq:implication assum 2^k/r_n})
and (\ref{eq:assum remainder 2})) to obtain
\begin{align*}
K_{n} & =cv(\log n\vee\log(Ab/\tilde{\sigma})) =c(\log n\vee\log(2^{kp}\mathcal{V}_{\cap,k}^{-1/2}C_{b}))\\
 & \leq c(\log n\vee\log(2^{k(p+1)}C_{b}))\\
 & \leq c(p+1)(\log n\vee\log(2^{k}C_{b})) =\mathcal{O}(\log n).
\end{align*}
Let $B_{k}$ be the centered Gaussian process defined in the claim.
For $\gamma=(\log n)^{-1}$ \citet[Corollary 2.2]{Chernozhukov2014}
yields that there exists a random variable
\[
\mathbf{S}_{k}\overset{\mathcal{D}}{=}\sup_{f\in\mathcal{F}_{k}\cup-\mathcal{F}_{k}}B_{k}f=\sup_{f\in\mathcal{F}_{k}}\vert B_{k}f\vert=\sup_{x_{0}\in[0,1]^{p}}\vert B_{k}f_{x_{0},k}\vert
\]
such that
\begin{align*}
\mathbb{P} & \left(\Big\vert\sup_{f\in\mathcal{F}_{k}}\vert\mathbb{G}_{n}f\vert-\mathbf{S}_{k}\Big\vert>\frac{bK_{n}(\log n)^{1/2}}{n^{1/2-1/\nu}}+\frac{(b\log n)^{1/2}K_{n}^{3/4}}{n^{1/4}}+\frac{(bK_{n}^{2}\log n)^{1/3}}{n^{1/6}}\right)\\
 & \leq C\left((\log n)^{-1}+\frac{\log n}{n}\right)
\end{align*}
since $\tilde{\sigma}=1$. Hence
\begin{align*}
\bigg\vert\sqrt{\frac{n}{\sigma^{2}}} & \sup_{x_{0}\in[0,1]^{p}}\vert\Psi_{k}^{-1/2}(x_{0})\hat{U}_{n,r_{n},\omega}^{(\varepsilon)}(x_{0})\vert-\mathbf{S}_{k}\bigg\vert
=\Big\vert\sup_{f\in\mathcal{F}_{k}}\vert\mathbb{G}_{n}f\vert-\mathbf{S}_{k}\Big\vert\\
 & =\mathcal{O}_{\mathbb{P}}\left(\frac{bK_{n}(\log n)^{1/2}}{n^{1/2-1/\nu}}+\frac{(b\log n)^{1/2}K_{n}^{3/4}}{n^{1/4}}+\frac{(bK_{n}^{2}\log n)^{1/3}}{n^{1/6}}\right)\\
 & =\mathcal{O}_{\mathbb{P}}\left(\frac{(\log n)^{3/2}}{\mathcal{V}_{\cap,k}^{1/2}n^{1/2-1/\nu}}+\frac{(\log n)^{5/4}}{\mathcal{V}_{\cap,k}^{1/4}n^{1/4}}+\frac{\log n}{\mathcal{V}_{\cap,k}^{1/6}n^{1/6}}\right).
\end{align*}
since $K_{n}=\mathcal{O}((\log n))$ and $b=\mathcal{O}(\mathcal{V}_{\cap,k}^{-1/2})$.
\end{proof}

\begin{proof}[Proof of Lemma \ref{lem: approx err diam}]
 With (\ref{eq:U^(1)}) we have
\begin{align*}
U_{n,r_{n},\omega}^{(m)} & (x_{0})-m(x_{0})U_{n,r_{n},\omega}^{(1)}(x_{0})\\
 & =\frac{1}{\binom{n}{r_{n}}}\sum_{I\in \setbinom{n}{r_n}}\sum_{j\in I}(m(X_{j})-m(x_{0}))\frac{\mathbb{I}\{X_{j}\in A_{k}(x_{0},\omega_{I})\}}{\sum_{i\in I}\mathbb{I}\{X_{i}\in A_{k}(x_{0},\omega_{I})\}},
\end{align*}
which implies
\begin{align*}
\sup_{x_{0}\in[0,1]^{p}} & \vert U_{n,r_{n},\omega}^{(m)}(x_{0})-m(x_{0})U_{n,r_{n},\omega}^{(1)}(x_{0})\vert\\
 & \leq\sup_{x_{0}\in[0,1]^{p}}\frac{1}{\binom{n}{r_{n}}}\sum_{I\in \setbinom{n}{r_n}}\sum_{j\in I}\vert m(X_{j})-m(x_{0})\vert\frac{\mathbb{I}\{X_{j}\in A_{k}(x_{0},\omega_{I})\}}{\sum_{i\in I}\mathbb{I}\{X_{i}\in A_{k}(x_{0},\omega_{I})\}}\\
 & \leq C_{H}\sup_{x_{0}\in[0,1]^{p}}\frac{1}{\binom{n}{r_{n}}}\sum_{I\in \setbinom{n}{r_n}}\mathfrak{d}(A_{k}(x_{0},\omega_{I}))^{\alpha}\frac{\sum_{j\in I}\mathbb{I}\{X_{j}\in A_{k}(x_{0},\omega_{I})\}}{\sum_{i\in I}\mathbb{I}\{X_{i}\in A_{k}(x_{0},\omega_{I})\}}\\
 & \leq C_{H}\sup_{x_{0}\in[0,1]^{p}}\frac{1}{\binom{n}{r_{n}}}\sum_{I\in \setbinom{n}{r_n}}\mathfrak{d}(A_{k}(x_{0},\omega_{I}))^{\alpha}.
\end{align*}
For the expectation we get
\begin{align*}
\mathbb{E} & \left[\sup_{x_{0}\in[0,1]^{p}}\frac{1}{\binom{n}{r_{n}}}\sum_{I\in \setbinom{n}{r_n}}\mathfrak{d}(A_{k}(x_{0},\omega_{I}))^{\alpha}\right]\\
 & =\mathbb{E}\left[\sup_{x_{0}\in[0,1]^{p}}\frac{1}{\binom{n}{r_{n}}}\sum_{I\in \setbinom{n}{r_n}}\mathfrak{d}(A_{k}(x_{0},\omega_{I}))^{\alpha}-\mathbb{E}[\mathfrak{d}(A_{k}(x,\omega))^{\alpha}]\right]+\mathbb{E}[\mathfrak{d}(A_{k}(x,\omega))^{\alpha}].
\end{align*}
Using the independence of the $\omega_{I}$ we obtain for the latter
expectation that
\begin{align*}
\mathbb{E} & \left[\sup_{x_{0}\in[0,1]^{p}}\frac{1}{\binom{n}{r_{n}}}\sum_{I\in \setbinom{n}{r_n}}\mathfrak{d}(A_{k}(x_{0},\omega_{I}))^{\alpha}-\mathbb{E}[\mathfrak{d}(A_{k}(x,\omega))^{\alpha}]\right]\\
 & \leq\sum_{x_{0}\in\mathcal{X}_{k}}\mathbb{E}\left[\bigg\vert\frac{1}{\binom{n}{r_{n}}}\sum_{I\in \setbinom{n}{r_n}}\mathfrak{d}(A_{k}(x_{0},\omega_{I}))^{\alpha}-\mathbb{E}[\mathfrak{d}(A_{k}(x,\omega))^{\alpha}]\bigg\vert\right]\\
 & \leq \mathcal N_{k}\var\left(\frac{1}{\binom{n}{r_{n}}}\sum_{I\in \setbinom{n}{r_n}}\mathfrak{d}(A_{k}(x_{0},\omega_{I}))^{\alpha}\right)^{1/2}\\
 & \leq \mathcal N_{k}\frac{1}{\sqrt{\binom{n}{r_{n}}}}\mathbb{E}\left[\mathfrak{d}(A_{k}(x_{0},\omega))^{2\alpha}\right]^{1/2}
 \leq \mathcal N_{k}\frac{p^{\alpha/2}}{\sqrt{\binom{n}{r_{n}}}}.
\end{align*}
This leads to
\begin{align*}
\mathbb{E} & \left[\sup_{x_{0}\in[0,1]^{p}}\vert U_{n,r_{n},\omega}^{(m)}(x_{0})-m(x_{0})U_{n,r_{n},\omega}^{(1)}(x_{0})\vert\right]\\
 & \leq C_{H}\mathbb{E}\left[\sup_{x_{0}\in[0,1]^{p}}\frac{1}{\binom{n}{r_{n}}}\sum_{I\in \setbinom{n}{r_n}}\mathfrak{d}(A_{k}(x_{0},\omega_{I}))^{\alpha}\right]\\
 & \leq C_{H}\left(\mathbb{E}[\mathfrak{d}(A_{k}(x,\omega))^{\alpha}]+\mathcal N_{k}p^{\alpha/2}\binom{n}{r_{n}}^{-1/2}\right).\qedhere
\end{align*}
\end{proof}
\begin{proof}[Proof of Lemma \ref{lem: approx err full cells}]
 For any $\omega$, let $\mathcal{X}_{k}(\omega)$ denote a set of
$2^{k}$ points, with exactly one point in each of the partition cells
created by omega. That means $\vert\mathcal{X}_{k}(\omega)\cap A_{k}(x_{0},\omega)\vert=1$
for every $x_{0}\in[0,1]^{p}$. Using (\ref{eq:U^(1)}) for $U_{n,r_{n},\omega}^{(1)}(x_{0})$
we obtain
\begin{align*}
\mathbb{E} & \left[\sup_{x_{0}\in[0,1]^{p}}\vert m(x_{0})(U_{n,r_{n},\omega}^{(1)}(x_{0})-1)\vert\right]\\
 & \leq\Vert m\Vert_{\infty}\mathbb{E}\left[\sup_{x_{0}\in[0,1]^{p}}\bigg\vert\frac{1}{\binom{n}{r_{n}}}\sum_{I\in \setbinom{n}{r_n}}(\mathbb{I}\{\exists j\in I:X_{j}\in A_{k}(x_{0},\omega_{I})\}-1)\bigg\vert\right]\\
 & \leq\Vert m\Vert_{\infty}\mathbb{E}\left[\frac{1}{\binom{n}{r_{n}}}\sum_{I\in \setbinom{n}{r_n}}\sup_{x_{0}\in[0,1]^{p}}\mathbb{I}\{\nexists j\in I:X_{j}\in A_{k}(x_{0},\omega_{I})\}\right]\\
 & =\Vert m\Vert_{\infty}\mathbb{E}\big[\sup_{x_{0}\in[0,1]^{p}}\mathbb{I}\{\nexists j\in[r_{n}]:X_{j}\in A_{k}(x_{0},\omega)\}\big]\\
 & =\Vert m\Vert_{\infty}\mathbb{E}\Big[\mathbb{E}\big[\max_{x_{0}\in\mathcal{X}_{k}(\omega)}\mathbb{I}\{\nexists j\in[r_{n}]:X_{j}\in A_{k}(x_{0},\omega)\}\mid\omega\big]\Big]\\
 & \leq\Vert m\Vert_{\infty}\mathbb{E}\left[\mathbb{E}\left[\sum_{x_{0}\in\mathcal{X}_{k}(\omega)}\mathbb{I}\{\nexists j\in[r_{n}]:X_{j}\in A_{k}(x_{0},\omega)\}\mid\omega\right]\right]\\
 & =\Vert m\Vert_{\infty}\mathbb{E}\left[\sum_{x_{0}\in\mathcal{X}_{k}(\omega)}(1-p_{x_{0}}(\omega))^{r_{n}}\right]
  \leq\Vert m\Vert_{\infty}2^{k}(1-c_{X}2^{-k})^{r_{n}}.
\end{align*}
This completes the proof.
\end{proof}
\begin{remark}
\label{rem:empty cells sharp} We note that
\begin{align*}
\mathbb{E} & \left[\Big\vert\frac{1}{\binom{n}{r_{n}}}\sum_{I\in \setbinom{n}{r_n}}\mathbb{I}\{\exists j\in I:X_{j}\in A_{k}(x_{0},\omega_{I})\}-1\Big\vert\right]\\
 & =\mathbb{E}\left[\frac{1}{\binom{n}{r_{n}}}\sum_{I\in \setbinom{n}{r_n}}\mathbb{I}\{\nexists j\in I:X_{j}\in A_{k}(x_{0},\omega_{I})\}\right]\\
 & =\mathbb{E}\left[\mathbb{I}\{\nexists j\in[r_{n}]:X_{j}\in A_{k}(x_{0},\omega)\}\right]
  =\mathbb{E}\left[(1-p_{x_{0}}(\omega))^{r_{n}}\right]\\
 & \geq(1-C_{X}2^{-k})^{r_{n}}.
\end{align*}
Therefore, the bound in Lemma \ref{lem: approx err full cells} is
sharp up to the constant and the union bound. The bound
\[
2^{k}(1-c_{X}2^{-k})^{r_{n}}\lesssim\exp\left(k\log2-c_{X}r_{n}/2^{k}\right)=\exp\left(-r_{n}/2^{k}(c_{X}-k2^{k}r_{n}^{-1}\log2)\right)
\]
illustrates that the union bound is negligible as long as $k2^{k}r_{n}^{-1}\to0$.
\end{remark}

\begin{proof}[Proof of Lemma \ref{lem:remainder 1}]
 We omit $x_{0}$ in the notation of
\[
h_{R,n}((X_{i},\varepsilon_{i})_{i\in I},\omega_{I}):=\sum_{j\in I}\varepsilon_{j}\mathbb{I}\{X_{j}\in A_{k}(x_{0},\omega_{I})\}\left(\frac{1}{\sum_{i\in I}\mathbb{I}\{X_{i}\in A_{k}(x_{0},\omega_{I})\}}-\frac{1}{r_{n}p_{x_{0}}(\omega_{I})}\right)
\]
because it is fixed throughout the proof. Then we have
\[
R_{n,r_{n},\omega}^{(1)}(x_{0})=\frac{1}{\binom{n}{r_{n}}}\sum_{I\in \setbinom{n}{r_n}}h_{R,n}((X_{i},\varepsilon_{i})_{i\in I},\omega_{I}),
\]
which is a generalized U-statistic with centered kernel $h_{R,n}$.
For even $q>0$ Lemma \ref{lem:Moments U-stat} yields
\begin{equation}
\mathbb{E}\left[R_{n,r_{n},\omega}^{(1)}(x_{0})^{q}\right]\lesssim\left(\frac{r_{n}}{n}\right)^{q/2}\mathbb{E}\left[h_{R,n}((X_{i},\varepsilon_{i})_{i=1}^{r_{n}},\omega)^{q}\right].\label{eq:prf rem 1 u-stat bound}
\end{equation}
With equation (\ref{eq:cond MZ inequality}) and H\"older's inequality
we get
\begin{align}
\mathbb{E} & \left[h_{R,n}((X_{i},\varepsilon_{i})_{i=1}^{r_{n}},\omega)^{q}\right]\nonumber \\
 & =\mathbb{E}\left[\left(\sum_{j=1}^{r_{n}}\varepsilon_{j}\mathbb{I}\{X_{j}\in A_{k}(x_{0},\omega)\}\left(\frac{1}{\sum_{i=1}^{r_{n}}\mathbb{I}\{X_{i}\in A_{k}(x_{0},\omega)\}}-\frac{1}{r_{n}p_{x_{0}}(\omega)}\right)\right)^{q}\right]\nonumber \\
 & \lesssim\mathbb{E}\left[\left(\sum_{j=1}^{r_{n}}\varepsilon_{j}^{2}\mathbb{I}\{X_{j}\in A_{k}(x_{0},\omega)\}\left(\frac{1}{\sum_{i=1}^{r_{n}}\mathbb{I}\{X_{i}\in A_{k}(x_{0},\omega)\}}-\frac{1}{r_{n}p_{x_{0}}(\omega)}\right)^{2}\right)^{q/2}\right]\nonumber \\
 & =\sum_{q_{1}+...+q_{r_{n}}=q/2}\binom{q/2}{q_{1},\ldots,q_{r_{n}}}\nonumber \\
 & \qquad\times\mathbb{E}\left[\prod_{j=1}^{r_{n}}\varepsilon_{j}^{2q_{j}}\mathbb{I}\{X_{j}\in A_{k}(x_{0},\omega)\}\left(\frac{1}{\sum_{i=1}^{r_{n}}\mathbb{I}\{X_{i}\in A_{k}(x_{0},\omega)\}}-\frac{1}{r_{n}p_{x_{0}}(\omega)}\right)^{2q_{j}}\right]\nonumber \\
 & =\sum_{q_{1}+...+q_{r_{n}}=q/2}\binom{q/2}{q_{1},\ldots,q_{r_{n}}}\mathbb{E}\left[\prod_{j=1}^{r_{n}}\varepsilon_{j}^{2q_{j}}\right]\nonumber \\
 & \qquad\times\mathbb{E}\left[\prod_{j=1}^{r_{n}}\mathbb{I}\{X_{j}\in A_{k}(x_{0},\omega)\}\left(\frac{1}{\sum_{i=1}^{r_{n}}\mathbb{I}\{X_{i}\in A_{k}(x_{0},\omega)\}}-\frac{1}{r_{n}p_{x_{0}}(\omega)}\right)^{2q_{j}}\right]\nonumber \\
 & \leq\mathbb{E}\left[\varepsilon_{1}^{q}\right]\sum_{q_{1}+...+q_{r_{n}}=q/2}\binom{q/2}{q_{1},\ldots,q_{r_{n}}}\nonumber \\
 & \qquad\times\mathbb{E}\left[\prod_{j=1}^{r_{n}}\mathbb{I}\{X_{j}\in A_{k}(x_{0},\omega)\}\left(\frac{1}{\sum_{i=1}^{r_{n}}\mathbb{I}\{X_{i}\in A_{k}(x_{0},\omega)\}}-\frac{1}{r_{n}p_{x_{0}}(\omega)}\right)^{2q_{j}}\right],\label{eq:prf rem 1 ex h to q}
\end{align}
with $q_{j}\in\mathbb{N}_{0}$. For any $\mathbf{q}=(q_{j})_{j\in[r_{n}]}$
with $\sum_{j=1}^{r_{n}}q_{j}=q/2$ let $J_{\mathbf{q}>0}:=\{j\in[r_{n}]:q_{j}>0\}$,
$J_{\mathbf{q}=0}:=\{j\in[r_{n}]:q_{j}=0\}$ and $\tilde{q}(\mathbf{q}):=\vert J_{\mathbf{q}>0}\vert$.
For the expectation from (\ref{eq:prf rem 1 ex h to q}) we get with
$p_{x_{0}}(\omega)\leq C_{X}2^{-k}$ (see (\ref{eq:bound px omega}))
and $\tilde{q}(\mathbf{q})\leq q/2$ that
\begin{align}
\mathbb{E} & \left[\prod_{j=1}^{r_{n}}\mathbb{I}\{X_{j}\in A_{k}(x_{0},\omega)\}\left(\frac{1}{\sum_{i=1}^{r_{n}}\mathbb{I}\{X_{i}\in A_{k}(x_{0},\omega)\}}-\frac{1}{r_{n}p_{x_{0}}(\omega)}\right)^{2q_{j}}\right]\nonumber \\
 & =\mathbb{E}\left[\prod_{j\in J_{\mathbf{q}>0}}\mathbb{I}\{X_{j}\in A_{k}(x_{0},\omega)\}\left(\frac{1}{\tilde{q}(\mathbf{q})+\sum_{i\in J_{\mathbf{q}=0}}\mathbb{I}\{X_{i}\in A_{k}(x_{0},\omega)\}}-\frac{1}{r_{n}p_{x_{0}}(\omega)}\right)^{2q_{j}}\right]\nonumber \\
 & =\mathbb{E}\left[\left(\frac{1}{\tilde{q}(\mathbf{q})+\sum_{i\in J_{\mathbf{q}=0}}\mathbb{I}\{X_{i}\in A_{k}(x_{0},\omega)\}}-\frac{1}{r_{n}p_{x_{0}}(\omega)}\right)^{q}\prod_{j\in J_{\mathbf{q}>0}}\mathbb{I}\{X_{j}\in A_{k}(x_{0},\omega)\}\right]\nonumber \\
 & =\mathbb{E}\Bigg[\left(\frac{1}{\tilde{q}(\mathbf{q})+\sum_{i\in J_{\mathbf{q}=0}}\mathbb{I}\{X_{i}\in A_{k}(x_{0},\omega)\}}-\frac{1}{r_{n}p_{x_{0}}(\omega)}\right)^{q}\nonumber \\
 & \qquad\times\mathbb{E}\bigg[\prod_{j\in J_{\mathbf{q}>0}}\mathbb{I}\{X_{j}\in A_{k}(x_{0},\omega)\}\Big\vert\,\omega,(X_{i})_{i\in J_{\mathbf{q}=0}}\bigg]\Bigg]\nonumber \\
 & =\mathbb{E}\left[\left(\frac{1}{\tilde{q}(\mathbf{q})+\sum_{i\in J_{\mathbf{q}=0}}\mathbb{I}\{X_{i}\in A_{k}(x_{0},\omega)\}}-\frac{1}{r_{n}p_{x_{0}}(\omega)}\right)^{q}p_{x_{0}}(\omega)^{\tilde{q}(\mathbf{q})}\right]\nonumber \\
 & \leq C_{X}^{q/2}2^{-k\tilde{q}(\mathbf{q})}\mathbb{E}\left[\left(\frac{1}{\tilde{q}(\mathbf{q})+\sum_{i\in J_{\mathbf{q}=0}}\mathbb{I}\{X_{i}\in A_{k}(x_{0},\omega)\}}-\frac{1}{r_{n}p_{x_{0}}(\omega)}\right)^{q}\right].\label{eq:prf rem 1 exp prod bound}
\end{align}
Conditioned on $\omega$, the sum in the expression above is Binomial
distributed with parameters $r_{n}-\tilde{q}(\mathbf{q})$ and $p_{x_{0}}(\omega)$.
We note that $\tilde{q}(\mathbf{q})\leq q/2$ and hence we get with
Lemma \ref{lem:bin +h fractures mom} that
\begin{align}
\mathbb{E} & \left[\left(\frac{1}{\tilde{q}(\mathbf{q})+\sum_{i=\tilde{q}(\mathbf{q})+1}^{r_{n}}\mathbb{I}\{X_{i}\in A_{k}(x_{0},\omega)\}}-\frac{1}{r_{n}p_{x_{0}}(\omega)}\right)^{q}\right]\nonumber \\
 & =\mathbb{E}\left[\mathbb{E}\left[\left(\frac{1}{\tilde{q}(\mathbf{q})+\sum_{i=\tilde{q}(\mathbf{q})+1}^{r_{n}}\mathbb{I}\{X_{i}\in A_{k}(x_{0},\omega)\}}-\frac{1}{r_{n}p_{x_{0}}(\omega)}\right)^{q}\:\bigg\vert\:\omega\right]\right]\nonumber \\
 & \lesssim\mathbb{E}\left[(r_{n}p_{x_{0}}(\omega))^{-q3/2}\right]
 \leq(c_{X}r_{n}2^{-k})^{-q3/2}.\label{eq:prf rem 1 bound binomial denom}
\end{align}
Combining (\ref{eq:prf rem 1 ex h to q}), (\ref{eq:prf rem 1 exp prod bound})
and (\ref{eq:prf rem 1 bound binomial denom}), we obtain that there
exists a constant $C$, independent of parameters involved, such that
\begin{align}
\mathbb{E} & \left[h_{R,n}((X_{i},\varepsilon_{i})_{i=1}^{r_{n}},\omega)^{q}\right]\nonumber \\
 & \leq\frac{C\mathbb{E}\left[\varepsilon_{1}^{q}\right]C_{X}^{q/2}}{c_{X}^{q3/2}}\left(r_{n}2^{-k}\right)^{-q3/2}\sum_{q_{1}+...+q_{r_{n}}=q/2}\binom{q/2}{q_{1},\ldots,q_{r_{n}}}2^{-k\tilde{q}((q_{j})_{j\in[r_{n}]})}.\label{eq:prf rem 1 h^q temp bound}
\end{align}
We note that
\[
\sum_{q_{1}+...+q_{r_{n}}=q/2}\binom{q/2}{q_{1},\ldots,q_{r_{n}}}=r_{n}^{q/2}.
\]
For $\check{q}\in\{1,\ldots,q/2\}$, let
\[
N(\check{q})=\sum_{q_{1}+...+q_{r_{n}}=q/2}\binom{q/2}{q_{1},\ldots,q_{r_{n}}}\mathbb{I}\{\vert\{j\in[r_{n}],q_{j}>0\}\vert=\check{q}\}.
\]
This implies
\[
\sum_{\check{q}=1}^{q/2}N(\check{q})=r_{n}^{q/2}.
\]
Using that $0\leq\check{q}\leq q/2$, we obtain
\begin{align*}
N(\check{q}) & =\sum_{q_{1}+...+q_{r_{n}}=q/2}\binom{q/2}{q_{1},\ldots,q_{r_{n}}}\mathbb{I}\{\check{q}=\vert\{j\in[r_{n}],q_{j}>0\}\vert\}\\
 & =\sum_{q_{1}+...+q_{r_{n}}=q/2}\frac{(q/2)!}{q_{1}!\ldots q_{r_{n}}!}\mathbb{I}\{\check{q}=\vert\{j\in[r_{n}],q_{j}>0\}\vert\}\\
 & \leq\frac{(q/2)!}{(\lfloor\frac{q}{2\check{q}}\rfloor!)^{\check{q}}}\sum_{q_{1}+...+q_{r_{n}}=q/2}\mathbb{I}\{\check{q}=\vert\{j\in[r_{n}],q_{j}>0\}\vert\}\\
 & =\frac{(q/2)!}{(\lfloor\frac{q}{2\check{q}}\rfloor!)^{\check{q}}}\binom{r_{n}}{\check{q}}\check{q}^{q/2-\check{q}}
 \leq r_{n}^{\check{q}}(q/2)^{q}.
\end{align*}
Most importantly this does not depend on the $q_{j}$. We note that
$\tilde{q}((q_{j})_{j\in[r_{n}]})=\vert\{j\in[r_{n}],q_{j}>0\}\vert$.
Using $2^{k}\leq r_{n}$, we obtain for the sum in (\ref{eq:prf rem 1 h^q temp bound})
that
\begin{align}
\sum_{q_{1}+...+q_{r_{n}}=q/2} & \binom{q/2}{q_{1},\ldots,q_{r_{n}}}2^{-k\tilde{q}((q_{j})_{j\in[r_{n}]})}\nonumber \\
 & =\sum_{\check{q}=1}^{q/2}2^{-\check{q}k}\sum_{q_{1}+...+q_{r_{n}}=q/2}\binom{q/2}{q_{1},\ldots,q_{r_{n}}}\mathbb{I}\{\check{q}=\vert\{j\in[r_{n}],q_{j}>0\}\vert\}\nonumber \\
 & =\sum_{\check{q}=1}^{q/2}2^{-\check{q}k}N(\check{q})\nonumber \\
 & \leq(q/2)^{q+1}\left(\frac{r_{n}}{2^{k}}\right)^{q/2} =\mathcal{O}\left(\left(r_{n}2^{-k}\right)^{q/2}\right).\label{eq:prf rem 1 bound sum qj}
\end{align}
Using that $\mathbb{E}\left[\varepsilon_{1}^{q}\right]$, $C_{X}^{q}$
and $c_{X}^{-q}$ are constants independent of $n$ and $k$ together
with (\ref{eq:prf rem 1 bound sum qj}) and (\ref{eq:prf rem 1 h^q temp bound})
we obtain
\[
\mathbb{E}\left[h_{R,n}((X_{i},\varepsilon_{i})_{i=1}^{r_{n}},\omega)^{q}\right]=\mathcal{O}\left((r_{n}2^{-k})^{-q}\right).
\]
Thus, with (\ref{eq:prf rem 1 u-stat bound}) we end up with
\begin{align*}
\mathbb{E}\left[\left(R_{n,r_{n},\omega}^{(1)}(x_{0})\right)^{q}\right] & \lesssim\left(\frac{r_{n}}{n}\right)^{q/2}\mathbb{E}\left[\left(h_{R,n}((X_{i},\varepsilon_{i})_{i=1}^{r_{n}},\omega)\right)^{q}\right]\\
 & \lesssim\left(\frac{r_{n}}{n}\right)^{q/2}\left(\frac{2^{k}}{r_{n}}\right)^{q}=\left(\frac{2^{2k}}{nr_{n}}\right)^{q/2}.\qedhere
\end{align*}
\end{proof}
\begin{proof}[Proof of Lemma \ref{lem:remainder 2}]
We obtain
\begin{align*}
\mathbb{E} & \left[R_{n,r_{n},\omega}^{(2)}(x_{0})^{2}\right] =\var\left(R_{n,r_{n},\omega}^{(2)}(x_{0})\right)\\
 & =\frac{1}{n^{2}}\var\left(\sum_{j=1}^{n}\varepsilon_{j}\frac{1}{\binom{n-1}{r_{n}-1}}\sum_{I\in \setbinom{n}{r_n}:j\in I}\left(p_{x_{0}}(\omega_{I})^{-1}\mathbb{I}\{X_{j}\in A_{k}(x_{0},\omega_{I})\}-K_{k}(x_{0},X_{j})\right)\right)\\
 & =\frac{\sigma^{2}}{n}\var\left(\frac{1}{\binom{n-1}{r_{n}-1}}\sum_{I\in \setbinom{n}{r_n}:1\in I}\left(p_{x_{0}}(\omega_{I})^{-1}\mathbb{I}\{X_{1}\in A_{k}(x_{0},\omega_{I})\}-K_{k}(x_{0},X_{1})\right)\right)\\
 & =\frac{\sigma^{2}}{n}\frac{1}{\binom{n-1}{r_{n}-1}}\var\left(p_{x_{0}}(\omega)^{-1}\mathbb{I}\{X_{1}\in A_{k}(x_{0},\omega)\}-K_{k}(x_{0},X_{1})\right)
\end{align*}
because
\begin{align*}
\cov & \left(\frac{\mathbb{I}\{X_{1}\in A_{k}(x_{0},\omega_{1})\}}{p_{x_{0}}(\omega_{1})}-K_{k}(x_{0},X_{1}),\frac{\mathbb{I}\{X_{1}\in A_{k}(x_{0},\omega_{2})\}}{p_{x_{0}}(\omega_{2})}-K_{k}(x_{0},X_{1})\right)\\
 & =\mathbb{E}\left[\left(\frac{\mathbb{I}\{X_{1}\in A_{k}(x_{0},\omega_{1})\}}{p_{x_{0}}(\omega_{1})}-K_{k}(x_{0},X_{1})\right)\left(\frac{\mathbb{I}\{X_{1}\in A_{k}(x_{0},\omega_{2})\}}{p_{x_{0}}(\omega_{2})}-K_{k}(x_{0},X_{1})\right)\right]\\
 & =0.
\end{align*}
The latter equality follows by conditioning on $X_{1}$ and using
that
\[
K_{k}(x_{0},X_{1})=\mathbb{E}\left[\frac{\mathbb{I}\{X_{1}\in A_{k}(x_{0},\omega_{1})\}}{p_{x_{0}}(\omega_{1})}\:\Big\vert\:X_{1}\right]
\]
in conjunction with the independence of $\omega_{1}$ and $\omega_{2}$.
We get
\begin{align*}
\mathbb{E} & \left[R_{n,r_{n},\omega}^{(2)}(x_{0})^{2}\right]\\
 & =\frac{\sigma^{2}}{n}\frac{1}{\binom{n-1}{r_{n}-1}}\var\left(p_{x_{0}}(\omega)^{-1}\mathbb{I}\{X_{1}\in A_{k}(x_{0},\omega)\}-K_{k}(x_{0},X_{1})\right)\\
 & =\frac{\sigma^{2}}{\kappa^{2}n}\frac{1}{\binom{n-1}{r_{n}-1}}\mathbb{E}\left[\left(p_{x_{0}}(\omega)^{-1}\mathbb{I}\{X_{1}\in A_{k}(x_{0},\omega)\}-K_{k}(x_{0},X_{1})\right)^{2}\right]\\
 & =\frac{\sigma^{2}}{n}\frac{1}{\binom{n-1}{r_{n}-1}}\left(\mathbb{E}\left[p_{x_{0}}(\omega)^{-2}\mathbb{I}\{X_{1}\in A_{k}(x_{0},\omega)\}\right]-\mathbb{E}\left[K_{k}^{2}(x_{0},X_{1})\right]\right)\\
 & \leq\frac{\sigma^{2}}{n}\frac{1}{\binom{n-1}{r_{n}-1}}\mathbb{E}\left[p_{x_{0}}(\omega)^{-2}\mathbb{I}\{X_{1}\in A_{k}(x_{0},\omega)\}\right]
\leq \frac{\sigma^{2}C_{X}}{c_{X}^{2}}\frac{2^{k}}{r_{n}}\left(\frac{r_{n}}{n}\right)^{r_{n}}.
\end{align*}
This yields the claim.
\end{proof}

\section{Moment inequalities\label{sec:emp pro =000026 lit Results}}

The Marcinkiewicz-Zygmund inequality below will be employed repeatedly
in this article. In particular, we will frequently use an implication
of the result which we add as a second statement of the proposition.
The first statement and its proof can be found as Theorem 2 in \citet[Section 10.3]{Chow1997}.
\begin{prop}[Marcinkiewicz-Zygmund Inequality]
 \label{prop:MZ ineq} Let $q\in[1,\infty)$, for centered and independent
random variables $(X_{i})_{i=1}^{n}$ it holds that there exist constants
$C_{q,1}$ and $C_{q,2}$ only dependent on $q$ such that
\[
C_{q,1}\mathbb{E}\left[\left(\sum_{i=1}^{n}X_{i}^{2}\right)^{q/2}\right]\leq\mathbb{E}\left[\Big\vert\sum_{i=1}^{n}X_{i}\Big\vert^{q}\right]\leq C_{q,2}\mathbb{E}\left[\left(\sum_{i=1}^{n}X_{i}^{2}\right)^{q/2}\right].
\]
Let $(Y_{i})_{i=1}^{n}$ be random variables that are independent
of $(X_{i})_{i=1}^{n}$ but not necessarily independent of each other.
Assume that $\mathbb{E}[g(X_{i},Y_{i})\mid Y_{i}]=0$, then it holds
that
\begin{align}
\mathbb{E}\left[\Big\vert\sum_{i=1}^{n}g(X_{i},Y_{i})\Big\vert^{q}\right] & \leq C_{q,2}\mathbb{E}\left[\left(\sum_{i=1}^{n}g^{2}(X_{i},Y_{i})\right)^{q/2}\right].\label{eq:cond MZ inequality}
\end{align}
\end{prop}

\begin{proof}[Proof of Proposition \ref{prop:MZ ineq}]
We prove the second assertion. Conditioned on $(Y_{i})_{i=1}^{n}$
the random variables $g(X_{i},Y_{i})$ are independent and centered.
We use the first inequality and get
\begin{align*}
\mathbb{E}\left[\Big\vert\sum_{i=1}^{n}g(X_{i},Y_{i})\Big\vert^{q}\right] & =\mathbb{E}\left[\mathbb{E}\left[\Big\vert\sum_{i=1}^{n}g(X_{i},Y_{i})\Big\vert^{q}\:\bigg\vert\:(Y_{i})_{i=1}^{n}\right]\right]\\
 & \leq C_{q,2}\mathbb{E}\left[\mathbb{E}\left[\left(\sum_{i=1}^{n}g^{2}(X_{i},Y_{i})\right)^{q/2}\:\bigg\vert\:(Y_{i})_{i=1}^{n}\right]\right]\\
 & =C_{q,2}\mathbb{E}\left[\left(\sum_{i=1}^{n}g^{2}(X_{i},Y_{i})\right)^{q/2}\right].\qedhere
\end{align*}
\end{proof}
The Lemma below is useful to bound moments of generalized U-statistics.
The result and the proof are a generalization of an analogue result
for standard U-statistics, see \citet[Section 1.5, Theorem 1]{Lee1990}.
\begin{lemma}
\label{lem:Moments U-stat} Let $U_{n,r_{n},\omega}$ be a generalized
U-statistic with kernel $h_{n}(Z_{1},\ldots,Z_{r_{n}},\omega)$ satisfying
$\mathbb{E}[h_{n}(Z_{1},\ldots,Z_{r_{n}},\omega)]=0$. For $q\in[2,\infty)$
there exists a constant $C$ such that
\[
\mathbb{E}\left[\vert U_{n,r_{n},\omega}\vert^{q}\right]\leq C\left(\frac{r_{n}}{n}\right)^{q/2}\mathbb{E}\left[\vert h_{n}(Z_{1},\ldots Z_{r_{n}},\omega)\vert^{q}\right].
\]
\end{lemma}

\begin{proof}
We denote $[n]:=\{1,\ldots,n\}$ and
\[
\Pi(n):=\{\pi:[n]\to[n]:\pi([n])=[n]\}
\]
which is the set of all permutations of $\{1,\ldots,n\}$. Let $n_{b}=\lfloor n/r_{n}\rfloor$,
we further denote
\[
g(z_{1},\ldots,z_{n})=\frac{1}{n_{b}}\sum_{l=0}^{n_{b}-1}h_{n}(z_{lr_{n}+1},\ldots,z_{(l+1)r_{n}},\omega_{\{lr_{n}+1,\ldots,(l+1)r_{n}\}})
\]
By counting the appropriate permutations it can be seen that
\begin{align*}
n_{b} & \sum_{\pi\in\Pi(n)}g(Z_{\pi(1)},\ldots,Z_{\pi(n)})\\
 & =\sum_{\pi\in\Pi(n)}\sum_{l=0}^{n_{b}-1}h_{n}(Z_{\pi(lr_{n}+1)},\ldots,Z_{\pi((l+1)r_{n})},\omega_{\{\pi(lr_{n}+1),\ldots,\pi((l+1)r_{n})\}})\\
 & =n_{b}r_{n}!(n-r_{n})!\sum_{I\in \setbinom{n}{r_n}}h_{n}((Z_{i})_{i\in I},\omega_{I}).
\end{align*}
Hence we can write a generalized U-statistic as
\[
U_{n,r_{n},\omega}=\frac{1}{n!}\sum_{\pi\in\Pi(n)}g(Z_{\pi(1)},\ldots,Z_{\pi(n)}).
\]
For any convex function $f$ we can now use
\[
f(U_{n,r_{n},\omega})=f\left(\frac{1}{n!}\sum_{\pi\in\Pi(n)}g(Z_{\pi(1)},\ldots,Z_{\pi(n)})\right)\leq\frac{1}{n!}\sum_{\pi\in\Pi(n)}f\left(g(Z_{\pi(1)},\ldots,Z_{\pi(n)})\right)
\]
and thus
\begin{align*}
\mathbb{E}\big[f(U_{n,r_{n},\omega})\big] & \leq\mathbb{E}\big[f(g(Z_{1},\ldots,Z_{n}))\big].
\end{align*}
For this inequality we do not need the centered kernel yet. We choose
$f(\xi)=\vert\xi\vert^{q}$ and obtain
\begin{align*}
\mathbb{E}\left[\vert U_{n,r_{n},\omega}\vert^{q}\right] & \leq\mathbb{E}\left[\vert g(Z_{1},\ldots,Z_{n})\vert^{q}\right]\\
 & =\mathbb{E}\left[\left|\frac{1}{n_{b}}\sum_{l=0}^{n_{b}-1}h_{n}(Z_{lr_{n}+1},\ldots,Z_{(l+1)r_{n}},\omega_{\{lr_{n}+1,\ldots,(l+1)r_{n}\}})\right|^{q}\right]\\
 & =\frac{1}{n_{b}^{q}}\mathbb{E}\left[\left|\sum_{l=0}^{n_{b}-1}h_{n}(Z_{lr_{n}+1},\ldots,Z_{(l+1)r_{n}},\omega_{\{lr_{n}+1,\ldots,(l+1)r_{n}\}})\right|^{q}\right].
\end{align*}
The terms
\[
\left(h_{n}(Z_{lr_{n}+1},\ldots Z_{(l+1)r_{n}},\omega_{\{lr_{n}+1,\ldots,(l+1)r_{n}\}})\right)_{l=0}^{n_{b}-1}
\]
are independent because $h_{n}$ is applied to blocks of independent
$Z_{j}$ and the $\omega_{I}$ are independent as well. We can apply
Proposition \ref{prop:MZ ineq} and H\"older's inequality to obtain
\begin{align*}
\mathbb{E}\left[\vert U_{n,r_{n},\omega}\vert^{q}\right] & \leq\frac{1}{n_{b}^{q}}\mathbb{E}\left[\left|\sum_{l=0}^{n_{b}-1}h_{n}(Z_{lr_{n}+1},\ldots,Z_{(l+1)r_{n}},\omega_{\{lr_{n}+1,\ldots,(l+1)r_{n}\}})\right|^{q}\right]\\
 & \leq C_{q,2}\frac{1}{n_{b}^{q}}\mathbb{E}\left[\left(\sum_{l=0}^{n_{b}-1}h_{n}(Z_{lr_{n}+1},\ldots Z_{(l+1)r_{n}},\omega_{\{lr_{n}+1,\ldots,(l+1)r_{n}\}})^{2}\right)^{q/2}\right]\\
 & \leq C_{q,2}\frac{1}{n_{b}^{q/2}}\mathbb{E}\left[\vert h_{n}(Z_{1},\ldots Z_{r_{n}},\omega)\vert^{q}\right]\\
 & =C_{q,2}\frac{1}{\left(\lfloor n/r_{n}\rfloor\right)^{q/2}}\mathbb{E}\left[\vert h_{n}(Z_{1},\ldots Z_{r_{n}},\omega)\vert^{q}\right]\\
 & \leq C\left(\frac{r_{n}}{n}\right)^{q/2}\mathbb{E}\left[\vert h_{n}(Z_{1},\ldots Z_{r_{n}},\omega)\vert^{q}\right]
\end{align*}
for a suitable constant $C$.
\end{proof}

\subsection{Moments of Binomial denominators \label{sec:Binomial-results}}

Throughout the work we will need different results for the binomial
distribution, especially for its moments and for moments of terms
with a Binomial distributed denominator. These results are gathered
in this section. Throughout the section $B$ will denote a Binomial
random variable. Its parameters will vary and will be stated in the
results. The first lemma is similar to Lemma 4.1 by \citet{Gyoerfi2002}.
The first statement in the lemma is a slight variation that gives
equality instead of an upper bound.
\begin{lemma}
\label{lem:Gy=0000F6rfi 4.1} Let $B\sim\text{Bin}(n,p)$, it holds
that
\begin{align*}
\mathbb{E}\left[\frac{1}{1+B}\right] & =\frac{1}{(n+1)p}\left(1-(1-p)^{n+1}\right)\leq\frac{1}{(n+1)p}\\
\textit{and}\qquad\mathbb{E}\left[\frac{\mathbb{I}\{B>0\}}{B}\right] & \leq\frac{2}{(n+1)p}.
\end{align*}
\end{lemma}
\begin{proof}
Using that $p\leq1$, we obtain the first claim by
\begin{align*}
\mathbb{E}\left[\frac{1}{1+B}\right] & =\sum_{i=0}^{n}\frac{1}{i+1}\binom{n}{i}p^{i}(1-p)^{(n-i)}\\
 & =\sum_{i=1}^{n+1}\frac{1}{i}\binom{n}{i-1}p^{i-1}(1-p)^{(n+1-i)}\\
 & =\frac{1}{(n+1)p}\left(1-(1-p)^{n+1}\right) \leq\frac{1}{(n+1)p}.
\end{align*}
Using the first part we obtain
\begin{align*}
\mathbb{E}\left[\frac{\mathbb{I}\{B>0\}}{B}\right]\leq\mathbb{E}\left[\frac{2}{1+B}\right] &=\frac{2}{(n+1)p}\left(1-(1-p)^{n+1}\right) \leq\frac{2}{(n+1)p},
\end{align*}
which completes the proof.
\end{proof}
The next lemma gives us bounds for higher moments of this kind.
\begin{lemma}
\label{lem:1/(1+B)^m} Let $B\sim\text{Bin}(n-1,p)$ and let $q\in\mathbb{N}$.
The following inequalities hold
\begin{align*}
\mathbb{E}\left[(1+B)^{-2}\right] & \geq\frac{1}{p^{2}n(n+1)}\left(1-(n+1)p(1-p)^{n}-(1-p)^{n+1}\right)\\
\mathbb{E}\left[(1+B)^{-q}\right] & \leq\frac{q!}{p^{q}}\frac{(n-1)!}{(n+q-1)!}\leq\frac{q!}{p^{q}n^{q}}.
\end{align*}
\end{lemma}
\begin{proof}
We get the lower bound by
\begin{align*}
\mathbb{E}\left[\frac{1}{(1+B)^{2}}\right] & =\sum_{i=0}^{n-1}\frac{1}{(1+i)^{2}}\binom{n-1}{i}p^{i}(1-p)^{n-1-i}\\
 & \geq\sum_{i=0}^{n-1}\frac{1}{(1+i)(2+i)}\binom{n-1}{i}p^{i}(1-p)^{n-1-i}\\
 & =\frac{1}{p^{2}n(n+1)}\sum_{i=2}^{n+1}\binom{n+1}{i}p^{i}(1-p)^{n+1-i}\\
 & =\frac{1}{p^{2}n(n+1)}\left(1-(n+1)p(1-p)^{n}-(1-p)^{n+1}\right).
\end{align*}
For the upper bound we have
\begin{align*}
\mathbb{E}\left[\frac{1}{(1+B)^{q}}\right] & =\sum_{i=0}^{n-1}\frac{1}{(1+i)^{q}}\binom{n-1}{i}p^{i}(1-p)^{n-1-i}\\
 & =\sum_{i=0}^{n-1}\frac{1}{(1+i)^{q}}\frac{(n-1)!}{i!(n-1-i)!}p^{i}(1-p)^{n-1-i}.
\end{align*}
For $j\geq1$ we note that
\[
\frac{1}{1+i}=1-\frac{i}{i+1}\leq1-\frac{i}{i+j}=\frac{j}{j+i}.
\]
 This implies
\[
\frac{1}{(1+i)^{q}}=\prod_{j=1}^{q}\frac{1}{(1+i)}\leq\prod_{j=1}^{q}\frac{j}{j+i}=\prod_{j=1}^{q}j\prod_{j=1}^{q}\frac{1}{j+i}=q!\frac{i!}{(i+q)!}.
\]
We get
\begin{align*}
\mathbb{E} & \left[\frac{1}{(1+B)^{q}}\right]\\
 & \leq\sum_{i=0}^{n-1}\frac{q!i!}{(i+q)!}\frac{(n-1)!}{i!(n-1-i)!}p^{i}(1-p)^{n-1-i}\\
 & =\frac{q!(n-1)!}{p^{q}(n+q-1)!}\sum_{i=0}^{n-1}\frac{i!}{(i+q)!}\frac{(n+q-1)!}{(i+q)!(n+q-1-(i+q))!}p^{i+q}(1-p)^{n+q-1-(i+q)}\\
 & =\frac{q!(n-1)!}{p^{q}(n+q-1)!}\sum_{i=0}^{n-1}\binom{n+q-1}{i+q}p^{i+q}(1-p)^{n+q-1-(i+q)}\\
 & =\frac{q!(n-1)!}{p^{q}(n+q-1)!}\sum_{i=q}^{n+q-1}\binom{n+q-1}{i}p^{i}(1-p)^{n+q-1-i}\\
 & \leq\frac{q!(n-1)!}{p^{q}(n+q-1)!} =\frac{q!}{p^{q}}\prod_{j=0}^{q-1}\frac{1}{n+j} \leq\frac{q!}{p^{q}n^{q}}.
\end{align*}
The second inequality holds due to the binomial theorem and the last
inequality is implied by $n+j\geq n$ if $j\geq0.$
\end{proof}
The proposition below is used to prove the two results that follow
subsequently. It is given as Theorem 4 in the work by \citet{Skorski2025}.
\begin{prop}
\label{prop:skorski}Let $B\sim\text{Bin}(n,p)$ and $\sigma^{2}=p(1-p)$.
Then for any integer $q>1$ we have
\[
\mathbb{E}\left[(B-np)^{q}\right]^{1/q}=C(n,p,q)\max\{k^{1-k/q}(n\sigma^{2})^{k/q}:k=1,\ldots,\lfloor q/2\rfloor\},
\]
where $C(n,p,q)$ is uniformly bounded by $(3e)^{-1}\leq C(n,p,q)\leq(5/2)^{1/5}e^{1/2}$.
\end{prop}

\begin{remark}
\label{rem:Skorski}Let $np>1$, it holds that
\[
\max\{k^{1-k/q}(n\sigma^{2})^{k/q}:k=1,\ldots,\lfloor q/2\rfloor\}\leq(np)^{\lfloor q/2\rfloor/q}\max\{k^{1-k/q}:k=1,\ldots,\lfloor q/2\rfloor\}.
\]
Hence $\mathbb{E}[(B-np)^{q}]=\mathcal{O}((np)^{q/2})$ if $q$ is
fixed for all $n$.

The next two lemmas consider the moments of terms like
$\frac{1}{B}-\frac{1}{np}$.
We need upper bounds for the rate of these moments in $n$ and $p$
because in the proofs of the main results we will need to replace
binomial random variables in the denominator by their expectation.
\end{remark}

\begin{lemma}
\label{lem:moments binomial fractures}Let $B\sim\text{Bin}(n,p)$
such that $p^{-1}=\mathcal{O}(n)$. For any fixed integer $q>1$ it
holds that
\begin{align*}
\mathbb{E}\left[\mathbb{I}\{B>0\}\left(\frac{1}{B}-\frac{1}{np}\right)^{q}\right] & \lesssim(np)^{-3q/2}.
\end{align*}
\end{lemma}
\begin{proof}
Using simple calculations one gets
\begin{equation}
\frac{1}{B}-\frac{1}{np}=\frac{np-B}{(np)^{2}}+\frac{(np-B)^{2}}{(np)^{3}}+\frac{(np-B)^{3}}{B(np)^{3}}.\label{eq:binom denom decomp}
\end{equation}
Hence we get
\begin{align*}
\mathbb{E} & \left[\mathbb{I}\{B>0\}\left(\frac{1}{B}-\frac{1}{np}\right)^{q}\right]\\
 & \lesssim\mathbb{E}\left[\left(\frac{np-B}{(np)^{2}}\right)^{q}\right]+\mathbb{E}\left[\left(\frac{(np-B)^{2}}{(np)^{3}}\right)^{q}\right]+\mathbb{E}\left[\mathbb{I}\{B>0\}\left(\frac{(np-B)^{3}}{B(np)^{3}}\right)^{q}\right]\\
 & \leq(np)^{-2q}\mathbb{E}\left[\left(np-B\right)^{q}\right]+(np)^{-3q}\mathbb{E}\left[(np-B)^{2q}\right]+(np)^{-3q}\mathbb{E}\left[(np-B)^{3q}\right].
\end{align*}
We apply Proposition \ref{prop:skorski} in the case where $p^{-1}=\mathcal{O}(n)$.
We obtain
\begin{align*}
\mathbb{E} & \left[\mathbb{I}\{B>0\}\left(\frac{1}{B}-\frac{1}{np}\right)^{q}\right]\\
 & \lesssim(np)^{-2q}\mathbb{E}\left[(np-B)^{q}\right]+(np)^{-3q}\mathbb{E}\left[(np-B)^{2q}\right]+(np)^{-3q}\mathbb{E}\left[(np-B)^{3q}\right]\\
 & \lesssim(np)^{-3q/2}+(np)^{-2q}+(np)^{-3q/2}\lesssim(np)^{-3q/2},
\end{align*}
which yields the claim.
\end{proof}

The last lemma is similar to the above, but instead of the indicator
an added positive term ensures that the denominator is positive.
\begin{lemma}
\label{lem:bin +h fractures mom} For any fixed integer $q>1$ and
$h\in\mathbb{N}$ with $1\leq h\leq q$ let $B\sim\text{Bin}(n-h,p)$.
If $p^{-1}=\mathcal{O}(n)$ it holds that
\begin{align*}
\mathbb{E}\left[\left(\frac{1}{h+B}-\frac{1}{np}\right)^{q}\right] & \lesssim(np)^{-3q/2}.
\end{align*}
\end{lemma}
\begin{proof}
Using (\ref{eq:binom denom decomp}) from the previous proof and replacing
$B$ by $h+B$ we get
\begin{align*}
\frac{1}{h+B}-\frac{1}{np} & =\frac{(n-h)p-B-(1-p)h}{(np)^{2}}+\frac{\left((n-h)p-B-(1-p)h\right)^{2}}{(np)^{3}}\\
 & \qquad+\frac{\left((n-h)p-B-(1-p)h\right)^{3}}{(h+B)(np)^{3}}.
\end{align*}
Hence
\begin{align*}
\mathbb{E} & \left[\left(\frac{1}{h+B}-\frac{1}{np}\right)^{q}\right]\\
 & \lesssim\mathbb{E}\left[\left(\frac{(n-h)p-B-(1-p)h}{(np)^{2}}\right)^{q}\right]+\mathbb{E}\left[\left(\frac{\left((n-h)p-B-(1-p)h\right)^{2}}{(np)^{3}}\right)^{q}\right]\\
 & \qquad+\mathbb{E}\left[\left(\frac{\left((n-h)p-B-(1-p)h\right)^{3}}{(h+B)(np)^{3}}\right)^{q}\right]\\
 & \leq(np)^{-2q}\mathbb{E}\left[\left((n-h)p-B-(1-p)h\right)^{q}\right]+(np)^{-3q}\mathbb{E}\left[\left((n-h)p-B-(1-p)h\right)^{2q}\right]\\
 & \qquad+(np)^{-3q}\mathbb{E}\left[\left((n-h)p-B-(1-p)h\right)^{3q}\right]\\
 & \lesssim(np)^{-2q}\left(\mathbb{E}\left[\left((n-h)p-B\right)^{q}\right]+h^{q}\right)+(np)^{-3q}\left(\mathbb{E}\left[\left((n-h)p-B\right)^{2q}\right]+h^{2q}\right)\\
 & \qquad+(np)^{-3q}\left(\mathbb{E}\left[\left((n-h)p-B\right)^{3q}\right]+h^{3q}\right).
\end{align*}
With Theorem 4 by \citet{Skorski2025} we get
\begin{align*}
\mathbb{E} & \left[\left(\frac{1}{h+B}-\frac{1}{np}\right)^{q}\right]\\
 & \lesssim(np)^{-2q}\left(\mathcal{O}\left(((n-h)p)^{q/2}\right)+h^{q}\right)+(np)^{-3q}\left(\mathcal{O}\left(((n-h)p)^{q}\right)+h^{2q}\right)\\
 & \qquad+(np)^{-3q}\left(\mathcal{O}\left(((n-h)p)^{3q/2}\right)+h^{3q}\right)\\
 & \lesssim(np)^{-q3/2},
\end{align*}
which completes the proof.
\end{proof}

\bibliography{rf_lit}

\end{document}